\documentclass[10pt]{amsart}
\usepackage{amssymb}
\usepackage{amssymb,amsthm,amsmath,enumitem}
\usepackage[numbers,sort&compress]{natbib}
\usepackage{color}
\usepackage{graphicx}
\usepackage{tikz}
\IfFileExists{bbm.sty}{\usepackage{bbm}}{}
\usepackage{mathrsfs}
\usepackage{amsfonts}
\usepackage{bm}
\usepackage{lmodern}
\usepackage{appendix}
\usepackage{multirow}   

\usepackage[margin=2.5cm]{geometry}
\newtheorem{thm}{Theorem}[section]

\newtheorem{lem}[thm]{Lemma}
\newtheorem{prop}[thm]{Proposition}
\newtheorem{defn}[thm]{Definition}

\theoremstyle{definition}
\newtheorem{definition}[thm]{Definition}
\theoremstyle{plain}
\DeclareUnicodeCharacter{2010}{-}
\DeclareUnicodeCharacter{2011}{-}
\numberwithin{equation}{section}

\title[almost periodic solutions]{On the construction of almost periodic solutions for the derivative nonlinear Schr\"odinger equation }

\author[Y. Wu]{ Yuchen Wu}
\address[Yuchen Wu]
{School of Mathematical Sciences,
	Fudan University,
	Shanghai 200433, China} \email{22110180044@fudan.edu.cn  }

\author[X. Yuan]{Xiaoping Yuan$^*$}
\address[Xiaoping Yuan]
{School of Mathematical Sciences,
	Fudan University,
	Shanghai 200433, China} 
  \thanks{$^*$ Corresponding author}  
    \email{xpyuan@fudan.edu.cn}


\begin{document}

	\begin{abstract}
	In this paper, we consider a derivative nonlinear Schr\"odinger equation 
	$$
	\mathrm{i}\partial_{t}u+\partial_{xx}u-V\ast u+\mathrm{i}\vert u\vert^{2}\partial_{x}u=0
	$$
	on the torus $\mathbb{T}$, depending on some potential $V$. We prove that for `almost all' potentials $V$, this equation admits an almost-periodic solution.
	\end{abstract}
	\maketitle
	
	\section{Introduction}
	For nonlinear evolution equations viewed as infinite dimensional dynamical systems, it is both fundamental and significant to understand their dynamical properties, especially over long times. A natural primary goal is therefore to seek solutions with recurrent properties, of which time‑periodic solutions seem to be the simplest one can expect. The classical roots of this problem can be traced back to the study of the N‑body problem by Poincaré and Birkhoff, who established the Poincaré–Birkhoff twist theorem, providing deep geometric and topological insights into the existence of periodic orbits in conservative systems. 
	
	The study of the existence of periodic solutions for PDEs began relatively late. The investigation via variational methods dates back to the pioneering work of Rabinowitz in the late 1960s, where he employed constrained minimization to obtain periodic solutions for nonlinear wave equations \cite{Rab67}. Since then, the application of variational methods to study periodic solutions of PDEs has developed into a vast and highly active field (see e.g. \cite{BCN80,Str08}). However, as stated in \cite{Rab67}, this variational method applies only when the period is a rational multiple of $2\pi$. In the irrational case, zero becomes an accumulation point of the spectrum, and we cannot obtain a bounded inverse of the wave operator. Thus we are led to the classical small divisor problem. As is wl known, Kolmogorov-Arnold-Moser (KAM) theory provides a powerful tool for handling small divisors. The quadratic convergence inherent in Newton’s iteration is a mature and effective mechanism to overcome the loss of regularity frequently encountered in the context of PDEs, and it is commonly applied to establish wl‑posedness and to study the long‑time dynamical behavior of solutions, see e.g. \cite{Hor76,Zeh76,Kla80,MV11,FMM24}.
	
	A more general notion than periodic solutions is that of quasi‑periodic and almost periodic solutions. We recall the definitions below.
	\begin{defn}
		A continuous function $q:\mathbb{R}\rightarrow\mathbb{R}$ is called quasi-periodic with frequencies $\omega=(\omega_{1},\cdots,\omega_{n})$, if there exists a continuous hull $Q:\mathbb{T}^{n}\rightarrow\mathbb{R}$ such that $q(t)=Q(\omega t)$ for all $t\in\mathbb{R}$. A continuous function is called almost-periodic if it is a uniform limit of quasi-periodic functions.
	\end{defn}
	From a perspective more oriented toward dynamical systems, finding quasi-periodic solutions amounts to constructing mappings
	$$Q:\mathbb{T}^{n}\rightarrow P,\ x\mapsto Q(x,\cdot)$$
	of the $n$-torus $\mathbb{T}^{n}$ into the phase space $P$ together with frequency vectors $\omega$ such that the straight windings $t\mapsto \omega t+x_{0}$ on the torus map into solutions of the equation we consider. Thus quasi-periodic solutions arise from what we call invariant tori and each orbit on a finite dimensional invariant torus corresponds to a quasi‑periodic solution, while an infinite dimensional one corresponds to an almost periodic solution. 
	
	For decades, KAM theory has attracted the interest of many mathematicians and become a powerful machine to establish the existence of quasi-periodic solutions (invariant tori) for nonlinear Hamiltonian PDEs.
	\subsection{Context and background}\label{116}
	Since the properties of dispersive equations vary widely, we begin with a rough classification before introducing the application of KAM theory to PDEs. Consider a Hamiltonian PDE
	\begin{equation}\label{121}
	\partial_{t} w=Aw+F(w),\ x\in\mathbb{T}^{d},t\in\mathbb{R},
	\end{equation}
	where $Aw$ is a linear Hamiltonian vector field with $p:=ord\ A>0$ and $F(w)$ is a nonlinear Hamiltonian vector field with $\tilde{p}:=ord\ F$. The order of an operator $T$, denoted by $ord\ T$, is defined as the smallest real number  such that $T$ can be extended to a bounded operator from $H^{s}$ to $H^{s-ord\ T}$ for all real $s$ (or for $s$ in a certain range), where $H^{s}$ is the standard Sobolev space on the torus. For instance, $ord\ -\Delta=2$ and $ord\ N=0$ for the Laplace operator $-\Delta$ and quintic nonlinear operator $N(w):=\vert w\vert^{4}w$, respectively.
	\begin{description}
		\item[Case 1 (bounded perturbation, $\tilde{p}\leq 0,p>0$)] Here are two typical examples.\\
		NLS($\tilde{p}=0,p=2$):$\mathrm{i}\partial_{t}u+\Delta u+V(x)u+\vert u\vert^{2}u+\cdots=0,x\in\mathbb{T}^{d}$.\\
		NLW($\tilde{p}=-1,p=1$):$\partial_{tt}u-\Delta u+V(x)u+u^{3}+\cdots=0,x\in\mathbb{T}^{d}$.\\
		There is a wl-established and substantial body of the theory concerning such equations (see e.g. \cite{K, Kuk00, KP,P1,P2,W} for the earliest results for 1d PDEs and \cite{B1,B2,EK,EGK, BB13,BB20, Wang16, Wang19, Wang20,Yua21} for some results in higher dimension). 
		\item[Case 2 (unbounded perturbation, $p>0,0<\tilde{p}<p-1$)] Two of the most typical examples are the perturbed Korteweg-de Vries (pKdV) equation and the Kadomtsev–Petviashvili (KP) equation.\\
		pKdV($p=3,\tilde{p}=1$):$\partial_{t}u+\partial_{x}^{3}u+u\partial_{x}u+\cdots=0,x\in\mathbb{T}$.\\
		KP($p=3,\tilde{p}=1$):$\partial_{t}u+\partial_{x}^{3}u+u\partial_{x}u+\partial_{x}^{-1}\partial_{yy}u+\cdots=0,(x,y)\in\mathbb{T}^{2}$.\\
		For the PDE with unbounded Hamiltonian perturbation, the first KAM theorem is due to Kuksin \cite{K2} whose lemma is used to prove the persistence of the finite-gap solutions of the KdV equation, as wl as its hierarchy, subject to periodic boundary conditions. See also Kappeler-Pöschel \cite{KP2}.
		\item[Case 3 (unbounded critical perturbation, $p>0,\tilde{p}=p-1$)] A typical example is the derivative nonlinear Schrödinger equation.\\
		DNLS($p=2,\tilde{p}=1$):$\mathrm{i}\partial_{t}u+\partial_{xx}u+V\ast u+\mathrm{i}\partial_{x}(\vert u\vert^{2}u)+\cdots=0,x\in\mathbb{T}$.\\
		In the case of unbounded critical perturbations, Kuksin’s lemma \cite{K2} is no longer applicable (see subsection \ref{401} for details). Liu-Yuan \cite{LY10,LY11} generalized Kuksin's lemma and quasi-periodic solutions can be obtained for a class of derivative nonlinear Schrödinger equations with Dirichlet boundary conditions and perturbed Benjamin-Ono equation with periodic boundary conditions.
		\item[Case 4 (quasi-linear and fully nonlinear perturbation,$p>0,\tilde{p}\geq p$)]: A typical example is KdV with quasi-linear perturbation:\\
		$\partial_{t}u+\partial_{x}^{3}u-6u\partial_{x}u+F(x,u,\partial_{x}u,\partial_{xx}u,\partial_{x}^{3}u)=0,x\in\mathbb{T}$.\\
		And for this case, the representative work (see e.g. \cite{BBP13,BBM,BBHM18,BHM}) belongs to the Italian school.

	\end{description}

     By contrast, work on the space-multidimensional PDEs with unbounded perturbation is only beginning to emerge. Reference can be made to recent studies on the existence of periodic solutions for nonlinear Schr\"odinger and wave equations with fractional derivative perturbations \cite{XYY,YY25}, and quasi-periodic solutions of traveling-wave-like   for quasilinear Schr\"odinger equations \cite{YY26}.
     
     Most of these results can guarantee the existence of invariant tori of any finite dimension. However, in the proofs, the higher the dimension of the torus, the smaller the perturbation needs to be. Therefore, one could not obtain the existence of infinite dimensional tori from them by merely passing to the limit as the  dimension of the invariant tori goes to infinity.
     
     On the other hand, upon consideration of a finite dimensional nearly integrable Hamiltonian system, classical KAM theory guarantees that under suitable non-degeneracy hypotheses a majority of invariant tori persist and most solutions are quasi-periodic in time with frequency vectors obeying a series of Diophantine condition of sufficient strength. This naturally extends to infinite dimensional phase spaces, where one anticipates the existence of infinite dimensional especially full dimensional invariant tori. In other words, almost periodic solutions are heuristically expected to be typical for PDEs. Most studies, however, remain largely confined to quasi-periodic solutions or finite dimensional invariant tori; investigations into almost periodic solutions, especially in settings involving unbounded perturbations, are remarkably sparse (see subsection \ref{402} and table 1 below). A critical obstacle arises from the fact that the small divisors involved are predicted to become ``extremely small". It is a standard observation in the construction of finite dimensional invariant tori that the non-resonance conditions become progressively weaker as the dimension increases (see subsection \ref{401} for details). The origin of parameters in KAM theory leads to two scenarios: either one works with equations including a potential $V$ from which external parameters are extracted, or, in the absence of external parameters (i.e., for a fixed PDE), one relies on amplitude–frequency modulation to achieve the non‑resonance conditions. We now summarize the existence of infinite dimensional invariant tori for 1d PDEs of different types.

     \begin{table}[htbp]
     	\centering
     	\begin{tabular}{|c|c|c|c|c|}
     		\hline
     		\multirow{2}{*}{} & \multicolumn{2}{c|}{random potential $V$} & \multicolumn{2}{c|}{deterministic potential $V$} \\
     		\cline{2-5}
     		& $\tilde{p}\leq 0$ & $\tilde{p}> 0$ & $\tilde{p}\leq 0$ & $\tilde{p}> 0$ \\
     		\hline
     		infinite but non-full dimension & \cite{BMP} & ? & \cite{BGR} & ? \\
     		\cline{1-3}\cline{4-5}
     		full dimension & \cite{Bou96-GAFA,P3,B3}$\cdots$ &$?^{*}$ & \cite{BP26} & ? \\
     		\hline
     	\end{tabular}
     \caption{Summary table: Infinite dimensional invariant tori for 1d PDEs}
     \end{table}
    We highlight with the notation $?$ some issues which are not yet resolved (as far as we are aware) and $?^{*}$ the problem to be resolved in this paper. It should be pointed out that, for PDEs regarded as infinite dimensional dynamical systems, a full dimensional torus is naturally an infinite dimensional one. The converse, however, does not necessarily hold: an infinite dimensional torus need not be of full dimension. Such infinite but non-full dimensional tori can be viewed as “lower-dimensional” in an appropriate sense, which leaves the freedom to choose suitable tangent and normal directions. Following this idea, although stricter requirements are imposed on the tangential sites, the decay rate of the amplitude can be significantly improved, thereby leading to solutions of lower regularity (see \cite{BMP}). On the other hand, in the case of a deterministic potential (i.e., without external parameters), the situation is more delicate because one must balance the decay speed of the amplitude against the size of the frequency modulation. A breakthrough in this direction is the work \cite{BGR}, where the twist condition is maintained throughout the KAM iteration by a suitable selection of tangential sites. More recently, Bernier--Procesi \cite{BP26} obtained full-dimensional tori for parameterless NLS. In addition, the corresponding problems for PDEs in higher spatial dimensions remain completely open.

    In this paper, we focus on the derivative nonlinear Schr\"odinger (DNLS) equations on the one dimensional torus $\mathbb{T}:=\mathbb{R}/2\pi\mathbb{Z}$ with a convolution potential $V$:
    \begin{equation}\label{1}
    \mathrm{i}\partial_{t}u+\partial_{xx}u-V\ast u+\mathrm{i}\vert u\vert^{2}\partial_{x}u=0.
    \end{equation}
    For our purposes, it is convenient to identify $V$ with the set of external parameters $\sigma=(\sigma_{n})_{n\in\mathbb{Z}}$, where
    \begin{equation}\label{69}
    V=\sum_{n\in\mathbb{Z}}\sigma_{n}\phi_{n}(x),\ \phi_{n}(x)=\frac{1}{\sqrt{2\pi}}e^{\mathrm{i}nx},\ \sigma=(\sigma_{n})_{n\in\mathbb{Z}}\in \prod_{i\in\mathbb{Z}}[0,\frac{1}{\langle i\rangle}],\ \langle i\rangle:=\max\lbrace 1,|i|\rbrace.
    \end{equation}
    
    The DNLS equation is a fundamental canonical model in nonlinear wave dynamics, mainly describing the propagation of Alfvén waves in collisionless plasmas \cite{MOT76} and ultrashort optical pulses in single-mode nonlinear fibers \cite{Agr}, which extends the applicable scope of the standard nonlinear Schrödinger equation. 
    
    \subsection{Statement of the main result}\label{sec:main-result}
    
    We first introduce the parameter space and measure. For each interval $[0,\frac{1}{\langle i\rangle}],i\in\mathbb{Z}$, we define  a probability measure $\mathbb{P}_{i}$ on it by setting $\mathbb{P}_{i}(S):=\langle i\rangle mea(S)$ for every Lebesgue measurable subset $S\subset[0,\frac{1}{\langle i\rangle}]$, where $mea$ is the standard Lebesgue measure. From this, we can naturally define the measure on the Cartesian product $\prod_{i\in\mathbb{Z}}[0,\frac{1}{\langle i\rangle}]$ as the product measure of an infinite number of probability measures and denote it by $\mathbb{P}$ (see Appendix \ref{117} for details).
    
We first specify the notions of symplectic deformation and full dimension. Let $a>0$, $p>3/2$, and let $\ell^{a,p}_{\mathbb C}$ be the weighted sequence space introduced in Section 2, with
\[
 w_j=\langle j\rangle^p e^{a\langle j\rangle},\qquad
 \|q\|_{a,p}^2=\sum_{j\in\mathbb Z}w_j^2|q_j|^2.
\]
We write
\[
 E=\{(q,\bar q)\in\ell^{a,p}_{\mathbb C}\times\ell^{a,p}_{\mathbb C}:
 \bar q_j=\overline{q_j}\},\qquad
 \|Q\|_E^2=\|q\|_{a,p}^2+\|\bar q\|_{a,p}^2
\]
for the real phase space, equipped with the weak symplectic form
\[
 \varpi(U,V)=-\frac{\mathrm{i}}{2}\sum_{j\in\mathbb Z}
 (\bar u_jv_j-\bar v_ju_j).
\]

\begin{definition}[Symplectic deformation]\label{def:symplectic-deformation}
Let $X,Y\subset E$. A homeomorphism $f:X\to Y$ is called a \emph{symplectic deformation} if there is a field $L:X\to\mathcal L(E)$ such that each $L(Q)$ is bounded, invertible and symplectic,
\[
 \lim_{\substack{Q'\to Q\\Q'\in X,\ Q'\ne Q}}
 \frac{\|f(Q')-f(Q)-L(Q)(Q'-Q)\|_E}{\|Q'-Q\|_E}=0
 \qquad(Q\in X),
\]
and $L,L^{-1}$ are strongly continuous and locally uniformly bounded. Strong continuity means that, for every fixed $h\in E$,
\[
 L(Q')h\longrightarrow L(Q)h,\qquad
 L(Q')^{-1}h\longrightarrow L(Q)^{-1}h
 \quad\text{as }Q'\to Q\text{ in }X.
\]
We call that the map $L$ a derivative of $f$ and $Y$ is a symplectic deformation of $X$.
\end{definition}
This follows \cite[Definition 2.1]{BP26}, except that operator-norm continuity of $L$ is replaced by strong continuity of $L$ and $L^{-1}$, with local uniform bounds. We retain the full first-order remainder estimate. For the possibility of weaker differential formulations, see \cite[Remark 2.2]{BP26}.

\begin{definition}[Full-dimensional torus]\label{def:full-dimensional-torus}
A subset $\mathcal T\subset E$ is a \emph{full-dimensional torus} if it is a symplectic deformation of a straight torus
\[
 T_I=\{(q,\bar q)\in E:|q_j|^2=2I_j\text{ for every }j\in\mathbb Z\},
 \qquad I_j>0\ (j\in\mathbb Z),\quad
 \sum_{j\in\mathbb Z}w_j^2 I_j<\infty.
\]
\end{definition}
This is the formulation of \cite[Definition 2.3]{BP26} in our phase space, using the above notion of symplectic deformation. Our actions correspond to $\xi_j=2I_j$ in [BP26]; full support includes the zero mode. The straight torus is compact and homeomorphic to $\mathbb T^{\mathbb Z}$ with its product topology, and the same holds for its image. On $T_I$, the local uniform bounds for $L$ and $L^{-1}$ are therefore uniform.

The Lagrangian tangent-cone property of \cite[Proposition 4.14]{BP26} still holds with this continuity requirement. Recall, as in \cite[Definition 4.9]{BP26}, that
\[
 T_QX=\{\gamma'(0):\gamma\in C^1([0,1];E),\
 \gamma([0,1])\subset X,\ \gamma(0)=Q\}.
\]
Strong continuity and the first-order remainder suffice to transport $C^1$ curves under both $f$ and $f^{-1}$. The argument in Section~\ref{sec:limiting-deformation} gives
\[
 T_{f(Q)}\mathcal T=L(Q)T_QT_I,
 \qquad (T_{f(Q)}\mathcal T)^{\perp_\varpi}=T_{f(Q)}\mathcal T,
\]
where $V^{\perp_\varpi}=\{h\in E:\varpi(h,v)=0\text{ for every }v\in V\}$. Thus every such full-dimensional torus is Lagrangian. Section~\ref{sec:limiting-deformation} also proves that the limiting transformation constructed in this paper is a symplectic deformation in the above sense.

    Our main result is stated as follows:
    \begin{thm}\label{53}
    	Consider  the derivative nonlinear Schr\"odinger equation \eqref{1} with a convolution potential $V$. Then there exists a set $\Pi\subset\prod_{i\in\mathbb{Z}}[0,\frac{1}{\langle i\rangle}]$ with $\mathbb{P}(\Pi)>0$ such that for every $\sigma\in\Pi$, the equation has a full dimensional invariant torus.
    \end{thm}
    To the best of our knowledge, this paper provides the first result concerning the existence of full dimensional tori in the context of PDEs featuring truly \textit{unbounded} nonlinearities.

	\subsection{Historical remarks on almost periodic solutions}\label{402}

	Before analyzing the difficulties and strategies of the present paper, we first briefly recall some historical results on almost periodic solutions for partial differential equations. The first result concerning almost periodic solutions (full dimensional invariant tori) originated from Bourgain \cite{Bou96-GAFA} who proved the existence of such solutions for the wave equation with a random potential, where the radius of the tori exhibits super-exponential decay. Several years later, in \cite{B3}, the same author considered the nonlinear Schrödinger equation with convolution potential and obtained a full dimensional invariant torus whose radius satisfies a slower decay. After that this result has been continuously revisited and improved by a series of papers (see e.g. \cite{CLSY18,CY21}). In \cite{con23}, the decay rate of the amplitude for full dimensional tori is substantially refined through precise tame estimates and dimension truncation, yielding results that match those obtained by Pöschel \cite{Pos90} in the short-range setting. We also mention that recently following Bourgain's strategy, Biasco-Massetti-Procesi \cite{BMP} obtained the first result on the existence of weak almost periodic solutions for non integrable PDEs in KAM theory. However, while their work represents a breakthrough in the regularity of solutions (corresponding to the decay rate of the torus radius), it makes a concession regarding the support of the solutions. More specifically, they look for a particular (not full dimensional) torus supported on a sparse subset of $\mathbb{Z}$, thereby enabling stronger Diophantine conditions to be imposed and reducing the loss of regularity.
	We recall the following open problem raised by Kuksin \cite{Kuk04}:
	
	\textit{Can the full dimensional KAM tori be expected with a suitable decay for Hamiltonian partial differential equations, for example,}
	$$I_{n}\thicksim\vert n\vert^{-C}$$
	\textit{with some $C>0$ as $\vert n\vert\rightarrow +\infty$?}
	
	On the other hand, solutions with low regularity emerge naturally in PDE settings and carry more physical significance. Nevertheless, obtaining a solution which lies  on a full dimensional torus with only finite regularity—and thereby answering the question posed by Kuksin \cite{Kuk04}—appears to require genuinely new and revolutionary ideas at this stage.
	
	Unlike Bourgain \cite{B3}, who handled all Fourier modes simultaneously under a Diophantine condition which is tailored for infinite dimension, Pöschel \cite{P3} obtained almost periodic solutions via  successive small perturbations of finite dimensional invariant tori. By restricting the analysis to a finite dimensional torus at each step, Pöschel's approach simplifies the Diophantine condition. Inherently, though, this leads to a severe compactness requirement. Specifically, to make the iterative estimates independent of the dimension of the torus, the distance between successive finite dimensional tori must exhibit super-exponential decay. This yields almost periodic solutions with extremely high regularity which are very close to quasi-periodic ones. On the other hand, as the dimension of the torus increases, the persistence of the twist condition requires  the nonlinearity to be regularizing (see subsubsection \ref{305} below), namely
	$$\mathrm{i}\partial_{t}u-\partial_{xx}u+V(x)u+\Psi(f(\vert\Psi u\vert^{2})\Psi u)=0,$$
	where $\Psi$ is a convolution operator which is smoothing of order $\iota>\frac{1}{4}$. More precisely, $\Vert\Psi u\Vert_{H^{s+\iota}}\lesssim \Vert u\Vert_{H^{s}}$ for all $0\leq s\leq 1$ with $\iota>\frac{1}{4}$. By employing the Töplitz-Lipschitz property to describe the drift of frequency more precisely, Geng and Xu \cite{GX} were able to reduce the regularity demand for nonlinear terms and handle the general nonlinear Schrödinger equation with the external parameters.
	
	Even more challenging and physically relevant is the case of a fixed PDE, where parameters need to be extracted from the initial data. Quoting Bourgain \cite{B3}:`` In the absence  of exterior parameters, these [nonresonance] conditions need to be realized from amplitude-frequency modulation and suitable restriction of the action-variables. This problem is harder. Indeed, a fast decay of the action-variables (enhancing convergence of the process) allows less frequency modulation and worse small divisors." Recently following Pöschel's scheme, Bernier-Gr\'{e}bert-Robert \cite{BGR} proved the existence of infinite dimensional invariant tori for the NLS equation on the circle with no external parameters:
	\begin{equation}\label{306}
	\mathrm{i}\partial_{t}u+\partial_{xx}u=f(\vert u\vert^{2})u.
	\end{equation}
	To preserve the twist condition, they selected a sparse subset of $\mathbb{Z}$ and constructed special infinite dimensional tori supported on it by exploiting dispersive properties of \eqref{306}. More recently, Bernier--Procesi \cite{BP26} achieved a breakthrough by establishing refined decay and uniformity estimates for the amplitude-frequency map, which enabled them to remove the sparsity restriction and construct full dimensional invariant tori.
	
	The vector fields discussed in the above articles on almost periodic solutions are all bounded. However, in the unbounded cases we are dealing with, the small-denominator problem is further exacerbated, and it remains poorly understood at the present time. Corsi-Montalto-Procesi \cite{CMP} proved the existence of almost periodic response solutions for a forced quasi-linear Airy equation by employing a Craig-Wayne approach combined with a KAM reducibility scheme and pseudo-differential calculus on $\mathbb{T}^{\infty}$. And in \cite{GH}, invariant tori of full dimension for the following equation with the external parameters is obtained via the conservation law and Töplitz-Lipschitz property.
	$$\partial_{t}u=\partial_{x}^{5}u+\partial_{x}(M_{\sigma}u+u^{3}),$$
	where $M_{\sigma}$ is a real Fourier multiplier. This equation resembles the KdV equation formally, but with an additional artificially introduced linear term $M_{\sigma}u$. Precisely this linear term plays a crucial role in the proof.

\subsection{Obstacle and strategy}\label{401}
In this paper, we approach the full dimensional torus through iteration from the finite dimensional ones by following Pöschel's framework. As discussed above, particularly in the unbounded case we are dealing with, preservation of the twist condition is delicate due to the unboundedness of the drift of frequency. To address this, we introduce an asymptotic expansion of the frequency for a more refined description. However, the story is far from over: the ineradicable diagonal terms arising from the unbounded vector field lead to new difficulties. Although the measure estimate in Pöschel’s strategy appears rather involved, a great advantage is that one can benefit from the considerable freedom in selecting the size of domain at each KAM step. In this subsection, we outline our approach without delving into technical details.
\subsubsection{About the homological equations with variable coefficients}
As is wl known, unbounded vector fields prevent the diagonal terms of the perturbation from being eliminated and thus they have to be included in the normal form. Consequently, we are led to solve the following homological equations with variable coefficients:
\begin{equation}\label{101}
	-\mathrm{i}\omega\cdot\partial_{x}u+(\lambda+\mu(x))u=r(x).
\end{equation}
Due to the unbounded nature of the vector field, $\mu(x)$ is typically large and involves the angle variables. The equations of this type are called ``small-denominators equations with large variable coefficients" by Kuksin \cite{K2}.

To ensure the KAM iterative procedure works, the solution $u$ should be defined on a strip-type neighborhood of $\mathbb{T}^{n}$ with some width $s>0:D(s):=\lbrace x\in\mathbb{C}^{n}/2\pi\mathbb{Z}^{n}:\vert Imx\vert\leq s\rbrace$. Assume
$$\mu(x)\approx C\gamma,$$
where $C$ is some small constant and $\gamma$ should usually be a large magnitude. Kuksin's Lemma \cite{K2} states that under suitable non-resonant conditions on $\omega$ and the assumption
\begin{equation}
	\vert\lambda\vert^{\theta}\approx\Vert\mu\Vert_{s},0<\theta<1,
\end{equation}
the solution $u$ satisfies the estimate
$$\Vert u\Vert_{s-\sigma}\leq C_{1}exp(C_{2}C_{3}^{\frac{1}{1-\theta}})\Vert r\Vert_{s}.$$
For a Hamiltonian PDE of the form \eqref{121}, the corresponding parameters appearing in the homological equation \eqref{101} are
\begin{equation}\label{421}
   \lambda\approx i^{p}-j^{p}\approx i^{p-1}+j^{p-1},\gamma\approx i^{\tilde{p}}+j^{\tilde{p}},i,j\in\mathbb{Z}. 
\end{equation}
It therefore follows that this lemma is valid for unbounded perturbations with parameter $0<\theta<1$ (i.e., $\tilde{p}<p-1$), and hence can be applied to some equations including mainly the KdV, while it fails for critically unbounded perturbations with parameter $\theta=1$ (i.e., $\tilde{p}=p-1$) as encountered, for example, in the DNLS equation. And for the latter case, Liu-Yuan's lemma \cite{LY10} provides the following estimate
\begin{equation}\label{102}
	\Vert u\Vert_{s-\sigma}\leq C_{0}e^{C\gamma s}\Vert r\Vert_{s},
\end{equation}
where $\Vert r\Vert_{s}:=sup_{x\in D(s)}\vert r(x)\vert$ and $C_{0}$ is a constant depending on the non-resonant conditions of $\omega$. 

Since the parameter $\gamma$ goes into the exponential in the right hand side of \eqref{102}, in this sense, the upper bound of this new estimate looks weaker than that of the original Kuksin's lemma. Fortunately, when $\gamma$ is greater than a certain large truncation constant $K$, there will be no problem of small denominators when dealing with the following truncated homological equations.
\begin{equation}\label{103}
	-\mathrm{i}\omega\cdot\partial_{x}u+\lambda u+\Gamma_{K}(\mu u)=\Gamma_{K}r,
\end{equation}
where $\Gamma_{K}f(x):=\sum_{\vert k\vert\leq K}\hat{f}_{k}e^{\mathrm{i}kx}$. When $K\vert\omega\vert\lesssim \lambda$, the homological equation \eqref{103} can be solved by the implicit function theorem. So the non-trivial case is when $\gamma\lesssim K\vert\omega\vert$. Therefore, when the dimension of the torus $n$ is fixed (where $\vert\omega\vert$ can be regarded as a constant), at the $v^{th}$ KAM iteration, setting $s_{v}=2^{-v},\gamma_{v}\lesssim K_{v}\lesssim 2^{v}$ will suffice to obtain that
\begin{equation}\label{104}
	e^{C\gamma s}\lesssim\epsilon_{v}^{-C},
\end{equation}
where $\epsilon_{v}:=\epsilon^{(\frac{5}{4})^{v}}$ is the size of the perturbation $P_{v}$ and the constant $C$ is sufficiently small. Inequality \eqref{104} can guarantee the KAM procedure to be iterated.

However, in order to achieve a full dimensional invariant torus, according to Poschel's scheme, a new action-angle variable needs to be introduced at each KAM iteration, causing $\omega_{v}\approx v^{C}$ to grow polynomially. Under such circumstances, it follows that
$$\Vert u\Vert_{s-\sigma}\leq C_{0}e^{CKsv^{C}}\Vert r\Vert_{s}.$$
Absorbing the remainder term resulting from truncation leads to the following estimate for the new perturbation $P_{v+1}$:
$$\epsilon_{v+1}=\Vert X_{P_{v+1}}\Vert_{r_{v},a,p-1,D_{v}\times\Pi}\gtrsim(e^{Ksv^{C}}\epsilon_{v}+e^{-Ks})\epsilon_{v}\gtrsim\epsilon_{v}^{1+\frac{1}{v^{C}}}\to \epsilon_{v},\quad \text{as}\; v^C\to\infty,$$
where the definition of the norm is standard in KAM theory (see Section \ref{403}). Because the size $v^C$ of the frequency drifts $\omega_v$ goes to $\infty$ as $v\to\infty$, the quadratic convergence brought about by the Newton iteration no longer exists. This makes it impossible for us to overcome the difficulties caused by small denominators. As a result, Liu–Yuan’s Lemma \cite{LY10} needs to be improved essentially for this case. 
In fact, since $\mu(x)\approx \epsilon_{0}$ is relatively small compared to $\lambda$ when $i\neq\pm j$ in \eqref{421}, the variable-coefficient equation \eqref{101} can be regarded as a perturbation of the constant-coefficient equation and rewritten in the following form.
\begin{equation}\label{105}
	-\mathrm{i}\omega\cdot\partial_{x}u+\lambda(1+a(x))u=r(x),
\end{equation}
where $a(x)$ is relatively small. In \cite{HXY}, the authors derived an estimate for equation \eqref{105} that circumvents the exponential coefficient by absorbing the method from the Italian school's works (see e.g. \cite{BBP13,BBM,BBHM18}):
\begin{equation}\label{412}
	\Vert u\Vert_{\frac{s}{2}}\leq C_{1}\frac{v^{Cv}}{s^{C v}}\Vert r\Vert_{s},
\end{equation}
where $C_{1}$ is a constant that depends on non-resonant conditions of $\omega$, and $C$ represents an absolute constant. The underlying idea is to convert the equation \eqref{105} into a constant‑coefficient form which we are adept at dealing with by the following coordinate transformation
\begin{equation}
	T:x\mapsto \phi=x+b(x)\omega
\end{equation}
with $b(x):=\partial_{\omega}^{-1}a(x)$. In the case of finite dimensional torus considered in \cite{HXY}, the invertibility of $T$ follows readily from the boundedness of $\vert\omega\vert$, the smallness of $a$, and a standard Picard iteration argument. Unfortunately, in the infinite dimensional torus setting we are facing, the growth of $\omega_{v}$ once again becomes the obstacle: the invertibility of $T$ appears to break down, and the estimates in the proof for the bound of $u$ in \cite{HXY} need to be improved further for the case $v\to \infty$.  The detail can be found in section \ref{119}. It is worth pointing out that in the case $i=-j$, the homological equation no longer takes the form given in \eqref{105}; hence our improved estimates cease to be applicable. However, Liu-Yuan's lemma remains valid in this case. This follows from the fact that for sufficiently large $\vert j\vert$, $\Vert R^{v}_{(-j)j}\Vert_{s_{v}}$  decays exponentially and can therefore be absorbed into the new perturbation $P_{v+1}$ directly. Consequently,
$$\Vert X_{P_{v+1}}\Vert_{r_{v},a,p-1,D_{v}\times\Pi}\leq (e^{Ks}\epsilon_{v}+e^{-K})\epsilon_{v}+other \ terms\leq\epsilon_{v+1},$$
and the KAM iteration can still be carried out smoothly.
\subsubsection{About the modulation of frequencies and twist condition}\label{305}
A crucial hypothesis in any classical KAM theorem is the twist condition. Suppose the frequency vector $\omega=(\omega_{j})_{j\in J}$ (with $J\subset\mathbb{Z}$ finite or infinite) depends smoothly on parameters $\sigma\in\mathbb{R}^{J}$, one attempts to modulate the frequencies via variation of the parameters. This requires that the derivative $\frac{\partial\omega}{\partial\sigma}$ is invertible in some sense depending on the topology we choose on the set of frequencies and parameters. Extracting parameters from the convolution potential $V$, we write $\omega_{j}=\sigma_{j}+\epsilon\acute{\omega}_{j}$ where $\epsilon$ is a small parameter linked to the size of the initial perturbation and $\acute{\omega}_{j}\approx j^{\tilde{p}}$. When $l^{2}$ is chosen as source and target space, the twist condition reduces to the invertibility of $Id+\epsilon \frac{\partial\acute{\omega}}{\partial\sigma}$ and, if we intend to apply a Neumann series argument, can be rewritten as
\begin{equation}\label{111}
	\Vert \frac{\partial\acute{\omega}}{\partial\sigma}\Vert_{l^{2}\rightarrow l^{2}}\lesssim 1.
\end{equation}
For the nonlinear Schrödinger equation ($\tilde{p}=0$), the left-hand side of \eqref{111} clearly increases as $\sharp J$ grows, which forces us to reduce the size of $\epsilon$ to ensure the twist condition holds. Consequently, $\epsilon=0$ as the iterative step $v\to\infty$. Therefore, we can obtain nothing about the existence of almost periodic solution directly. This is precisely why it is necessary for the nonlinearity to have some smoothing condition ($\tilde{p}<0$) in \cite{P3} and Töplitz-Lipschitz condition was introduced in \cite{GX}. Unfortunately, due to the unbounded nature of the vector field, Töplitz-Lipschitz condition is clearly no longer valid. Nevertheless, through the observation of frequency drift and more precise estimate, we obtain that the frequency has the following asymptotic expansion
\begin{equation}\label{114}
	\acute{\omega}_{j}=j\bar{\omega}+\tilde{\omega}+\frac{\hat{\omega}_{j}}{j},
\end{equation}
where $\bar{\omega},\tilde{\omega}$ are independent of $j$ and $\vert \hat{\omega}\vert\lesssim 1$. Therefore, from the momentum conservation and mass conservation (see Definition \ref{112}), the twist condition \eqref{111} can be replaced by 
$$\Vert \frac{\partial\hat{\omega}}{\partial\sigma}\Vert_{l^{2}\rightarrow l^{2}}\lesssim 1$$
which naturally holds. In order to establish \eqref{114}, we need to ensure the corresponding asymptotic expansion property of Hamiltonian functions $\frac{\partial^{2}P}{\partial q_{m+t}\partial\bar{q}_{n+t}}=tP_{1}+P_{2}+\frac{P_{3}}{t}$ (see Definition \ref{70}) is preserved throughout the iteration. However, the appearance of the unbounded term $tP_{1}$ makes it no longer natural for $\lbrace P,F\rbrace$ to possess this property similarly, and a higher-order asymptotic expansion of $F$ (see Definition \ref{113}) is required and further technical intricacies are encountered. Indeed, in order to establish the asymptotic expansion of the form \eqref{114}, we naturally require that  $P_{1}$ and $P_{2}$ be independent of $t$,  leaving $\frac{P_{3}}{t}$ as the remainder term. If we instead establish for $F$ only an asymptotic expansion $F=F_{1}+\frac{F_{2}}{t}$ similar to that of the Töplitz-Lipschitz condition—where the coefficient of the remainder $F_{2}$ depends on 
$t$—then, the Poisson bracket $\lbrace P,F\rbrace$ naturally yields a $t$-dependent term like $tP_{1}\cdot \frac{F_{2}}{t}=P_{1}F_{2}\approx 1$ that cannot be absorbed into the remainder. Consequently, it becomes necessary to pursue a higher-order asymptotic expansion of $F$, and the following second-order estimate is adequate here. 
$$F=F_{1}+\frac{F_{2}}{t}+\frac{F_{3}}{t^{2}}.$$
Furthermore, to obtain the above expression, we need to derive the asymptotic expansion of the solution from that of the inhomogeneous term in the homological equation. This approach appears to be related to the method of power series solutions for differential equations and is, notably, employed here in KAM theory for the first time. In this sense, compared to Töplitz-Lipschitz property, expression \eqref{114} constitutes a genuine asymptotic expansion of the frequency and gives rise to estimates that are substantially more complex and delicate.
	
	\subsubsection{About normal frequencies depend on the angle variables}
	Assume we are now in the $v^{th}$ KAM iterative step. Write the integrable part $N_{v}$ of the Hamiltonian $H_{v}$,
	$$N_{v}=\sum_{j\in J}\omega_{j}y_{j}+\frac12\sum_{j\in\mathbb{Z}}\Omega_{j}z_{j}\bar{z}_{j}$$
	 and develop the perturbation $P_{v}$ into the Taylor series in $(y,z,\bar{z})$:
	 $$P_{v}=\epsilon_{v}R_{v}+O(\vert y\vert^{2}+\vert y\vert\Vert z\Vert_{a,p}+\Vert z\Vert_{a,p}^{3}),$$
	 where the norms in the formula above are to be defined in the following section and
	 $$R_{v}=R^{x}+\langle R^{y},y\rangle+\langle R^{z},z\rangle+\langle R^{\bar{z}},\bar{z}\rangle+\langle R^{zz}z,z\rangle+\langle R^{\bar{z}\bar{z}}\bar{z},\bar{z}\rangle+\langle R^{z\bar{z}}z,\bar{z}\rangle.$$
	 Analogous to the procedure in classical KAM theory, we seek a Hamiltonian function that shares the same form as $R_{v}$:
	 $$F_{v}= F^{x}+\langle F^{y},y\rangle+\langle F^{z},z\rangle+\langle F^{\bar{z}},\bar{z}\rangle+\langle F^{zz}z,z\rangle+\langle F^{\bar{z}\bar{z}}\bar{z},\bar{z}\rangle+\langle F^{z\bar{z}}z,\bar{z}\rangle $$
	 which satisfies
	 \begin{equation}\label{131}
	 	\lbrace N_{v},F_{v}\rbrace+R_{v}=0.
	 \end{equation}
	 With $F_{ij}(x)$ and $R_{ij}(x)$ denoting the matrix elements of the operators $F^{z\bar{z}}$ and $R^{z\bar{z}}$, respectively, we obtain from equation \eqref{131} the following equations:
	 \begin{equation}\label{132}
	 	(\langle k,\omega\rangle+\Omega_{i}-\Omega_{j})\hat{F}_{ij}(k)=-\mathrm{i}\hat{R}_{ij}(k),
	 \end{equation}
	 where $\hat{F}_{ij}(k)$ and $\hat{R}_{ij}(k)$ are the $k$-Fourier coefficients of the corresponding functions. In the KAM iteration for bounded perturbation, equation \eqref{132} can be solved directly by
	 $$\hat{F}_{ij}(k)=\frac{-\mathrm{i}\hat{R}_{ij}(k)}{\langle k,\omega\rangle+\Omega_{i}-\Omega_{j}}$$
	 under non-resonant conditions
	 $$\langle k,\omega\rangle+\Omega_{i}-\Omega_{j}\neq 0\ unless\ k=0,i=j.$$
	 In the case of unbounded perturbations, in order for the coordinate transformation $\Phi_{v}=X_{\epsilon_{v}F_{v}}^{t}\vert_{t=1}$ to remain bounded, it is necessary to require that
	 $$\vert \langle k,\omega\rangle+\Omega_{i}-\Omega_{j}\vert\approx \vert i\vert^{\tilde{p}}+\vert j\vert^{\tilde{p}}$$
	 which evidently fails when $i=j$. To circumvent this difficulty, Kuksin ingeniously incorporated the whole $2R_{jj}(x)$ into $\Omega_{j}$ as a modification of the normal form, thereby making the latter depends on the variable $x$. While this maneuver does not affect the construction of any finite-dimensional invariant tori, attempting to proceed further by adding action-angle variables following Pöschel's scheme becomes intractable. Nevertheless, we observe that although the unbounded nature of $\Omega$ prevents us from eliminating its all elements at once, it is still feasible to only eradicate  $\Omega_{v}$ in the $v$-th KAM step. At this point, the symplectic transform remains bounded and approaches the identity as the amplitude of $z_{v}$ undergoes rapid decay. To be more precise, at the $v^{th}$ step of the iteration, we first eliminate 
	 $$\frac12\Omega_{v}z_{v}\bar{z}_{v}=\sum_{i=1}^{v-1}R_{i,vv}(x)z_{v}\bar{z}_{v}$$
	  prior to introducing new action-angle variables by attempting to seek a Hamiltonian function 
	  $$F_{v}'=F'_{z\bar{z}}(x)z_{v}\bar{z}_{v}=\sum_{i=1}^{v-1}F_{i,vv}(x)z_{v}\bar{z}_{v}$$ which satisfies
	  $$\lbrace N_{v},F_{v}'\rbrace+\frac12\Omega_{v}z_{v}\bar{z}_{v}=0$$
	  or equivalently,
	  $$\hat{F}_{i,vv}(k)=\frac{-\mathrm{i}\hat{R}_{i,vv}(k)}{\langle k,\omega\vert_{J_{i}}\rangle}.$$
	  From the observation that the drifted term $R_{i,vv}(x)$ from the $i^{th}$ step is analytic in $\lbrace x\in\mathbb{C}^{i}/2\pi\mathbb{Z}^{i}:\vert Im x\vert\leq s_{i}\rbrace$, we conclude that 
	  \begin{equation}\label{304}
	  	|F_{z\bar{z}}'|\approx v.
	  \end{equation}
	  The Poincaré map generated by the vector field of the Hamiltonian function $F_{v}'$ is $z_{v}\approx e^{-2\mathrm{i}F'_{z\bar{z}}}z_{v}(0)$. From the small magnitude of the width $s_{v}$ together with the reality of $F'_{z\bar{z}}$ on the real axis, we know that $e^{-2\mathrm{i}F'_{z\bar{z}}}$ is very close to a unitary operator. Therefore, although the bound \eqref{304} increases progressively as the iteration proceeds, it can still be dominated owing to the rapid decay of perturbation $P_{v}$ and amplitude $z_{v}$.
	\subsection{Remarks and outlook}
	 The torus in Theorem~\ref{53} is a symplectic deformation of a full-support straight torus and hence is full dimensional. Section~\ref{sec:limiting-deformation} verifies the deformation properties and proves that its tangent cone is Lagrangian at every point. Rational independence of the frequencies makes its solutions almost periodic and not quasi-periodic.
	
	We expect our result to extend to the following more general family of the derivative nonlinear Schr\"odinger equations 
	$$\mathrm{i}\partial_{t}u+\partial_{xx}u-V\ast u+\mathrm{i}\partial_{x}(f(\vert u\vert^{2})u)=0$$
	and even a broader class of periodic dispersive PDEs whose nonlinear term contains first-order derivative, where $f$ is real analytic in some neighborhood of the origin in $\mathbb{C}$ and $f(0)=0$. This would only make the computations more cumbersome without really shedding more light on the relevant features.
	
	In addition, we can naturally arrange, as in classical KAM theory, that the measure of the excluded parameter set vanishes as the magnitude of the initial perturbation tends to zero. Consequently, we obtain the existence of almost periodic solutions for `almost all' potentials. Here, `almost all' potentials means that the complementary set of the Fourier coefficients of  potentials in $\prod_{i\in\mathbb{Z}}[0,\frac{1}{\langle i\rangle}]$ has $\mathbb{P}$-measure zero.
	
	The paper is organized as follows. In section 2, we introduce some definitions and notations as the preliminary. In section 3-5, an iterative lemma is proved in detail. In section 6, we execute the measure estimate. In section 7, Theorem \ref{53} is proved. The estimate for the solution of small denominator equation with large variable coefficient \eqref{105} is presented in section \ref{119}.
	
	In this paper, our aim is to prove the existence of the full-dimensional invariant torus without pursuing the optimality of various parameters. Therefore, in the following sections, the notation $C$ is used to replace some insignificant constants and may vary in different lines. When it occurs as an exponent, we denote it by $exp$.
	\section{Some functional analysis}\label{403}
	
	In this section, we write the derivative nonlinear  Schr\"odinger equation(DNLS) as an infinite dimensional Hamiltonian system. And at the same time, we introduce some definitions and notations as the preliminary.
	
	We study these equations as Hamiltonian systems on some suitable phase space $P$. We may take, for example, $P=H^{2}(\mathbb{T})$, the usual Sobolev space on $\mathbb{T}$. The same as in \cite{KP}, introducing the inner product
	$$\langle u,v\rangle=Re\int_{\mathbb{T}}u\bar{v} dx,$$
	then \eqref{1} can be written in Hamiltonian form
	\begin{equation}\label{2}
		\begin{aligned}
		&\frac{\partial u}{\partial t}=-\mathrm{i}\nabla H(u),\\
		&H(u)=\frac{1}{2}\int_{\mathbb{T}}\vert\partial_{x}u\vert^{2}dx+\frac{1}{2}\int_{\mathbb{T}}(V\ast u)\bar{u}dx-\frac{\mathrm{i}}{4}\int_{\mathbb{T}}\bar{u}^{2}u\partial_{x}udx,
		\end{aligned}
	\end{equation}
	where the $\nabla$ is defined with respect to $\langle \cdot,\cdot\rangle$.
	
	To write it in infinitely many coordinates, we make the ansatz 
	$$u=\sum_{j\in\mathbb{Z}}q_{j}\phi_{j}.$$
	The coordinates are taken from the Hilbert space $l^{a,p}_{\mathbb{C}}$ of all complex-valued sequences $q=(\cdots,q_{-1},q_{0},q_{1},\cdots)$ with $\Vert q\Vert_{a,p}^{2}=\sum_{j\in\mathbb{Z}}\vert q_{j}\vert^{2}\langle j\rangle^{2p}e^{2a\langle j\rangle}<\infty$. We fix $a > 0$ and $p>\frac{3}{2}$ in the following. Then \eqref{2} can be written as
	\begin{equation}\label{54}
		\dot{q_{j}}=-2\mathrm{i}\frac{\partial H}{\partial \bar{q}_{j}},j\in\mathbb{Z}
	\end{equation}
	with the Hamiltonian
	\begin{equation}\label{55}
		H(q)=\frac{1}{2}\sum_{j\in\mathbb{Z}}(j^{2}+\sigma_{j})\vert q_{j}\vert^{2}-\frac{\mathrm{i}}{4}\int_{\mathbb{T}}(\sum \bar{q}_{j}\phi_{-j})^{2}(\sum q_{j}\phi_{j})(\mathrm{i}\sum jq_{j}\phi_{j})dx.
	\end{equation}
	
	Let $N$ be an infinite dimensional Hamiltonian in the parameter dependent normal form
	$$N=\sum_{j\in J}\omega_{j}(\sigma)y_{j}+\frac12\sum_{\mathbb{Z}\setminus J}\Omega_{j}(\sigma)z_{j}\bar{z}_{j}$$
	on a phase space $P^{a,p}=\mathbb{T}^{n}\times\mathbb{R}^{n}\times l^{a,p}\times l^{a,p}\ni(x,y,z,\bar{z})$ with symplectic structure $\sum_{j\in J}dx_{j}\wedge dy_{j}-\frac{\mathrm{i}}{2}\sum_{j\in\mathbb{Z}\setminus J}d\bar{z}_{j}\wedge dz_{j}$, where $J=\lbrace i_{1},\cdots,i_{n}\rbrace\subset \mathbb{Z}$. The tangential frequencies $\omega=(\omega_{i_{1}},\cdots,\omega_{i_{n}})$ and normal frequencies $\Omega=(\Omega_{j})_{j\in\mathbb{Z}\setminus J}$ are real vectors depending on parameters $\sigma=(\sigma_{J},\sigma_{\mathbb{Z}\setminus J})\in \Pi=\Pi_{n}\times \Pi '\subset \prod_{j\in J}[0,\frac{1}{\langle j\rangle}]\times\prod_{j\in\mathbb{Z}\setminus J}[0,\frac{1}{\langle j\rangle}]$. The map $\omega:\sigma_{J}\mapsto(\omega_{i_{1}},\cdots,\omega_{i_{n}})$ is a one-to-one map between the set $\Pi_{n}$ and its image in $\mathbb{R}^{n}$ for any fixed $\sigma_{\mathbb{Z}\setminus J}\in\prod_{j\in\mathbb{Z}\setminus J}[0,\frac{1}{\langle j\rangle}]$. It should be noted that the definitions of the following norms and domains involve the set of tangential sites $J$, whose augmentation occurs only in a systematic manner throughout the KAM procedure (see Section 5). For notational simplicity and as long as no confusion is possible, the explicit indication of $J$ is omitted.

	Changing partial coordinates into action-angle variables, from the Hamiltonian partial differential equation \eqref{54}, \eqref{55} one can obtain a small perturbation $H=N+P$ of the normal form. The perturbation term $P$ is real analytic in the space coordinates and Lipschitz in the parameters, and for each $\sigma\in\Pi$ its Hamiltonian vector field $X_{P}=(P_{y},-P_{x},-2\mathrm{i}P_{\bar{z}},2\mathrm{i}P_{z})^{T}$ defines near $T_{0}:=\mathbb{T}^{n}\times\lbrace y=0\rbrace\times\lbrace z=0\rbrace\times\lbrace \bar{z}=0\rbrace$ a real analytic map
	$$X_{P}:P^{a,p}\rightarrow P^{a,p-1}.$$
	
	We denote by $P_{\mathbb{C}}^{a,p}$ the complexification of $P^{a,p}$. For $s,r>0$, we introduce the complex neighborhood in $P_{\mathbb{C}}^{a,p}$,
	$$D(s,r):\vert Imx\vert<s,\vert y\vert<r^{2}, \Vert z\Vert_{a,p}<r,\Vert \bar{z}\Vert_{a,p}<r,$$
	and weighed norm for $W=(X,Y,Z,\bar{Z})\in P_{\mathbb{C}}^{a,p-1}$,
	$$\Vert W\Vert_{r,a,p-1}=\vert X\vert+\frac{\vert Y\vert}{r^{2}}+\frac{\Vert Z \Vert_{a,p-1}}{r}+\frac{\Vert \bar{Z} \Vert_{a,p-1}}{r},$$
	where $\vert\cdot\vert$ denotes the sup-norm for complex vectors. Furthermore, for a map $W:D(s,r)\times\Pi\rightarrow P_{\mathbb{C}}^{a,p-1}$, for example, the Hamiltonian vector field $X_{P}$, we define the norms
	$$\Vert W\Vert_{r,a,p-1,D(s,r)\times\Pi}=\sup_{D(s,r)\times\Pi}\Vert W\Vert_{r,a,p-1},$$
	$$\Vert W\Vert_{r,a,p-1,D(s,r)\times\Pi}^{lip}=\sup_{\sigma,\sigma'\in\Pi,\sigma\neq\sigma'}\sup_{D(s,r)}\frac{\Vert\Delta_{\sigma\sigma'} W\Vert_{r,a,p-1}}{\Vert\sigma-\sigma'\Vert_{l^{2}}},$$
	where $\Delta_{\sigma\sigma'}W=W(\cdot;\sigma)-W(
	\cdot;\sigma')$. In a analogous manner, the Lipschitz semi-norm of the frequencies $\omega,\omega^{-1}$ and $\Omega$ are defined as
	$$\vert\omega\vert^{lip}_{\Pi}=\sup_{\sigma,\sigma'\in\Pi,\sigma\neq\sigma'}\frac{\Vert\Delta_{\sigma\sigma'}\omega\Vert_{l^{2}}}{\Vert \sigma-\sigma'\Vert_{l^{2}}},$$
	$$\vert\Omega\vert^{lip}_{-1,\Pi}=\sup_{\sigma,\sigma'\in\Pi,\sigma\neq\sigma'}\sup_{j\in\mathbb{Z}\setminus J}\frac{\vert \langle j\rangle^{-1}\Delta_{\sigma\sigma'}\Omega_{j}\vert}{\Vert \sigma-\sigma'\Vert_{l^{2}}},$$
	$$\vert\omega^{-1}\vert^{lip}_{\omega(\Pi_{n})\times \Pi'}=\sup_{\sigma_{J},\sigma'_{J}\in\Pi_{n},\sigma_{J}\neq\sigma'_{J},\sigma_{\mathbb{Z}\setminus J}\in\Pi'}\frac{\Vert\sigma_{J}-\sigma'_{J}\Vert_{l^{2}}}{\Vert\omega(\sigma_{J},\sigma_{\mathbb{Z}\setminus J})-\omega(\sigma'_{J},\sigma_{\mathbb{Z}\setminus J})\Vert_{l^{2}}}.$$
	
	In other words, the Lipschitz condition of $\omega^{-1}$ holds with respect to $\sigma_{J}\in\Pi_{n}$ uniformly with respect to $\sigma_{\mathbb{Z}\setminus J}\in\Pi'$.
	
	Similarly, for a function $f$ on $D(s,r)\times\Pi$, we introduce the supremum norm
	$$\Vert f\Vert_{D(s,r)\times\Pi}=\sup_{D(s,r)\times\Pi}\vert f\vert.$$
	And to characterize the rate of decay of the derivative with respect to parameters, we define
	
	$$\Vert f\Vert_{D(s,r)\times\Pi,L}=\sup_{D(s,r)\times\Pi}\vert f\vert+\sup_{l\in\mathbb{Z}}\sup_{D(s,r)\times\Pi}\vert \langle l\rangle\frac{\partial f}{\partial\sigma_{l}}\vert.$$
	
	Now we introduce some definitions and notations regarding the asymptotic estimate condition and conserved quantities which are used to describe the drift of frequencies and obtain the measure estimates. Although the asymptotic estimate condition is naturally motivated by power series solutions of differential equations and the Töplitz-Lipschitz property, it becomes considerably more intricate for technical reasons.

	\begin{defn}\label{112}
		A Hamiltonian function $F=\sum_{k,l,\alpha,\beta}F_{kl\alpha\beta}e^{\mathrm{i}kx}y^{l}z^{\alpha}\bar{z}^{\beta}$ on $D(s,r)$ is said to satisfy momentum conservation if
		$$\	F_{kl\alpha\beta}=0,\ if\sum_{j\in J}jk_{j}-\sum_{j\in\mathbb{Z}\setminus J}j(\alpha_{j}-\beta_{j})\neq 0,$$
		i.e., the $F$-Poisson commutes with the momentum $\sum_{j\in J}jy_{j}+\frac12\sum_{j\in\mathbb{Z}\setminus J}jz_{j}\bar{z}_{j}$.
		
		A Hamiltonian function $F=\sum_{k,l,\alpha,\beta}F_{kl\alpha\beta}e^{\mathrm{i}kx}y^{l}z^{\alpha}\bar{z}^{\beta}$ on $D(s,r)$ is said to satisfy mass conservation if
		$$\	F_{kl\alpha\beta}=0,\ if\sum_{j\in J}k_{j}-\sum_{j\in\mathbb{Z}\setminus J}(\alpha_{j}-\beta_{j})\neq 0,$$
		i.e., the $F$-Poisson commutes with the mass $\sum_{j\in J}y_{j}+\frac12\sum_{j\in\mathbb{Z}\setminus J}z_{j}\bar{z}_{j}$.
	For the corresponding small divisors we use $l=\beta-\alpha$, so the conservation conditions read $\sum_{j\in J}jk_j+\sum_{j\notin J}jl_j=0$ and $\sum_{j\in J}k_j+\sum_{j\notin J}l_j=0$.
	\end{defn}
	\begin{defn}\label{70}
		A Hamiltonian function $P$ on $D(s,r)$ is said to satisfy the first type of asymptotic estimate condition with parameters $\Lambda,\epsilon,\rho$ if there is the following estimate:
		$$\Vert\frac{\partial^{2}P}{\partial q_{m+t}\partial\bar{q}_{n+t}}\Vert_{D(s,r)\times\Pi,L}\leq \max\lbrace\langle m+t\rangle,\langle n+t\rangle\rbrace\epsilon e^{-\vert n-m\vert\rho},(m+t,n+t)\in\mathbb{Z}^{2},$$
		$$\Vert\frac{\partial^{2}P}{\partial q_{m+t}\partial q_{n-t}}\Vert_{D(s,r)\times\Pi,L}\leq \max\lbrace\langle m+t\rangle,\langle n-t\rangle\rbrace\epsilon e^{-\vert n+m\vert\rho},(m+t,n-t)\in\mathbb{Z}^{2},$$
		$$\Vert\frac{\partial^{2}P}{\partial \bar{q}_{m+t}\partial\bar{q}_{n-t}}\Vert_{D(s,r)\times\Pi,L}\leq \max\lbrace\langle m+t\rangle,\langle n-t\rangle\rbrace\epsilon e^{-\vert n+m\vert\rho},(m+t,n-t)\in\mathbb{Z}^{2}.$$
		Moreover, when $\vert t\vert \geq \Lambda \max\lbrace \vert m\vert,\vert n\vert\rbrace,t\neq 0 $, there exists some analytic functions $a_{i},b_{i},c_{i}(i=1,2,3)$ on $D(s,r)$ such that the following asymptotic expansions hold.
		$$\frac{\partial^{2}P}{\partial q_{m+t}\partial\bar{q}_{n+t}}=ta_{1}+b_{1}+c_{1}t^{-1}, $$
		$$\frac{\partial^{2}P}{\partial q_{m+t}\partial q_{n-t}}=ta_{2}+b_{2}+c_{2}t^{-1}, $$
		$$\frac{\partial^{2}P}{\partial \bar{q}_{m+t}\partial\bar{q}_{n-t}}=ta_{3}+b_{3}+c_{3}t^{-1},$$
		where 
		$$\Vert a_{1}\Vert_{D(s,r)\times\Pi},\Vert b_{1}\Vert_{D(s,r)\times\Pi},\Vert c_{1}\Vert_{D(s,r)\times\Pi}\leq\epsilon e^{-\vert n-m\vert\rho},$$
			$$\Vert a_{i}\Vert_{D(s,r)\times\Pi},\Vert b_{i}\Vert_{D(s,r)\times\Pi},\Vert c_{i}\Vert_{D(s,r)\times\Pi}\leq\epsilon e^{-\vert n+m\vert\rho},i=2,3,$$
			$$\sup_{D(s,r)\times\Pi}\vert\frac{\partial c_{1}}{\partial \sigma_{d}}\vert\leq\epsilon(e^{-\vert n-m\vert\rho}\frac{1}{\langle d\rangle}+\vert t\vert e^{-\vert m+t-d\vert\rho}e^{-\vert n+t-d\vert\rho}+\vert t\vert e^{-\vert -m-t-d\vert\rho}e^{-\vert -n-t-d\vert\rho}),\forall d\in\mathbb{Z},$$
			$$\sup_{D(s,r)\times\Pi}\vert\frac{\partial c_{i}}{\partial\sigma_{d}}\vert\leq\epsilon(e^{-\vert n+m\vert\rho}\frac{1}{\langle d\rangle}+\vert t\vert e^{-\vert m+t-d\vert\rho}e^{-\vert -n+t-d\vert\rho}+\vert t\vert e^{-\vert -m-t-d\vert\rho}e^{-\vert n-t-d\vert\rho}),\forall d\in\mathbb{Z},i=2,3.$$
		 At the same time, $a_{1},b_{1}$ can be expanded as Taylor series with respect to $\frac{1}{t}$:
		 $$a_{1}=a_{1}^{1}+a_{1}^{2}t^{-1}+a_{1}^{3}t^{-2}.$$
		 $$b_{1}=b_{1}^{1}+b_{1}^{2}t^{-1}.$$
		 Here $$\Vert a_{1}^{1}\Vert_{D(s,r)\times\Pi,L},\Vert b_{1}^{1}\Vert_{D(s,r)\times\Pi,L}\leq \epsilon e^{-\vert n-m\vert\rho},$$
		 $$\Vert a_{1}^{2}\Vert_{D(s,r)\times\Pi,L},\Vert b_{1}^{2}\Vert_{D(s,r)\times\Pi}\leq (\vert m\vert+\vert n\vert)\epsilon e^{-\vert n-m\vert\rho},$$
		 $$\Vert a_{1}^{3}\Vert_{D(s,r)\times\Pi}\leq (\vert m\vert+\vert n\vert)^{2}\epsilon e^{-\vert n-m\vert\rho},$$
		 $$\sup_{D(s,r)\times\Pi}\vert \frac{\partial a_{1}^{3}}{\partial\sigma_{d}}\vert\leq\epsilon[(\vert m\vert+\vert n\vert)^{2}e^{-\vert n-m\vert\rho}\frac{1}{\langle d\rangle}+\vert t\vert e^{-\vert m+t-d\vert\rho}e^{-\vert n+t-d\vert\rho}+\vert t\vert e^{-\vert -m-t-d\vert\rho}e^{-\vert -n-t-d\vert\rho} ],\forall d\in\mathbb{Z},$$
		 $$\sup_{D(s,r)\times\Pi}\vert \frac{\partial b_{1}^{2}}{\partial\sigma_{d}}\vert\leq\epsilon[(\vert m\vert+\vert n\vert) e^{-\vert n-m\vert\rho}\frac{1}{\langle d\rangle}+\vert t\vert e^{-\vert m+t-d\vert\rho}e^{-\vert n+t-d\vert\rho}+\vert t\vert e^{-\vert -m-t-d\vert\rho}e^{-\vert -n-t-d\vert\rho} ],\forall d\in\mathbb{Z}.$$ And $a_{1}^{1},a_{1}^{2},b_{1}^{1}$ are independent of $t$. Similarly, $f=a_{2},a_{3},g=b_{2},b_{3}$ also have this expansion:
		 $$f=f_{1}+f_{2}t^{-1}+f_{3}t^{-2},$$
		 $$g=g_{1}+g_{2}t^{-1}.$$
		 Here 
		 $$\Vert f_{1}\Vert_{D(s,r)\times\Pi,L},\Vert g_{1}\Vert_{D(s,r)\times\Pi,L}\leq \epsilon e^{-\vert n+m\vert\rho},$$
		 $$\Vert f_{2}\Vert_{D(s,r)\times\Pi,L},\Vert g_{2}\Vert_{D(s,r)\times\Pi}\leq (\vert m\vert+\vert n\vert)\epsilon e^{-\vert n+m\vert\rho},$$
		 $$\Vert f_{3}\Vert_{D(s,r)\times\Pi}\leq (\vert m\vert+\vert n\vert)^{2}\epsilon e^{-\vert n+m\vert\rho},$$
		 $$\sup_{D(s,r)\times\Pi}\vert\frac{\partial f_{3}}{\partial\sigma_{d}}\vert\leq\epsilon[(\vert m\vert+\vert n\vert)^{2}e^{-\vert n+m\vert\rho}\frac{1}{\langle d\rangle}+\vert t\vert e^{-\vert m+t-d\vert\rho}e^{-\vert -n+t-d\vert\rho}+\vert t\vert e^{-\vert -m-t-d\vert\rho}e^{-\vert n-t-d\vert\rho}],\forall d\in\mathbb{Z},$$
		 $$\sup_{D(s,r)\times\Pi}\vert\frac{\partial g_{2}}{\partial\sigma_{d}}\vert\leq\epsilon[(\vert m\vert+\vert n\vert)e^{-\vert n+m\vert\rho}\frac{1}{\langle d\rangle}+\vert t\vert e^{-\vert m+t-d\vert\rho}e^{-\vert -n+t-d\vert\rho}+\vert t\vert e^{-\vert -m-t-d\vert\rho}e^{-\vert n-t-d\vert\rho}],\forall d\in\mathbb{Z}.$$
		  And $f_{1},f_{2},g_{1}$ are independent of $t$. Throughout the paper, when $m=n=0$, we set $(\vert m\vert+\vert n\vert)=1$ by a slight abuse of notations.
		 
		  For convenience we briefly denote this condition as $P\in FAE_{\Lambda,\epsilon,\rho}$.
		  \end{defn}
		\begin{defn}\label{113}
			A Hamiltonian function $F$ on $D(s,r)$ is said to satisfy the second type of asymptotic estimate condition with parameters $\Lambda,\epsilon,\rho$ if there is the following estimate:
			$$\Vert\frac{\partial^{2}F}{\partial q_{m+t}\partial\bar{q}_{n+t}}\Vert_{D(s,r)\times\Pi,L}\leq \epsilon e^{-\vert n-m\vert\rho},(m+t,n+t)\in\mathbb{Z}^{2},$$
			$$\Vert\frac{\partial^{2}F}{\partial q_{m+t}\partial q_{n-t}}\Vert_{D(s,r)\times\Pi,L}\leq \epsilon e^{-\vert n+m\vert\rho},(m+t,n-t)\in\mathbb{Z}^{2},$$
			$$\Vert\frac{\partial^{2}F}{\partial \bar{q}_{m+t}\partial\bar{q}_{n-t}}\Vert_{D(s,r)\times\Pi,L}\leq \epsilon e^{-\vert n+m\vert\rho},(m+t,n-t)\in\mathbb{Z}^{2}.$$
			Moreover, when $\vert t\vert \geq \Lambda \max\lbrace \vert m\vert,\vert n\vert\rbrace,t\neq 0 $, there exists some analytic functions $a_{i},b_{i},c_{i}(i=1,2,3)$ on $D(s,r)$ such that the following asymptotic expansions hold.
			$$\frac{\partial^{2}F}{\partial q_{m+t}\partial\bar{q}_{n+t}}=a_{1}+b_{1}t^{-1}+c_{1}t^{-2}, $$
			$$\frac{\partial^{2}F}{\partial q_{m+t}\partial q_{n-t}}=a_{2}+b_{2}t^{-1}+c_{2}t^{-2}, $$
			$$\frac{\partial^{2}F}{\partial \bar{q}_{m+t}\partial\bar{q}_{n-t}}=a_{3}+b_{3}t^{-1}+c_{3}t^{-2},$$
			where 
			$a_{i},b_{i},c_{i}(i=1,2,3)$ satisfy the same condition as that in Definition \ref{70}.
			
			For convenience we briefly denote this condition as $F\in SAE_{\Lambda,\epsilon,\rho}$.
		\end{defn}

	Before concluding this section, we introduce some notations that will be frequently used in the following text.
	
	$\forall k\in\mathbb{Z}^{J},l\in\mathbb{Z}^{\mathbb{Z}\setminus J}$, let $\langle k\rangle=\max\lbrace 1,\vert k\vert\rbrace,\langle l\rangle_{2}=\max\lbrace 1,\vert\sum j^{2}l_{j}\vert\rbrace$. Moreover, for an analytic function $u$ on $D(s)=\mathbb{T}^{J}_{s}:=\lbrace x\in\mathbb{C}^{J}/2\pi\mathbb{Z}^{J}:\vert Imx\vert\leq s\rbrace$, we define $\Vert u\Vert_{s,\tau}=\sum_{k\in\mathbb{Z}^{J}}\vert\hat{u}(k)\vert k\vert^{\tau}e^{\vert k\vert s}$, where $\hat{u}(k)=(2\pi)^{-\sharp J}\int_{\mathbb{T}^{J}}u(x)e^{-\mathrm{i}k\cdot x}dx$ is the $k$-Fourier coefficient of $u$. $\bar{u}=[u]=\frac{1}{(2\pi)^{\sharp J}}\int_{\mathbb{T}^{J}}u(x)dx$ and $\tilde{u}=u-\bar{u}$. 
\section{The homological equations}\label{301}
\subsection{Derivation of homological equations}	
Let $N$ be the normal form
$$N=\sum_{j\in J}\omega_{j}(\sigma)y_{j}+\frac12\sum_{j\in\mathbb{Z}\setminus J}\Omega_{j}(x,\sigma)z_{j}\bar{z}_{j}$$
and $R$ be the 3-order Taylor polynomial truncation of $P$, that is,
$$
\begin{aligned}
    R=&R^{x}+\langle R^{y},y\rangle+\langle R^{z},z\rangle+\langle R^{\bar{z}},\bar{z}\rangle+\langle R^{zz}z,z\rangle+\langle R^{\bar{z}\bar{z}}\bar{z},\bar{z}\rangle+\langle R^{z\bar{z}}z,\bar{z}\rangle+\langle R^{yz}y,z\rangle+\langle R^{y\bar{z}}y,\bar{z}\rangle\\
    &+R^{zzz}[z,z,z]+R^{zz\bar{z}}[z,z,\bar{z}]+R^{z\bar{z}\bar{z}}[z,\bar{z},\bar{z}]+R^{\bar{z}\bar{z}\bar{z}}[\bar{z},\bar{z},\bar{z}],
\end{aligned}
$$
 where $R^{x},R^{y},R^{z},R^{\bar{z}},R^{zz},R^{\bar{z}\bar{z}},R^{z\bar{z}},R^{yz},R^{y\bar{z}},R^{zzz},R^{zz\bar{z}},R^{z\bar{z}\bar{z}},R^{\bar{z}\bar{z}\bar{z}}$ depend on $x$ and $\sigma$. $\Omega_{j}(j\in\mathbb{Z}\setminus J)$ have the decomposition $\Omega_{j}=\sum_{i=0}^{n}\Omega_{j}^{i}$ and $\Omega_{j}^{i}$ depends only on $x\in \mathbb{T}^{J_{i}}_{s}$ and $\sigma\in\Pi$, where $J_{0}\subset J_{1}\subset\cdots\subset J_{n}=J$. The coordinate transformation $\Phi$ is obtained as the time-1-map $X^{t}_{F}\vert_{t=1}$ of a Hamiltonian vector field $X_{F}$, where $F$ is of the same form as $R$:
 $$ 
 \begin{aligned}
     F=&F^{x}+\langle F^{y},y\rangle+\langle F^{z},z\rangle+\langle F^{\bar{z}},\bar{z}\rangle+\langle F^{zz}z,z\rangle+\langle F^{\bar{z}\bar{z}}\bar{z},\bar{z}\rangle+\langle F^{z\bar{z}}z,\bar{z}\rangle+\langle F^{yz}y,z\rangle+\langle F^{y\bar{z}}y,\bar{z}\rangle\\
     &+F^{zzz}[z,z,z]+F^{zz\bar{z}}[z,z,\bar{z}]+F^{z\bar{z}\bar{z}}[z,\bar{z},\bar{z}]+F^{\bar{z}\bar{z}\bar{z}}[\bar{z},\bar{z},\bar{z}].
 \end{aligned}
 $$
 Denote $\partial_{\omega}=\sum_{j\in J}\omega_{j}\frac{\partial}{\partial x_{j}},\Lambda=diag(\Omega_{j},j\in\mathbb{Z}\setminus J)$. Then we have
 \begin{subequations}
 	\begin{align}
 	H_{+}&=H\circ\Phi\\
 	=&(N+R)\circ X_{F}^{1}+(P-R)\circ X_{F}^{1}\\
 	=&N+\lbrace N,F\rbrace+R+\int_{0}^{1}\lbrace (1-t)\lbrace N,F\rbrace+R,F\rbrace\circ X_{F}^{t}dt+(P-R)\circ X_{F}^{1}\\
 	=&N+\frac12\sum_{j\in\mathbb{Z}\setminus J}\langle\partial_{x}\Omega_{j},F^{y}\rangle z_{j}\bar{z}_{j}+\langle [R^{y}],y\rangle+\langle diag(R^{z\bar{z}})z,\bar{z}\rangle\label{28}\\
 	&+(-\partial_{\omega}F^{x}+R^{x})\label{20}\\
 	&+\langle -\partial_{\omega} F^{y}+R^{y}-[R^{y}],y\rangle\label{21}\\
 	&+\langle -\partial_{\omega} F^{z}+\mathrm{i}\Lambda F^{z}+R^{z},z\rangle\label{22}\\
 	&+\langle -\partial_{\omega} F^{\bar{z}}-\mathrm{i}\Lambda F^{\bar{z}}+R^{\bar{z}},\bar{z}\rangle\label{23}\\
 	&+\langle (-\partial_{\omega}F^{zz}+\mathrm{i}\Lambda F^{zz}+\mathrm{i}F^{zz}\Lambda+R^{zz})z,z\rangle\label{24}\\
 	&+\langle (-\partial_{\omega}F^{\bar{z}\bar{z}}-\mathrm{i}\Lambda F^{\bar{z}\bar{z}}-\mathrm{i}F^{\bar{z}\bar{z}}\Lambda+R^{\bar{z}\bar{z}})\bar{z},\bar{z}\rangle\label{25}\\
 	&+\langle (-\partial_{\omega}F^{z\bar{z}}-\mathrm{i}\Lambda F^{z\bar{z}}+\mathrm{i}F^{z\bar{z}}\Lambda+R^{z\bar{z}}-diag(R^{z\bar{z}}))z,\bar{z}\rangle\label{26}\\
    &+\langle(-\partial_{\omega}F^{yz}+\mathrm{i}\Lambda F^{yz}+R^{yz})y,z\rangle\\
    &+\langle(-\partial_{\omega}F^{y\bar{z}}-\mathrm{i}\Lambda F^{y\bar{z}}+R^{y\bar{z}})y,\bar{z}\rangle\\
    &\sum_{ijl}(-\partial_{\omega}F^{zzz}_{ijl}+\mathrm{i}(\Omega_{i}+\Omega_{j}+\Omega_{l})F^{zzz}_{ijl}+R^{zzz}_{ijl})z_{i}z_{j}z_{l}\\
     &\sum_{ijl}(-\partial_{\omega}F^{\bar{z}\bar{z}\bar{z}}_{ijl}-\mathrm{i}(\Omega_{i}+\Omega_{j}+\Omega_{l})F^{\bar{z}\bar{z}\bar{z}}_{ijl}+R^{\bar{z}\bar{z}\bar{z}}_{ijl})\bar{z}_{i}\bar{z}_{j}\bar{z}_{l}\\
      &\sum_{ijl}(-\partial_{\omega}F^{zz\bar{z}}_{ijl}+\mathrm{i}(\Omega_{i}+\Omega_{j}-\Omega_{l})F^{zz\bar{z}}_{ijl}+\tilde{R}^{zz\bar{z}}_{ijl})z_{i}z_{j}\bar{z}_{l}\\
       &\sum_{ijl}(-\partial_{\omega}F^{z\bar{z}\bar{z}}_{ijl}+\mathrm{i}(\Omega_{i}-\Omega_{j}-\Omega_{l})F^{z\bar{z}\bar{z}}_{ijl}+\tilde{R}^{z\bar{z}\bar{z}}_{ijl})z_{i}\bar{z}_{j}\bar{z}_{l}\\
    &+\int_{0}^{1}\lbrace (1-t)\lbrace N,F\rbrace+R,F\rbrace\circ X_{F}^{t}dt+(P-R)\circ X_{F}^{1}.\label{27}
 	\end{align}
 \end{subequations}
where $\sum_{ijl}\tilde{R}^{zz\bar{z}}_{ijl}z_{i}z_{j}\bar{z}_{l}=\sum_{ijl}R^{zz\bar{z}}_{ijl}z_{i}z_{j}\bar{z}_{l}+\frac12\sum_{mn}\langle \partial_{x}\Omega_{m},F^{yz}_{\cdot n}\rangle z_{n}z_{m}\bar{z}_{m}$ and $\tilde{R}_{ijl}^{z\bar{z}\bar{z}}$ is defined similarly.

 We wish to find the function $F$ such that \eqref{20}-\eqref{26} vanish. To this end, $F^{x},F^{y},F^{z},F^{\bar{z}}$,$F^{zz}$,$F^{\bar{z}\bar{z}}$,$F^{z\bar{z}}$,\\
 $F^{yz}$,$F^{y\bar{z}}$,$F^{zzz}$,$F^{zz\bar{z}}$,$F^{z\bar{z}\bar{z}}$ and $F^{\bar{z}\bar{z}\bar{z}}$ should satisfy the homological equations:
 \begin{subequations}
 	\begin{align}
 	\partial_{\omega} F^{x}&=R^{x},\label{3}\\
 	\partial_{\omega} F^{y}&=R^{y}-[R^{y}],\label{4}\\
 	\partial_{\omega} F^{z}_{j}-\mathrm{i}\Omega_{j} F^{z}_{j}&=R^{z}_{j},\label{5}\\
 	\partial_{\omega} F^{\bar{z}}_{j}+\mathrm{i}\Omega_{j} F^{\bar{z}}_{j}&=R^{\bar{z}},\label{6}\\
 	\partial_{\omega} F^{zz}_{ij}-\mathrm{i}(\Omega_{i}+\Omega_{j})F^{zz}_{ij}&=R^{zz}_{ij},\label{7}\\
 	\partial_{\omega}F^{\bar{z}\bar{z}}_{ij}+\mathrm{i}(\Omega_{i}+\Omega_{j})F^{\bar{z}\bar{z}}_{ij}&=R^{\bar{z}\bar{z}}_{ij},\label{8}\\
 	\partial_{\omega}F^{z\bar{z}}_{ij}+\mathrm{i}(\Omega_{i}-\Omega_{j})F^{z\bar{z}}_{ij}&=R^{z\bar{z}}_{ij},i\neq\pm j,\label{9}\\
 	\partial_{\omega}F^{z\bar{z}}_{(-j)j}+\mathrm{i}(\Omega_{-j}-\Omega_{j})F^{z\bar{z}}_{(-j)j}&=R^{z\bar{z}}_{(-j)j},\vert j\vert\leq K,\label{414}\\
    \partial_{\omega}F^{yz}_{bj}-\mathrm{i}\Omega_{j}F^{yz}_{bj}&=R^{yz}_{bj},\label{901}\\
\partial_{\omega}F^{y\bar{z}}_{bj}+\mathrm{i}\Omega_{j}F^{y\bar{z}}_{bj}&=R^{y\bar{z}}_{bj},\\
\partial_{\omega}F^{zzz}_{ijl}-\mathrm{i}(\Omega_{i}+\Omega_{j}+\Omega_{l})F^{zzz}_{ijl}&=R^{zzz}_{ijl},\\
\partial_{\omega}F^{\bar{z}\bar{z}\bar{z}}_{ijl}+\mathrm{i}(\Omega_{i}+\Omega_{j}+\Omega_{l})F^{\bar{z}\bar{z}\bar{z}}_{ijl}&=R^{\bar{z}\bar{z}\bar{z}}_{ijl},\\
\partial_{\omega}F^{z\bar{z}\bar{z}}_{ijl}-\mathrm{i}(\Omega_{i}-\Omega_{j}-\Omega_{l})F^{z\bar{z}\bar{z}}_{ijl}&=\tilde{R}^{z\bar{z}\bar{z}}_{ijl},\label{902}\\
\partial_{\omega}F^{zz\bar{z}}_{ijl}-\mathrm{i}(\Omega_{i}+\Omega_{j}-\Omega_{l})F^{zz\bar{z}}_{ijl}&=\tilde{R}^{zz\bar{z}}_{ijl}.\label{903}
 	\end{align}
 \end{subequations}

 \subsection{Solving the homological equations}
 
 Equations will be solved  under the following conditions: uniformly on $\Pi$,
 \begin{subequations}
 	\begin{align}
 		\vert \langle k,\omega\vert_{J_{i}}\rangle\vert&\geq\alpha_{i}\frac{1}{\langle k\rangle^{\tau_{i}}},0\neq k\in\mathbb{Z}^{ J_{i}},\label{10}\\
 		\vert\langle k,\omega\rangle+\langle l,\bar{\Omega}\rangle\vert&\geq\beta\frac{\langle l\rangle_{2}}{\langle k\rangle^{\tau_{n}}},k\neq 0,0<\vert l\vert\leq 3,\label{11}\\
 		\vert\langle k,\omega\rangle+\bar{\Omega}_{j}-\bar{\Omega}_{-j}\vert&\geq\beta\frac{\langle j\rangle}{\langle k\rangle^{\tau_{n}}},j\in\mathbb{Z}\setminus J,\vert j\vert\leq K_{1},\label{415}\\
 		\vert\langle l,\bar{\Omega}\rangle\vert&\geq m\langle l\rangle_{2},0<\vert l\vert\leq 2,\\
 		\vert\tilde{\Omega}_{j}^{i}\vert_{s_{i},\tau_{i}}&\leq \alpha_{i}\gamma_{i}\langle j\rangle,\label{12}
 	\end{align}
 \end{subequations}
 with constants $\tau_{i}\geq\sharp J_{i}+10,0<\sum_{i=0}^{n}\gamma_{i}\leq\frac{1}{10},m>0$. Moreover, we assume that $P\in FAE_{\Lambda,\epsilon,\rho}$. For convenience, we denote $\tau_{n}$ simply by $\tau$.
 
 Equations \eqref{3}, \eqref{4} can be easily solved by a standard approach in classical finite dimensional KAM theory, so we only give the related results at the end of this subsection. Equations \eqref{5}-\eqref{8}, \eqref{901}-\eqref{902} are easier than \eqref{9}, \eqref{903} and can be solved in the same way as \eqref{9}\eqref{903} done, so we only give the details of solving \eqref{9}, \eqref{414} and \eqref{903} in the following.
 
 Letting $\Omega_{ij}=\Omega_{i}-\Omega_{j}=\bar{\Omega}_{ij}+\tilde{\Omega}_{ij}$ and dropping the superscript `$z\bar{z}$' for brevity, \eqref{9} become
 \begin{equation}\label{13}
 	-\mathrm{i}\partial_{\omega}F_{ij}+\bar{\Omega}_{ij}F_{ij}+\tilde{\Omega}_{ij}F_{ij}=-\mathrm{i}R_{ij}.
 \end{equation}

 Set $0<\sigma<min\lbrace 1,\frac{s}{n^{Cn}}\rbrace$.  From \eqref{10}, \eqref{11} we get
 $$\begin{aligned}
 \vert\langle k,\omega\vert_{J_{i}}\rangle\vert&\geq\frac{\alpha_{i}}{\vert k\vert^{\tau_{i}}},\\
 \vert\langle k,\omega\rangle+\bar{\Omega}_{ij}\vert&\geq\frac{\beta\vert i^{2}-j^{2}\vert}{1+\vert k\vert^{\tau}},
 \end{aligned}$$
 From \eqref{12} we get
 $$\Vert\tilde{\Omega}_{ij}^{l}\Vert_{s_{l},\tau_{l}}\leq 2\alpha_{l}\gamma_{l}\vert i^{2}-j^{2}\vert.$$
 Applying Lemma \ref{118} to \eqref{13}, we have
 \begin{equation}\label{14}
 	\Vert F_{ij}\Vert_{D(\frac{s-\sigma}{n^{Cn}})}\leq\frac{C}{\beta\vert i^{2}-j^{2}\vert\sigma^{C(n+\tau)}}\Vert R_{ij}\Vert_{D(s-\sigma)}.
 \end{equation}
 
 For \eqref{414}, by \eqref{10},\eqref{415},\eqref{12} and Lemma \ref{416}, we obtain that 
 \begin{equation}\label{417}
 	\Vert F_{(-j)j}\Vert_{D(s-2\sigma)}\leq\frac{C}{\langle j\rangle\beta\sigma^{C(n+\tau)}}e^{C\langle j\rangle s\sum_{i=0}^{n}\gamma_{i}}\Vert R_{(-j)j}\Vert_{D(s-\sigma)}\leq\frac{C}{\langle j\rangle\beta\sigma^{C(n+\tau)}}\Vert R_{(-j)j}\Vert_{D(s-\sigma)}.
 \end{equation}
 if $e^{CK_{1} s\sum_{i=0}^{n}\gamma_{i}}\leq \frac{1}{\sigma^{C(n+\tau)}}$.
 
 For a bounded linear operator from $l^{a,p}$ to $l^{a,p-1}$, define its operator norm by $\Vert\cdot\Vert_{a,p-1,p}$. As in Lemma 19.1 of \cite{KP2}, in view of \eqref{14}\eqref{417}, using Lemma \ref{56} below, we get the estimate of $F^{z\bar{z}}$:
 $$\begin{aligned}
 \Vert F^{z\bar{z}}\Vert_{a,p,p,D(\frac{s-\sigma}{n^{Cn}})},\Vert F^{z\bar{z}}\Vert_{a,p-1,p-1,D(\frac{s-\sigma}{n^{Cn}})}&\leq\frac{C}{\beta \sigma^{C(n+\tau)}}\Vert R^{z\bar{z}}\Vert_{a,p-1,p,D(s)}\\
 &\leq\frac{C}{\beta\sigma^{C(n+\tau)}}\Vert X_{R}\Vert_{r,a,p-1,D(s,r)}.
 \end{aligned}$$
 
 Multiplying by $z,\bar{z}$ we then get
 $$\frac{1}{r^{2}}\Vert\langle F^{z\bar{z}}z,\bar{z}\rangle\Vert_{D(\frac{s-\sigma}{n^{Cn}},r)}\leq\Vert F^{z\bar{z}}\Vert_{a,p,p,D(\frac{s-\sigma}{n^{Cn}})},$$
 and finally by Cauchy's estimate we have 
 \begin{equation}\label{17}
 	\Vert X_{\langle F^{z\bar{z}}z,\bar{z}\rangle}\Vert_{r,a,p,D(\frac{s-\sigma}{n^{Cn}}-\sigma,r)}\leq\frac{C}{\beta\sigma^{C(n+\tau)}}\Vert X_{R}\Vert_{r,a,p-1,D(s,r)}.
 \end{equation}

 To obtain the estimate of the Lipschitz semi-norm, we proceed as follows. Shortening $\Delta_{\sigma\sigma'}$ as $\Delta$ and applying it to \eqref{13}, one gets that
 \begin{equation}\label{15}
 	\mathrm{i}\partial_{\omega}(\Delta F_{ij})+\bar{\Omega}_{ij}\Delta F_{ij}+\tilde{\Omega}_{ij}\Delta F_{ij}=-\mathrm{i}\partial_{\Delta\omega}F_{ij}-(\Delta\Omega_{ij})F_{ij}+\mathrm{i}\Delta R_{ij}=:Q_{ij}.
 \end{equation}
 
 We have
 $$\begin{aligned}
 \Vert Q_{ij}\Vert_{D(\frac{s-\sigma}{n^{Cn}}-\sigma)}&\leq\frac{\vert\Delta\omega\vert}{\sigma}\Vert F_{ij}\Vert_{D(\frac{s-\sigma}{n^{Cn}})}+(\langle i\rangle+\langle j\rangle)\Vert\Delta\Omega\Vert_{-1,D(s)}\Vert F_{ij}\Vert_{D(\frac{s-\sigma}{n^{Cn}})}+\Vert\Delta R_{ij}\Vert_{D(s-\sigma)}\\
 &\leq\frac{C}{\beta\sigma^{C(n+\tau)}}(\vert\Delta\omega\vert+\Vert\Delta\Omega\Vert_{-1,D(s)})\Vert R_{ij}\Vert_{D(s-\sigma)}+\Vert\Delta R_{ij}\Vert_{D(s-\sigma)}\\
 &\leq\frac{C}{\sigma^{C(n+\tau)}}(\frac{\vert\Delta\omega\vert+\vert\Delta\Omega\vert_{-1,D(s)}}{\beta}\Vert R_{ij}\Vert_{D(s-\sigma)}+\Vert\Delta R_{ij}\Vert_{D(s-\sigma)})
 \end{aligned}$$
 Again applying Lemma \ref{118} to \eqref{15}, we have
 \begin{equation}\label{16}
 	\Vert\Delta F_{ij}\Vert_{D(\frac{s-\sigma}{n^{2Cn}}-\frac{\sigma}{n^{Cn}})}\leq\frac{C}{\beta\vert i^{2}-j^{2}\vert\sigma^{C(n+\tau)}}(\frac{\vert\Delta\omega\vert+\Vert\Delta\Omega\Vert_{-1,D(s)}}{\beta}\Vert R_{ij}\Vert_{D(s-\sigma)}+\Vert\Delta R_{ij}\Vert_{D(s-\sigma)}).
 \end{equation}
 Similarly, applying Lemma \ref{413} to \eqref{414}, we obtain that
 \begin{equation}\label{418}
 \Vert\Delta F_{(-j)j}\Vert_{D(s-4\sigma)}\leq\frac{C}{\beta\langle j\rangle\sigma^{C(n+\tau)}}(\frac{\vert\Delta\omega\vert+\Vert\Delta\Omega\Vert_{-1,D(s)}}{\beta}\Vert R_{(-j)j}\Vert_{D(s-\sigma)}+\Vert\Delta R_{(-j)j}\Vert_{D(s-\sigma)}).
 \end{equation}
 In view of \eqref{16}\eqref{418}, applying Lemma \ref{56} below again, we get the estimates of $\Delta F^{z\bar{z}}$:
 $$\begin{aligned}
 \Vert \Delta &F^{z\bar{z}}\Vert_{a,p,p,D(\frac{s-\sigma}{n^{2Cn}}-\frac{\sigma}{n^{Cn}})},\Vert \Delta F^{z\bar{z}}\Vert_{a,p-1,p-1,D(\frac{s-\sigma}{n^{2Cn}}-\frac{\sigma}{n^{Cn}})}\\
 &\leq\frac{C}{\beta\sigma^{C(n+\tau)}}(\frac{\vert\Delta\omega\vert+\Vert\Delta\Omega\Vert_{-1,D(s)}}{\beta}\Vert R^{z\bar{z}}\Vert_{a,p-1,p,D(s)}+\Vert\Delta R^{z\bar{z}}\Vert_{a,p-1,p,D(s)}).
 \end{aligned}$$
 Dividing by $\Vert\sigma-\sigma'\Vert_{l^{2}}\neq 0$ and taking the supremum over $\Pi$, we get
 $$\begin{aligned}
 \Vert  &F^{z\bar{z}}\Vert^{lip}_{a,p,p,D(\frac{s-\sigma}{n^{2Cn}}-\frac{\sigma}{n^{Cn}})},\Vert  F^{z\bar{z}}\Vert^{lip}_{a,p-1,p-1,D(\frac{s-\sigma}{n^{2Cn}}-\frac{\sigma}{n^{Cn}})}\\
 &\leq\frac{C}{\beta\sigma^{C(n+\tau)}}(\frac{M}{\beta}\Vert R^{z\bar{z}}\Vert_{a,p-1,p,D(s)\times\Pi}+\Vert R^{z\bar{z}}\Vert^{lip}_{a,p-1,p,D(s)\times\Pi}).
 \end{aligned}$$
 where $M:=\vert\omega\vert_{\Pi}^{lip}+\Vert\Omega\Vert^{lip}_{-1,D(s)\times\Pi}$. Thus, in the same way as \eqref{17}, we get
 \begin{equation}\label{18}
 	\begin{aligned}
 	\Vert  &X_{\langle F^{z\bar{z}}z,\bar{z}\rangle}\Vert^{lip}_{r,a,p,D(\frac{s-\sigma}{n^{2Cn}}-\frac{\sigma}{n^{Cn}}-\sigma,r)\times\Pi}\\
 	&\leq\frac{C}{\beta\sigma^{C(n+\tau)}}(\frac{M}{\beta}\Vert X_{R}\Vert_{r,a,p-1,D(s,r)\times\Pi}+\Vert X_{R}\Vert^{lip}_{r,a,p-1,D(s,r)\times\Pi}).
 	\end{aligned}
 \end{equation}

 For $\lambda\geq 0$, define
 $$\Vert\cdot\Vert^{\lambda}_{\Pi}=\Vert\cdot\Vert_{\Pi}+\lambda\Vert\cdot\Vert_{\Pi}^{lip}.$$
 Set $0\leq\lambda\leq\frac{\beta}{M}$. From \eqref{17}\eqref{18} we get
 \begin{equation}\label{19}
 	\Vert  X_{\langle F^{z\bar{z}}z,\bar{z}\rangle}\Vert^{\lambda}_{r,a,p,D(\frac{s-\sigma}{n^{2Cn}}-\frac{\sigma}{n^{Cn}},r)\times\Pi}\\
 	\leq\frac{C}{\beta\sigma^{C(n+\tau)}}\Vert X_{R}\Vert^{\lambda}_{r,a,p-1,D(s,r)\times\Pi}.
 \end{equation}

 We now give the estimate for the most difficult cubic term $F^{zz\bar{z}}$. Write
 $$\sum_{ijl}\tilde{R}^{zz\bar{z}}_{ijl}z_{i}z_{j}\bar{z}_{l}=:R_{1}+R_{2},$$
where $R_{1}=(\sum_{j=l,|j|>10|i|}+\sum_{i=l,|i|>10|j|})\tilde{R}^{zz\bar{z}}_{ijl}z_{i}z_{j}\bar{z}_{l}$ and $R_{2}$ is the complementary part. For $A\in\mathbb{Z}$, let $R_{A}:=\sum_{i+j-l=A}R_{2,ijl}z_{i}z_{j}\bar{z}_{l}$. We solve \eqref{903} only for $R_{A}$ with $|A|\le K_{2}$. The terms $R_{1}$ and $\sum_{|A|>K_{2}}R_{A}$ are left in the new perturbation. For coefficients with $i\ne l, j\ne l$ and $i+j-l=A$, applying Lemma \ref{416}, we have
\begin{equation}\label{904}
    \Vert F_{ijl}^{zz\bar{z}}\Vert_{D(s-2\sigma)}\le\frac{Ce^{CK_{2}s}}{\beta \langle i^{2}+j^{2}-l^{2}\rangle\sigma^{C(n+\tau)}}\Vert R_{2,ijl}\Vert_{D(s-\sigma)}.
\end{equation}

Fix $u\in l^{a,p}$ with $\Vert u\Vert_{a,p}\le 1$ and define the matrices $(R_{A}^{u})_{lj}:=R_{2,A+l-j,j,l}u_{A+l-j},(F_{A}^{u})_{lj}:=F_{A+l-j,j,l}u_{A+l-j}$. With $\frac{\langle l\rangle}{\langle i^{2}+j^{2}-l^{2}\rangle}\le C\langle A\rangle^{4}(\frac{1}{|l-j|}+\frac{1}{|j-A|})$, \eqref{904} and Lemma \ref{56} gives 
$$\Vert F_{A}^{u}\Vert_{a,p,p,D(s-2\sigma)}\le \frac{C\langle A\rangle^{4}e^{CK_{2}s}}{\beta\sigma^{C(n+\tau)}}\Vert R_{A}^{u}\Vert_{a,p,p-1,D(s)}.$$
For a bounded linear operator from $l^{a,p}\times l^{a,p}$ to $l^{a,p-1}$, define its operator norm by $\Vert\cdot\Vert_{a,p,p,p-1}$. By the arbitrariness of $u$, we obtain
$$\Vert F_{A}\Vert_{a,p,p,p,D(s-2\sigma)}\le\frac{C\langle A\rangle^{4}e^{CK_{2}s}}{\beta\sigma^{C(n+\tau)}}\Vert R_{A}\Vert_{a,p,p,p-1,D(s)}.$$
Summing over $A$, we deduce that
$$\Vert F^{zz\bar{z}}\Vert_{a,p,p,p,D(s-2\sigma)}\le\sum_{|A|\le K}\Vert F_{A}\Vert_{a,p,p,p,D(s-2\sigma)}\le\frac{Ce^{CK_{2}s}}{\beta\sigma^{C(n+\tau)}}\Vert R^{zz\bar{z}}\Vert_{a,p,p,p-1,D(s)}.$$
Finally by Cauchy's estimate we have
$$\Vert X_{F^{zz\bar{z}}[z,z,\bar{z}]}\Vert_{r,a,p,D(s-3\sigma,r)}\le\frac{Ce^{CK_{2}s}}{\beta\sigma^{C(n+\tau)}}\Vert X_{R}\Vert_{r,a,p-1,D(s,r)}. $$
For $l=j,10|i|>|j|$, because $\frac{\langle l\rangle}{\langle i^{2}+j^{2}-l^{2}\rangle}\le\frac{C}{\langle j\rangle}$, analogous estimates remain valid. The estimate for the Lipschitz semi-norm is completely similar and we omit it.
 
 Now considering the homological equations \eqref{3}\eqref{4}, by a standard approach in finite dimensional KAM theory, we can easily get
 \begin{equation}\label{30}
 	\Vert X_{F^{x}}\Vert^{\lambda}_{r,a,p,D(s-2\sigma,r)\times\Pi},\Vert X_{\langle F^{y},y\rangle}\Vert^{\lambda}_{r,a,p,D(s-2\sigma,r)\times\Pi}\leq\frac{C}{\beta\sigma^{C(n+\tau)}}\Vert X_{R}\Vert^{\lambda}_{r,a,p-1,D(s,r)\times\Pi}.
 \end{equation}

 For the other terms of $F$, i.e. $\langle F^{z},z\rangle,\langle F^{\bar{z}},\bar{z}\rangle,\langle F^{zz}z,z\rangle,\langle F^{\bar{z}\bar{z}}\bar{z},\bar{z}\rangle, \cdots$, the same results as \eqref{19} can be obtained. Thus, we finally get the estimate for $F$:
 \begin{equation}\label{33}
 	\Vert X_{F}\Vert^{\lambda}_{r,a,p,D(\frac{s}{n^{3Cn}},r)\times\Pi}\leq\frac{Ce^{CK_{2}s}}{\beta\sigma^{C(n+\tau)}}\Vert X_{R}\Vert^{\lambda}_{r,a,p-1,D(s,r)\times\Pi}.
 \end{equation}
 
 Following similar steps as above, we can also obtain that
 \begin{equation}\label{41}
 \begin{aligned}
 &\Vert X_{\partial_{\sigma_{d}}F}\Vert_{r,a,p,D(\frac{s}{n^{3Cn}},r)\times\Pi}\\
 &\leq\frac{Ce^{CK_{2}s}}{\beta\sigma^{C(n+\tau)}}(\frac{D}{\beta}\Vert X_{R}\Vert_{r,a,p-1,D(s,r)\times\Pi}+\Vert X_{\partial_{\sigma_{d}}R}\Vert_{r,a,p-1,D(s,r)\times\Pi})
 \end{aligned}
 \end{equation}
 for $\forall d\in\mathbb{Z}$, where $D:=\vert\partial_{\sigma_{d}}\omega\vert_{\Pi}+\Vert\partial_{\sigma_{d}}\Omega\Vert_{-1,D(s)\times\Pi}$.
 \section{The New Hamiltonian}
 From \eqref{28}-\eqref{27} we get the new Hamiltonian
 $$H\circ\Phi=N_{+}+P_{+},$$
 where $N_{+}=$\eqref{28} and 
 $$P_{+}=\int_{0}^{1}\lbrace (1-t)\lbrace N,F\rbrace+R,F\rbrace\circ X_{F}^{t}dt+(P-R)\circ X_{F}^{1}+R',$$ 
 where $R'=R_{2}+R_{3}:=\sum_{\vert j\vert>K}R_{(-j)j}z_{-j}\bar{z}_{j}+R_{3}$. Here $R_{3}$ denotes the remaining cubic terms from the above discussion.
 \subsection{The new normal form}
 In view of \eqref{28}, denote $N_{+}=N+N'$ with
 $$N'=\sum_{j\in J}\omega'_{j}y_{j}+\frac12\sum_{j\in\mathbb{Z}\setminus J}\Omega'_{j}z_{j}\bar{z}_{j},$$
 with
 \begin{equation}\label{29}
 	\omega'=[R^{y}],
 \end{equation}
 
 $$\Omega'_{j}=2R^{z\bar{z}}_{jj}+\langle \partial_{x}\Omega_{j},F^{y}\rangle.$$
 From \eqref{29} we easily get
 \begin{equation}\label{35}
 	\vert\omega'\vert^{\lambda}_{\Pi}\leq C\Vert X_{R}\Vert^{\lambda}_{r,a,p-1,D(s,r)\times\Pi}.
 \end{equation}
 In view of \eqref{30},
 $$\Vert\langle\partial_{x}\tilde{\Omega}_{j}^{l},F^{y}\rangle\Vert_{D(s-\sigma)}\leq\Vert\tilde{\Omega}_{j}^{l}\Vert_{s,\tau_{l}}\Vert X_{\langle F^{y},y\rangle}\Vert_{r,a,p,D(s-\sigma,r)}\leq\frac{C\alpha_{l}\gamma_{l}\langle j\rangle}{\beta\sigma^{C(n+\tau)}}\Vert X_{R}\Vert_{r,a,p-1,D(s,r)\times\Pi}.$$
 Thus, together with
 $$\Vert R_{jj}\Vert_{D(s-\sigma)}\leq \langle j\rangle\Vert X_{R}\Vert_{r,a,p-1,D(s,r)},$$
 we get
 \begin{equation}\label{31}
 \Vert\Omega'\Vert_{-1,D(s-\sigma)}\leq\frac{C}{\sigma^{C(n+\tau)}}\Vert X_{R}\Vert_{r,a,p-1,D(s,r)}
 \end{equation}
 if $\frac{\sum_{l=0}^{n}\alpha_{l}\gamma_{l}}{\beta}\leq\frac{1}{\sigma^{C(n+\tau)}}$. Applying $\Delta$ to $\Omega'_{j}$, we have
 $$\Delta\Omega'_{j}=2\Delta R_{jj}+\langle\partial_{x}\Delta\tilde{\Omega}_{j},F^{y}\rangle+\langle\partial_{x}\tilde{\Omega}_{j},\Delta F^{y}\rangle.$$
 Since
 $$\begin{aligned}
 \Vert\Delta R_{jj}\Vert_{D(s-2\sigma)}&\leq\langle j\rangle\Vert\Delta X_{R}\Vert_{r,a,p-1,D(s,r)},\\
 \Vert\langle\partial_{x}\Delta\tilde{\Omega}_{j}^{l},F^{y}\rangle\Vert_{D(s-2\sigma)}&\leq\frac{1}{\sigma}\Vert\Delta\Omega_{j}^{l}\Vert_{D(s-\sigma)}\Vert X_{\langle F^{y},y\rangle}\Vert_{r,a,p,D(s-2\sigma,r)}\\
 &\leq\frac{C\Vert\Delta\Omega_{j}^{l}\Vert_{-1,D(s-\sigma)}\langle j\rangle}{\beta\sigma^{C(n+\tau)}}\Vert X_{R}\Vert_{r,a,p-1,D(s,r)},\\
 \Vert\langle\partial_{x}\tilde{\Omega}_{j}^{l},\Delta F^{y}\rangle\Vert_{D(s-2\sigma)}&\leq \Vert\tilde{\Omega}_{j}^{l}\Vert_{s,\tau_{l}}\Vert\Delta X_{\langle F^{y},y\rangle}\Vert_{r,a,p,D(s-2\sigma,r)}\\
 &\leq\alpha_{l}\gamma_{l}\langle j\rangle\Vert\Delta X_{\langle F^{y},y\rangle}\Vert_{r,a,p,D(s-2\sigma,r)},
 \end{aligned}$$
 we get
 \begin{equation}\label{32}
 	\Vert\Omega'\Vert^{lip}_{-1,D(s-2\sigma)\times\Pi}\leq\frac{C}{\sigma^{C(n+\tau)}}(\frac{M}{\beta}\Vert X_{R}\Vert_{r,a,p-1,D(s,r)\times\Pi}+\Vert X_{R}\Vert^{lip}_{r,a,p-1,D(s,r)\times\Pi})
 \end{equation}
  if $\frac{\sum_{l=0}^{n}\alpha_{l}\gamma_{l}}{\beta}\leq\frac{1}{\sigma^{C(n+\tau)}}$. Therefore, from \eqref{31}\eqref{32}, we get
 $$\Vert\Omega'\Vert^{\lambda}_{-1,D(s-2\sigma)\times\Pi}\leq\frac{C}{\sigma^{C(n+\tau)}}\Vert X_{R}\Vert^{\lambda}_{r,a,p-1,D(s,r)\times\Pi}.$$
 \subsection{The new perturbation}
 We first estimate the error term $R'$. In view of $P\in FAE_{\Lambda,\epsilon,\rho}$, we have that
 $$\sum_{\vert j\vert>K_{1}}\vert j\vert^{2(p-1)}e^{2a\vert j\vert}\vert\frac{\partial R_2}{\partial z_{-j}}\vert^{2}=\sum_{\vert j\vert>K_{1}}\vert j\vert^{2(p-1)}e^{2a\vert j\vert}\vert R_{(-j)j}\vert^{2}\vert \bar{z}_{j}\vert^{2}\leq r^{2}\epsilon^{2} \sum_{\vert j\vert>K_{1}} e^{-4\rho\vert j\vert}$$
 on $D(s,r)$. Similarly, we obtain the estimate for $\frac{\partial R_{2}}{\partial \bar{z}},\frac{\partial R_{2}}{\partial x}$, and therefore we have 
 \begin{equation}\label{419}
 	\Vert X_{R_{2}}\Vert^{\lambda}_{r,a,p-1,D(s-\sigma,r)\times\Pi}\leq\frac{C \epsilon}{\sigma^{C(n+\tau)}} e^{-\rho K_{1}}.
 \end{equation}
Let $R_{31}$ and $R_{32}$ be the diagonal parts and high-frequency remainder ($|A|>K_{2}$) respectively in $R_{3}$. From the exponential decay of the Fourier coefficients and $\sum_{b\in J}|k_{b}|\geq \frac{|\sum_{b\in J}k_{b}b|}{Cn}=\frac{|A|}{Cn}$, we readily obtain
\begin{equation}\label{905}
\begin{aligned}
  \Vert X_{R_{31}}\Vert^{\lambda}_{r,a,p-1,D(\frac{s}{n^{3Cn}},r)\times\Pi}&\le\frac{C}{\beta \sigma^{C(n+\tau)}}\Vert X_{R}\Vert^{\lambda}_{r,a,p-1,D(s,r)\times\Pi},\\
  \Vert X_{R_{32}}\Vert^{\lambda}_{r,a,p-1,D(\frac{s}{n^{3Cn}},r)\times\Pi}&\le\frac{Ce^{-\frac{K_{2}s}{Cn}}}{\beta \sigma^{C(n+\tau)}}\Vert X_{R}\Vert^{\lambda}_{r,a,p-1,D(s,r)\times\Pi}.
\end{aligned}
 \end{equation}
 
 Now we consider the new perturbation. By setting $R(t)=(1-t)(N'+R')+tR$, we have
 $$X_{P_{+}}=\int_{0}^{1}(X_{F}^{t})^{*}[X_{R(t)},X_{F}]dt+(X_{F}^{1})^{*}(X_{P}-X_{R})+X_{R'}.$$
 We assume that
 $$\Vert X_{P}\Vert_{r,a,p-1,D(s,r)\times\Pi}^{\lambda}\leq\frac{\beta\eta^{2}}{B_{\sigma}}e^{-CK_{2}s},$$
  for $0\leq\lambda\leq\frac{\beta}{M}$ with some $0<\eta<\frac{1}{16}$, where $B_{\sigma}=(\tau^{\tau})^{exp}\sigma^{-C\tau}$. Furthermore, suppose that $\epsilon\leq B_{\sigma}\Vert X_{P}\Vert_{r,a,p-1,D(s,r)\times\Pi}^{\lambda}$. Since $R$ is the 3-order Taylor polynomial truncation in $y,z,\bar{z}$ of $P$, we can obtain
  \begin{equation}\label{34}
  	\Vert X_{R}\Vert^{\lambda}_{r,a,p-1,D(s,r)\times\Pi}\leq C\Vert X_{P}\Vert^{\lambda}_{r,a,p-1,D(s,r)\times\Pi},
  \end{equation}
  \begin{equation}\label{36}
  	\Vert X_{P}-X_{R}\Vert^{\lambda}_{\eta r,a,p-1,D(s,4\eta r)\times\Pi}\leq C\eta^{2}\Vert X_{P}\Vert^{\lambda}_{r,a,p-1,D(s,r)\times\Pi}.
  \end{equation}
  As in Lemma 19.3 of \cite{KP2}, and from \eqref{33}\eqref{34}, we get
  $$\Vert X_{F}\Vert^{\lambda}_{r,a,p,D(\frac{s}{n^{3Cn}},r)\times\Pi},\Vert DX_{F}\Vert^{\lambda}_{r,a,p,p,D(\frac{s}{n^{4Cn}},r)\times\Pi},$$
  $$\Vert DX_{F}\Vert^{\lambda}_{r,a,p-1,p-1,D(\frac{s}{n^{4Cn}},r)\times\Pi}\leq\frac{\eta^{2}\sigma}{c_{0}}$$
  with some suitable constant $c_{0}\geq 1$. Then the flow $X_{F}^{t}$ of the vector field $X_{F}$ exists on $D(\frac{s}{n^{4Cn}}-\sigma,\frac{r}{2})$ for $-1\leq t\leq 1$ and takes this domain into $D(\frac{s}{n^{4Cn}},r)$. Similarly, it takes $D(\frac{s}{n^{4Cn}}-2\sigma,\frac{r}{4})$ into $D(\frac{s}{n^{4Cn}}-\sigma,\frac{r}{2})$. In the same way as (20.6) in \cite{KP2}, we obtain
  \begin{equation}\label{302}
  \begin{aligned}
  \Vert X_{F}^{t}-id\Vert^{\lambda}_{r,a,p,D(\frac{s}{n^{4Cn}}-\sigma,\frac{r}{2})\times\Pi}&\leq C\Vert X_{F}\Vert^{\lambda}_{r,a,p,D(\frac{s}{n^{4Cn}},r)\times\Pi},\\
  \Vert DX_{F}^{t}-I\Vert^{\lambda}_{r,a,p,p,D(\frac{s}{n^{4Cn}}-2\sigma,\frac{r}{4})\times\Pi}&\leq C\Vert DX_{F}\Vert^{\lambda}_{r,a,p,p,D(\frac{s}{n^{4Cn}},r)\times\Pi},\\
  	\Vert DX_{F}^{t}-I\Vert^{\lambda}_{r,a,p-1,p-1,D(\frac{s}{n^{4Cn}}-2\sigma,\frac{r}{4})\times\Pi}&\leq C\Vert DX_{F}\Vert^{\lambda}_{r,a,p-1,p-1,D(\frac{s}{n^{4Cn}},r)\times\Pi}.
  \end{aligned}
  \end{equation}
 Also in the same way as (20.7) in \cite{KP2}, we obtain that for any vector field $Y$,
 $$\Vert(X_{F}^{t})^{*}Y\Vert^{\lambda}_{\eta r,a,p-1,D(\frac{s}{n^{4Cn}}-3\sigma,\eta r)\times\Pi}\leq C\Vert Y\Vert^{\lambda}_{\eta r,a,p-1,D(\frac{s}{n^{4Cn}}-\sigma,4\eta r)\times\Pi}.$$
 From \eqref{34}\eqref{35}\eqref{31}\eqref{419}\eqref{905}, we have
 $$\begin{aligned}
 	\Vert[X&_{R(t)},X_{F}]\Vert^{\lambda}_{r,a,p-1,D(\frac{s}{n^{4Cn}},\frac{r}{2})\times\Pi}\\
 	\leq&C\Vert DX_{R(t)}\Vert^{\lambda}_{r,a,p-1,p,D(\frac{s}{n^{4Cn}},\frac{r}{2})\times\Pi}\Vert X_{F}\Vert^{\lambda}_{r,a,p,D(\frac{s}{n^{4Cn}},\frac{r}{2})\times\Pi}\\
 	&+C\Vert DX_{F}\Vert^{\lambda}_{r,a,p-1,p-1,D(\frac{s}{n^{4Cn}},\frac{r}{2})\times\Pi}\Vert X_{R(t)}\Vert^{\lambda}_{r,a,p-1,D(\frac{s}{n^{4Cn}},\frac{r}{2})\times\Pi}\\
 	\leq&\frac{Ce^{CK_{2}s}}{\beta \sigma^{C(n+\tau)}}(\Vert X_{P}\Vert^{\lambda}_{r,a,p-1,D(s,r)\times\Pi})^{2}.
 \end{aligned}$$
 Hence, also
 $$\Vert[X_{R(t)},X_{F}]\Vert^{\lambda}_{\eta r,a,p-1,D(\frac{s}{n^{4Cn}},\frac{r}{2})\times\Pi}\leq\frac{Ce^{CK_{2}s}}{\beta \sigma^{C(n+\tau)}\eta^{2}}(\Vert X_{P}\Vert^{\lambda}_{r,a,p-1,D(s,r)\times\Pi})^{2}.$$
Together with \eqref{36}, we finally arrive at the estimate
 $$\Vert X_{P_{+}}\Vert^{\lambda}_{\eta r,a,p-1,D(\frac{s}{n^{5Cn}},\eta r)\times\Pi}\leq C(\frac{B_{\sigma}e^{CK_{2}s}}{\beta\eta^{2}}\Vert X_{P}\Vert^{\lambda}_{r,a,p-1,D(s,r)\times\Pi}+B_{\sigma} e^{-\rho K_{1}}+B_{\sigma}\eta)\Vert X_{P}\Vert^{\lambda}_{r,a,p-1,D(s,r)\times\Pi}.$$
 
 \subsection{The persistence of asymptotic estimate condition}

 In this subsection we prove the persistence of asymptotic estimate condition at every KAM step under the following conditions:
 
 (i)$\lambda=((\omega_{j})_{j\in J},(\Omega_{j})_{j\in\mathbb{Z}\setminus J})$ satisfies
 \begin{equation}\label{42}
 	\lambda_{n}=\vert n\vert^{2}+\sigma_{n}+n\bar{\lambda}+\tilde{\lambda}+\hat{\lambda}_{n},n\in\mathbb{Z}
 \end{equation}
 with estimates for $\forall n,d\in\mathbb{Z}$
 \begin{equation}\label{43}
 	\begin{aligned}
 	&\Vert\hat{\lambda}_{n}\Vert_{D(s,r)\times\Pi}\leq\epsilon_{0}\langle n\rangle^{-1},\Vert\frac{\partial}{\partial\sigma_{d}}\hat{\lambda}_{n}\Vert_{D(s,r)\times\Pi}\leq\epsilon_{0}(\langle n\rangle^{-1}\langle d\rangle^{-1}+e^{-2\vert n-d\vert\rho}+e^{-2\vert n+d\vert\rho}),\\
 	&\Vert\bar{\lambda}\Vert_{D(s,r)\times\Pi,L},\Vert\tilde{\lambda}\Vert_{D(s,r)\times\Pi,L}\leq\epsilon_{0}.
 	\end{aligned}
 \end{equation}
 
 (ii)$P\in FAE_{\Lambda,\epsilon,\rho}$, $\Vert X_{P}\Vert_{r,a,p-1,D(s,r)\times\Pi}^{\lambda}\leq\epsilon$ and $\Vert X_{\partial_{\sigma_{d}}P}\Vert_{r,a,p-1,D(s,r)\times\Pi}\leq\epsilon\frac{1}{\langle d\rangle}$ for $\forall d\in\mathbb{Z}$.
 
 We first need to obtain that the solution of homological equation $F\in SAE_{\Lambda,\epsilon^{C},\rho}$ for $0<C<1$ and sufficiently small $\epsilon$. We still only consider the homological equations \eqref{9}\eqref{903} which correspond respectively to the quadratic and cubic terms and other cases are easier. We first examine the quadratic part of $F$.
 $$-\mathrm{i}\partial_{\omega}F_{ij}+\bar{\Omega}_{ij}F_{ij}+\tilde{\Omega}_{ij}F_{ij}=-\mathrm{i}R_{ij}.$$
  Let $i=m+t\neq \pm j=\pm(n+t)$. From Lemma \ref{118}, we can immediately obtain the estimate
  $$\Vert\frac{\partial^{2}F}{\partial q_{m+t}\partial\bar{q}_{n+t}}\Vert_{D(\frac{s}{n^{Cn}},r)\times\Pi,L}\leq \epsilon^{C} e^{-\vert n-m\vert\rho},(m+t,n+t)\in\mathbb{Z}^{2}.$$
  When $\vert t\vert\geq\Lambda \max\lbrace\vert m\vert,\vert n\vert\rbrace,t\neq 0$, we can write the quadratic part of $P$
  $$R_{m+t\ n+t}=tR_{1}+R_{2}+R_{3}\frac{1}{t}=t(R_{1}^{1}+R_{1}^{2}\frac{1}{t}+R_{1}^{3}\frac{1}{t^{2}})+(R_{2}^{1}+R_{2}^{2}\frac{1}{t})+R_{3}\frac{1}{t}$$ 
  by the first type asymptotic estimate condition of $P$. Since $t$ is sufficiently large, `a big denominator' appears in the homological equations. Therefore, we set $F_{m+t\ n+t}=F_{1}+\frac{F_{2}}{t}+\frac{F_{3}}{t^{2}}$, where $F_{1},F_{2},F_{3}$ represent the solutions to the following equations.
 $$\bar{\Omega}_{m+t\ n+t}F_{1}+\tilde{\Omega}_{m+t\ n+t}F_{1}=-\mathrm{i}tR_{1},$$
  $$\bar{\Omega}_{m+t\ n+t}F_{2}+\tilde{\Omega}_{m+t\ n+t}F_{2}=-\mathrm{i}tR_{2}+\mathrm{i}t\partial_{\omega}F_{1},$$
   $$-\mathrm{i}\partial_{\omega}F_{3}+\bar{\Omega}_{m+t\ n+t}F_{3}+\tilde{\Omega}_{m+t\ n+t}F_{3}=-\mathrm{i}tR_{3}+\mathrm{i}t\partial_{\omega}F_{2}.$$
 
 We first focus on the first algebraic equation and obtain its solution directly.
 $$F_{1}=\frac{-\mathrm{i}tR_{1}}{\Omega_{m+t\ n+t}}.$$
 Since $\Omega_{m+t}=(m+t)^{2}+\sigma_{m+t}+(m+t)\bar{\lambda}+\tilde{\lambda}+c_{m+t}t^{-1},\Omega_{n+t}=(n+t)^{2}+\sigma_{n+t}+(n+t)\bar{\lambda}+\tilde{\lambda}+c_{n+t}t^{-1}$, we have
 $$F_{1}=\frac{-\mathrm{i}tR_{1}}{(m-n)(m+n+\bar{\lambda}+2t)+\sigma_{m+t}-\sigma_{n+t}+(c_{m+t}-c_{n+t})t^{-1}}=\frac{1}{2}\frac{-\mathrm{i}R_{1}}{(m-n)(\frac{m+n+\bar{\lambda}}{2t}+\frac{C}{t^{2}}+1)}$$
 From the Taylor expansion of function $(1+x)^{-1}$ when $\vert x\vert<1$, we have 
 $$F_{1}=F_{1}^{1}+F_{1}^{2}\frac{1}{t}+F_{1}^{3}\frac{1}{t^{2}},$$
 where 
 $$F_{1}^{1}=\frac{-\mathrm{i}R_{1}^{1}}{2(m-n)},$$
 $$F_{1}^{2}=\frac{-\mathrm{i}R_{1}^{2}}{2(m-n)}+\frac{\mathrm{i}R_{1}^{1}}{2(m-n)}\frac{m+n+\bar{\lambda}}{2},$$
 $$
 \begin{aligned}
 F_{1}^{3}=&\frac{-\mathrm{i}R_{1}^{3}}{2(m-n)}\sum_{i=0}^{\infty}(-1)^{i}(\frac{m+n+\bar{\lambda}}{2t}+\frac{C}{t^{2}})^{i}+\frac{\mathrm{i}R_{1}^{2}t}{2(m-n)}\sum_{i=1}^{\infty}(-1)^{i}(\frac{m+n+\bar{\lambda}}{2t}+\frac{C}{t^{2}})^{i}\\
 &+\frac{\mathrm{i}R_{1}^{1}t^{2}}{2(m-n)}[\sum_{i=2}^{\infty}(-1)^{i}(\frac{m+n+\bar{\lambda}}{2t}+\frac{C}{t^{2}})^{i}+\frac{C}{t^{2}}].
 \end{aligned}
 $$
 From these expressions, it is easy to obtain that $\Vert F_{1}^{1}\Vert_{D(s,r)\times\Pi,L}\leq \epsilon^{C}e^{-\vert n-m\vert\rho},\Vert F_{1}^{2}\Vert_{D(s,r)\times\Pi,L}\leq (\vert m\vert+\vert n\vert)\epsilon^{C}e^{-\vert n-m\vert\rho},\Vert F_{1}^{3}\Vert_{D(s,r)\times\Pi}\leq(\vert m\vert+\vert n\vert)^{2}\epsilon^{C}e^{-\vert n-m\vert}$. Now we estimate the derivative of $F_{1}^{3}$ with respect to $\sigma$.
 $$
 \begin{aligned}
 \Vert \frac{\partial F_{1}^{3}}{\partial\sigma_{d}}\Vert\leq&\Vert\frac{\partial R_{1}^{3}}{\partial \sigma_{d}}\Vert\sum_{i=0}^{\infty}\Vert\frac{m+n+\bar{\lambda}}{2t}+\frac{C}{t^{2}}\Vert^{i}+\Vert\frac{\partial R_{1}^{2}}{\partial\sigma_{d}}\Vert\vert t\vert\sum_{i=1}^{\infty}\Vert\frac{m+n+\bar{\lambda}}{2t}+\frac{C}{t^{2}}\Vert^{i}\\
 &+\Vert\frac{\partial R_{1}^{1}}{\partial\sigma_{d}}\Vert\vert t\vert^{2}(\sum_{i=2}^{\infty}\Vert\frac{m+n+\bar{\lambda}}{2t}+\frac{C}{t^{2}}\Vert^{i}+\Vert \frac{C}{t^{2}}\Vert)\\
 &+\Vert R_{1}^{3}\Vert(\Vert\frac{\partial\bar{\lambda}}{\partial\sigma_{d}}\Vert\frac{1}{\vert t\vert}+\Vert\frac{\partial(\sigma_{m+t}-\sigma_{n+t})}{\partial\sigma_{d}}\Vert\frac{1}{\vert t\vert}+\Vert\frac{\partial(c_{m+t}-c_{n+t})}{\partial\sigma_{d}}\Vert\frac{1}{\vert t\vert^{2}}+\Vert\frac{\partial}{\partial\sigma_{d}}\frac{(\frac{m+n+\bar{\lambda}}{2t}+\frac{C}{t^{2}})^{2}}{1+\frac{m+n+\lambda}{2t}+\frac{C}{t^{2}}}\Vert)\\
 &+\Vert R_{1}^{2}\Vert(\Vert\frac{\partial\bar{\lambda}}{\partial\sigma_{d}}\Vert+\Vert\frac{\partial(\sigma_{m+t}-\sigma_{n+t})}{\partial\sigma_{d}}\Vert+\Vert\frac{\partial(c_{m+t}-c_{n+t})}{\partial\sigma_{d}}\Vert\frac{1}{\vert t\vert}+\Vert\frac{\partial}{\partial\sigma_{d}}\frac{(\frac{m+n+\bar{\lambda}}{2t}+\frac{C}{t^{2}})^{2}}{1+\frac{m+n+\lambda}{2t}+\frac{C}{t^{2}}}\Vert\vert t\vert)\\
 	&+\Vert R_{1}^{1}\Vert(\vert t\vert^{2}\Vert\frac{\partial}{\partial\sigma_{d}}\frac{(\frac{m+n+\bar{\lambda}}{2t}+\frac{C}{t^{2}})^{2}}{1+\frac{m+n+\lambda}{2t}+\frac{C}{t^{2}}}\Vert+\vert t\vert\Vert\frac{\partial(\sigma_{m+t}-\sigma_{n+t})}{\partial\sigma_{d}}\Vert+\Vert\frac{\partial(c_{m+t}-c_{n+t})}{\partial\sigma_{d}}\Vert)\\
 	\leq&\epsilon^{C}[(\vert m\vert+\vert n\vert)^{2}\frac{1}{\langle d\rangle}e^{-\vert n-m\vert\rho}+\vert t\vert e^{-\vert m+t-d\vert\rho}e^{-\vert n+t-d\vert\rho}]\\
 	&+\epsilon^{C}[(\vert m\vert+\vert n\vert)e^{-\vert n-m\vert\rho}\frac{1}{\langle d\rangle}+\vert t\vert e^{-\vert m+t-d\vert\rho}e^{-\vert n+t-d\vert\rho}+\vert t\vert e^{-\vert m+t+d\vert\rho}e^{-\vert n+t+d\vert\rho}]\\
 	\leq&\epsilon^{C}[(\vert m\vert+\vert n\vert)^{2}e^{-\vert n-m\vert\rho}\frac{1}{\langle d\rangle}+\vert t\vert e^{-\vert m+t-d\vert\rho}e^{-\vert n+t-d\vert\rho}+\vert t\vert e^{-\vert m+t+d\vert\rho}e^{-\vert n+t+d\vert\rho}].
\end{aligned}
$$
 Similarly, we can obtain the expansion of the solution of the second equation $F_{2}=F_{2}^{1}+F_{2}^{2}\frac{1}{t}$. From the estimate of $F_{2}$, by using Lemma \ref{118}, we have
 $$\Vert F_{3}\Vert_{D(\frac{s}{n^{Cn}},r)\times\Pi}\leq\epsilon^{C}e^{-\vert n-m\vert\rho},$$
 $$\Vert\frac{\partial F_{3}}{\partial\sigma_{d}}\Vert_{D(\frac{s}{n^{2Cn}},r)\times\Pi}\leq\epsilon^{C}(e^{-\vert n-m\vert\rho}\frac{1}{\langle d\rangle}+\vert t\vert e^{-\vert m+t-d\vert\rho}e^{-\vert n+t-d\vert\rho}+\vert t\vert e^{-\vert m+t+d\vert\rho}e^{-\vert n+t+d\vert\rho}).$$

 We now focus on the equation \eqref{903}:
 $$\partial_{\omega}F^{zz\bar{z}}_{ijl}-\mathrm{i}(\Omega_{i}+\Omega_{j}-\Omega_{l})F^{zz\bar{z}}_{ijl}=\tilde{R}^{zz\bar{z}}_{ijl}.$$
 Let $i=m+t\ne l=n+t$. When $|t|\ge\Lambda \max\lbrace |m|,|n|\rbrace,10j^{2}<|(m-n)t|$,
 $$\Omega_{j}+\Omega_{m+t}-\Omega_{n+t}=2(m-n)t+(m-n)\Xi_{j,m,n}+\frac{C}{t},$$
 where $\Xi_{j,m,n}:=m+n+\bar{\lambda}+\frac{j\bar{\lambda}+j^{2}+\sigma_{j}+\tilde{\lambda}+\hat{\lambda}_{j}}{m-n}$. Thus the preceding power series argument for the quadratic equation applies with $m+n+\bar{\lambda}$ replaced by $\Xi_{j,m,n}$. We obtain
 $$\frac{\partial^{3}F}{\partial z_{j}\partial z_{m+t}\partial\bar{z}_{n+t}}=(F_{11}^{j}+\frac{F_{12}^{j}}{t}+\frac{F_{13}^{j}}{t^{2}})+\frac{1}{t}(F_{21}^{j}+\frac{F_{22}^{j}}{t})+\frac{F_{3}^{j}}{t^{2}}.$$
 From the momentum conservation $\sum_{b\in J}k_{b}b-m+n-j=0$ and Cauchy's estimate, after a reduction of the strip, 
 \begin{equation}\label{906}
     \Vert F_{11}^{j}\Vert_{D(\frac{s}{n^{Cn}}-\sigma)}\le\frac{1}{r}\epsilon^{C} e^{a|j|}|j|^{p} e^{-\frac{\sigma}{Cn}|m-n+j|}e^{-\rho|m-n|}
 \end{equation}
 and analogous bounds hold for other coefficients. 

 The cubic contribution to the second derivative of $F$ is
 \begin{equation}\label{907}
     \begin{aligned}
    \frac{\partial^{2}F}{\partial z_{m+t}\partial\bar{z}_{n+t}}=&\sum_{10j^{2}<|(m-n)t|}[(F_{11}^{j}+\frac{F_{12}^{j}}{t}+\frac{F_{13}^{j}}{t^{2}})+\frac{1}{t}(F_{21}^{j}+\frac{F_{22}^{j}}{t})+\frac{F_{3}^{j}}{t^{2}}]z_{j}+\sum_{10j^{2}\ge|(m-n)t|}\frac{\partial^{3}F}{\partial z_{j}\partial z_{m+t}\partial\bar{z}_{n+t}}z_{j}\\
    =&\sum_{j}[(F_{11}^{j}+\frac{F_{12}^{j}}{t}+\frac{F_{13}^{j}}{t^{2}})+\frac{1}{t}(F_{21}^{j}+\frac{F_{22}^{j}}{t})+\frac{F_{3}^{j}}{t^{2}}]z_{j}\\
    &+\sum_{10j^{2}\ge|(m-n)t|}[(F_{11}^{j}+\frac{F_{12}^{j}}{t}+\frac{F_{13}^{j}}{t^{2}})
     +\frac{1}{t}(F_{21}^{j}+\frac{F_{22}^{j}}{t})+\frac{F_{3}^{j}}{t^{2}}]z_{j}\\
    &+\sum_{10j^{2}\ge|(m-n)t|}\frac{\partial^{3}F}{\partial z_{j}\partial z_{m+t}\partial\bar{z}_{n+t}}z_{j}\\
    =&:I+II+III
\end{aligned}
 \end{equation}
 Since $|z_{j}|\le r|j|^{-p}e^{-a|j|}$, the estimates for the coefficients include \eqref{906} imply absolute convergence of the sum in $I$. On the set occurring in $II$ and $III$, $e^{-\sigma|m-n+j|}\le C e^{-C\sigma\sqrt{|t|}}$ allows them to be absorbed into the remainder $\frac{F_{3}}{t^{2}}$.
 
 In addition, from \eqref{41}\eqref{42}\eqref{43} and the estimate for vector field $X_{P},X_{\partial_{\sigma_{d}}P}$ in (ii), we have
 \begin{equation}\label{45}
 	\Vert X_{F}\Vert_{r,a,p,D(\frac{s}{n^{3Cn}},r)\times\Pi}\leq\epsilon^{C},
 \end{equation}
 \begin{equation}\label{44}
 	\Vert X_{\partial_{\sigma_{d}}F}\Vert_{r,a,p,D(\frac{s}{n^{3Cn}},r)\times\Pi}\leq\epsilon^{C}\frac{1}{\langle d\rangle},\forall d\in\mathbb{Z}.
 \end{equation}
 
 Now we are at the position to prove the asymptotic estimate for the new perturbation. Given the expression of the new perturbation $P_{+}$ in \eqref{27}, we only need to prove that  $\lbrace P,F\rbrace\in FAE_{\Lambda+\delta,\epsilon^{1+C},\rho-\kappa}$ for some $\delta,\kappa>0$. Assume that $\vert t\vert\geq(\Lambda+\delta)\max\lbrace \vert m\vert,\vert n\vert\rbrace,t\neq 0$.
 $$\lbrace P,F\rbrace=\sum_{j\in J}(\frac{\partial P}{\partial x_{j}}\frac{\partial F}{\partial y_{j}}-\frac{\partial P}{\partial y_{j}}\frac{\partial F}{\partial x_{j}})-2\mathrm{i}\sum_{j\in\mathbb{Z}\setminus J}(\frac{\partial P}{\partial z_{j}}\frac{\partial F}{\partial\bar{z}_{j}}-\frac{\partial P}{\partial\bar{z}_{j}}\frac{\partial F}{\partial z_{j}}).$$
 $$
 \begin{aligned}
 \frac{\partial^{2}\lbrace P,F\rbrace}{\partial z_{m+t}\partial\bar{z}_{n+t}}=&[\sum_{j\in J}(\frac{\partial^{3}P}{\partial z_{m+t}\partial\bar{z}_{n+t}\partial x_{j}}\frac{\partial F}{\partial y_{j}}-\frac{\partial^{3} P}{\partial z_{m+t}\partial\bar{z}_{n+t}\partial y_{j}}\frac{\partial F}{\partial x_{j}}-\frac{\partial^{2}P}{\partial z_{m+t}\partial y_{j}}\frac{\partial^{2}F}{\partial\bar{z}_{n+t}\partial x_{j}}\\
 &-\frac{\partial^{2}P}{\partial\bar{z}_{n+t}\partial y_{j}}\frac{\partial^{2}F}{\partial z_{m+t}\partial x_{j}}-\frac{\partial P}{\partial y_{j}}\frac{\partial^{3} F}{\partial z_{m+t}\partial \bar{z}_{n+t}\partial x_{j}}\\
 &+\frac{\partial^{2}P}{\partial z_{m+t}\partial x_{j}}\frac{\partial^{2}F}{\partial\bar{z}_{n+t}\partial y_{j}}+\frac{\partial^{2}P}{\partial \bar{z}_{n+t}\partial x_{j}}\frac{\partial^{2}F}{\partial z_{m+t}\partial y_{j}})]\\
 &+[-2\mathrm{i}\sum_{j\in\mathbb{Z}\setminus J}(\frac{\partial^{3} P}{\partial z_{m+t}\partial\bar{z}_{n+t}\partial z_{j}}\frac{\partial F}{\partial\bar{z}_{j}}-\frac{\partial^{3} P}{\partial z_{m+t}\partial\bar{z}_{n+t}\partial\bar{z}_{j}}\frac{\partial F}{\partial z_{j}}\\
 &-\frac{\partial^{3} F}{\partial z_{m+t}\partial\bar{z}_{n+t}\partial z_{j}}\frac{\partial P}{\partial\bar{z}_{j}}+\frac{\partial^{3} F}{\partial z_{m+t}\partial\bar{z}_{n+t}\partial\bar{z}_{j}}\frac{\partial P}{\partial z_{j}})]\\
 &+[-2\mathrm{i}\sum_{j\in\mathbb{Z}\setminus J}(\frac{\partial^{2}P}{\partial z_{m+t}\partial z_{j}}\frac{\partial^{2} F}{\partial\bar{z}_{n+t}\partial\bar{z}_{j}}-\frac{\partial^{2}P}{\partial z_{m+t}\partial\bar{z}_{j}}\frac{\partial^{2}F}{\partial\bar{z}_{n+t}\partial z_{j}}\\
 &+\frac{\partial^{2}P}{\partial \bar{z}_{n+t}\partial z_{j}}\frac{\partial^{2} F}{\partial z_{m+t}\partial\bar{z}_{j}}-\frac{\partial^{2}P}{\partial \bar{z}_{n+t}\partial\bar{z}_{j}}\frac{\partial^{2}F}{\partial z_{m+t}\partial z_{j}})]\\
 &=:I+II+III.
 \end{aligned}
 $$
 
 We first consider II. Based on the expansion $\frac{\partial^{2}P}{\partial z_{m+t}\partial\bar{z}_{n+t}}=tP_{1}+P_{2}+P_{3}\frac{1}{t}$,
 $$\sum_{j\in\mathbb{Z}\setminus J}\frac{\partial^{3}P}{\partial z_{m+t}\bar{z}_{n+t}z_{j}}\frac{\partial F}{\partial\bar{z}_{j}}=\sum_{j\in\mathbb{Z}\setminus J}t\frac{\partial P_{1}}{\partial z_{j}}\frac{\partial F}{\partial\bar{z}_{j}}+\sum_{j\in\mathbb{Z}\setminus J}\frac{\partial P_{2}}{\partial z_{j}}\frac{\partial F}{\partial\bar{z}_{j}}+\sum_{j\in\mathbb{Z}\setminus J}\frac{1}{t}\frac{\partial P_{3}}{\partial z_{j}}\frac{\partial F}{\partial\bar{z}_{j}}.$$
 And owing to the Cauchy estimate, \eqref{45} and \eqref{44}, we have
 $$
 \begin{aligned}
 \Vert\sum_{j\in\mathbb{Z}\setminus J}\frac{\partial P_{i}}{\partial z_{j}}\frac{\partial F}{\partial\bar{z}_{j}}\Vert_{D(\frac{s}{n^{3Cn}},\frac{r}{2})\times\Pi}&\leq C\Vert\frac{\partial P_{i}}{\partial z}\Vert_{-a,-p,D(\frac{s}{n^{3Cn}},\frac{r}{2})\times\Pi}\Vert\frac{\partial F}{\partial z}\Vert_{a,p,D(\frac{s}{n^{3Cn}},\frac{r}{2})\times\Pi}\\
 &\leq C\Vert P_{i}\Vert_{D(\frac{s}{n^{3Cn}},r)\times\Pi}\Vert X_{F}\Vert_{r,a,p,D(\frac{s}{n^{3Cn}},\frac{r}{2})\times\Pi}\\
 &\leq\epsilon^{1+C}e^{-\vert n-m\vert\rho},i=1,2,3.
 \end{aligned}
 $$
 $$
 \begin{aligned}
 \Vert\frac{\partial}{\partial\sigma_{d}}\sum_{j\in\mathbb{Z}\setminus J}\frac{\partial P_{3}}{\partial z_{j}}\frac{\partial F}{\partial\bar{z}_{j}}\Vert_{D(\frac{s}{n^{3Cn}},\frac{r}{2})\times\Pi}\leq&\Vert\frac{\partial^{2} P_{3}}{\partial \sigma_{d}\partial z}\Vert_{-a,-p,D(\frac{s}{n^{3Cn}},\frac{r}{2})\times\Pi}\Vert\frac{\partial F}{\partial\bar{z}}\Vert_{a,p,D(\frac{s}{n^{3Cn}},\frac{r}{2})\times\Pi}\\
 &+\Vert\frac{\partial P_{3}}{\partial z}\Vert_{-a,-p,D(\frac{s}{n^{3Cn}},\frac{r}{2})\times\Pi}\Vert\frac{\partial^{2}F}{\partial\sigma_{d}\partial\bar{z}}\Vert_{a,p,D(\frac{s}{n^{3Cn}},\frac{r}{2})\times\Pi}\\
 \leq&C\Vert\frac{\partial P_{3}}{\partial\sigma_{d}}\Vert_{D(\frac{s}{n^{3Cn}},r)\times\Pi}\Vert X_{F}\Vert_{r,a,p,D(\frac{s}{n^{3Cn}},\frac{r}{2})\times\Pi}\\
 &+C\Vert P_{3}\Vert_{D(\frac{s}{n^{3Cn}},r)\times\Pi}\Vert X_{\frac{\partial F}{\partial\sigma_{d}}}\Vert_{r,a,p,D(\frac{s}{n^{3Cn}},\frac{r}{2})\times\Pi}\\
 \leq&\epsilon^{1+C}(\frac{1}{\langle d\rangle}e^{-\vert n-m\vert\rho}+\vert t\vert e^{-\vert m+t-d\vert\rho}e^{-\vert n+t-d\vert\rho}+\vert t\vert e^{-\vert m+t+d\vert\rho}e^{-\vert n+t+d\vert\rho}).
 \end{aligned}
 $$
 The estimates of the other terms in II follow by the same argument as above or in \eqref{907} and are therefore omitted. Now we focus on the cross-term III. To keep the explanation concise, we only take one of them as an example. The estimates for the other terms are completely similar.
 $$
 \begin{aligned}
 \sum_{j\in\mathbb{Z}\setminus J}\frac{\partial^{2}P}{\partial z_{m+t}\partial\bar{z}_{j}}\frac{\partial^{2}F}{\partial\bar{z}_{n+t}\partial z_{j}}&=\sum_{l\in(\mathbb{Z}\setminus J)-t}\frac{\partial^{2}P}{\partial z_{m+t}\partial\bar{z}_{l+t}}\frac{\partial^{2}F}{\partial\bar{z}_{n+t}\partial z_{l+t}}\\
 &=\sum_{\Lambda\vert l\vert\leq\vert t\vert}(P_{m+t\ l+t}^{1}t+P_{m+t\ l+t}^{2}+P_{m+t\ l+t}^{3}\frac{1}{t})(F_{l+t,n+t}^{1}+F_{l+t,n+t}^{2}\frac{1}{t}+F_{l+t,n+t}^{3}\frac{1}{t^{2}})\\
 &+\sum_{\Lambda\vert l\vert>\vert t\vert,l\in(\mathbb{Z}\setminus J)-t}\frac{\partial^{2}P}{\partial z_{m+t}\partial\bar{z}_{l+t}}\frac{\partial^{2}F}{\partial\bar{z}_{n+t}\partial z_{l+t}}\\
 &=\sum_{l\in\mathbb{Z}}(P_{m+t\ l+t}^{1}t+P_{m+t\ l+t}^{2}+P_{m+t\ l+t}^{3}\frac{1}{t})(F_{l+t,n+t}^{1}+F_{l+t,n+t}^{2}\frac{1}{t}+F_{l+t,n+t}^{3}\frac{1}{t^{2}})\\
 &-\sum_{\Lambda\vert l\vert>\vert t\vert}(P_{m+t\ l+t}^{1}t+P_{m+t\ l+t}^{2}+P_{m+t\ l+t}^{3}\frac{1}{t})(F_{l+t,n+t}^{1}+F_{l+t,n+t}^{2}\frac{1}{t}+F_{l+t,n+t}^{3}\frac{1}{t^{2}})\\
 &+\sum_{\Lambda\vert l\vert>\vert t\vert,l\in(\mathbb{Z}\setminus J)-t}\frac{\partial^{2}P}{\partial z_{m+t}\partial\bar{z}_{l+t}}\frac{\partial^{2}F}{\partial\bar{z}_{n+t}\partial z_{l+t}}\\
 &=(1)+(2)+(3).
 \end{aligned}
 $$
 Since only when $\vert t\vert\geq\Lambda\max\lbrace \vert m\vert,\vert l\vert\rbrace$ does the asymptotic expansion of $\frac{\partial^{2}P}{\partial z_{m+t}\partial\bar{z}_{l+t}}$ occurs, here we extend it to $\vert t\vert<\Lambda\vert l\vert$. For a fixed $l$, when $\vert t\vert$ is sufficiently large, $P_{m+t,l+t}^{11},P_{m+t,l+t}^{12},P_{m+t,l+t}^{21}$ are wl-defined and do not depend on $t$. We can directly perform the identity extension for them. And for $P_{m+t,l+t}^{13},P_{m+t,l+t}^{22},P_{m+t,l+t}^{3}$, we perform zero extension because they are the remainder. The extension of $F$ is completely similar.

 Based on the asymptotic expansions and estimates of $P$ and $F$,
 $$
 \begin{aligned}
 \vert(2)\vert+\vert (3)\vert&\leq\sum_{\Lambda\vert l\vert>\vert t\vert} \vert l\vert^{2}\frac{1}{\vert t\vert}\epsilon^{1+C}e^{-\vert m-l\vert\rho}e^{-\vert l-n\vert\rho}\\
 &\leq\epsilon^{1+C}e^{-\vert n-m\vert(\rho-\kappa)}e^{-\vert t\vert \kappa(\frac{1}{\Lambda}-\frac{1}{\Lambda+\delta})}\frac{1}{\vert t\vert}\sum_{\Lambda\vert l\vert>\vert t\vert} \vert l\vert^{2} e^{-\frac{\kappa}{2}\vert m-l\vert}e^{-\frac{\kappa}{2}\vert n-l\vert}\\
 &\leq\epsilon^{1+C}e^{-\vert n-m\vert(\rho-\kappa)}\frac{1}{\vert t\vert}.\\
 \vert\frac{\partial}{\partial\sigma_{d}}(2)\vert&\leq\sum_{\Lambda\vert l\vert>\vert t\vert}\vert\frac{\partial}{\partial\sigma_{d}}(P_{m+t\ l+t}^{1}t+P_{m+t\ l+t}^{2}+P_{m+t\ l+t}^{3}\frac{1}{t})\vert\vert(F_{l+t,n+t}^{1}+F_{l+t,n+t}^{2}\frac{1}{t}+F_{l+t,n+t}^{3}\frac{1}{t^{2}})\vert\\
 &+\sum_{\Lambda\vert l\vert>\vert t\vert}\vert(P_{m+t\ l+t}^{1}t+P_{m+t\ l+t}^{2}+P_{m+t\ l+t}^{3}\frac{1}{t})\vert\vert\frac{\partial}{\partial\sigma_{d}}(F_{l+t,n+t}^{1}+F_{l+t,n+t}^{2}\frac{1}{t}+F_{l+t,n+t}^{3}\frac{1}{t^{2}})\vert\\
 &\leq\epsilon^{1+C}\sum_{\Lambda\vert l\vert>\vert t\vert}\frac{\vert l\vert}{\vert t\vert}(\frac{\vert l\vert}{\langle d\rangle}e^{-\vert l-m\vert\rho}+e^{-\vert m+t-d\vert\rho} e^{-\vert l+t-d\vert\rho}+ e^{-\vert m+t+d\vert\rho}e^{-\vert l+t+d\vert\rho})e^{-\vert l-n\vert\rho}\\
 &+\epsilon^{1+C}\sum_{\Lambda\vert l\vert>\vert t\vert}\vert l\vert e^{-\vert l-m\vert\rho}[ e^{-\vert l-n\vert\rho}\frac{\vert l\vert}{\vert td\vert}+\frac{1}{\vert t\vert} e^{-\vert l+t+d\vert\rho}e^{-\vert n+t+d\vert\rho}+ \frac{1}{\vert t\vert}e^{-\vert n+t-d\vert\rho}e^{-\vert l+t-d\vert\rho}]\\
 &\leq\epsilon^{1+C}(e^{-\vert n-m\vert(\rho-\kappa)}\frac{1}{\vert td\vert}+e^{-\vert m+t-d\vert(\rho-\kappa)}e^{-\vert n+t-d\vert(\rho-\kappa)}+e^{-\vert m+t+d\vert(\rho-\kappa)}e^{-\vert n+t+d\vert(\rho-\kappa)}).
 \end{aligned}
 $$
 The estimate of the derivative of (3) is similar. Therefore, we can classify (2) and (3) as the remainder and group them under category $\lbrace P,F\rbrace_{3}\frac{1}{t}$. As for (1), we only need to categorize it according to the power of $t$: put $\sum_{l\in\mathbb{Z}}P^{1}_{m+t\ l+t}F^{1}_{l+t\ n+t}t$ into $\lbrace P,F\rbrace_{1}t$, put $\sum_{l\in\mathbb{Z}}(P_{m+t\ l+t}^{1}F_{l+t\ n+t}^{2}+P_{m+t\ l+t}^{2}F_{l+t\ n+t}^{1})$ into $\lbrace P,F\rbrace_{2}$ and place all the remaining items into $\lbrace P,F\rbrace_{3}\frac{1}{t}$. Since I is a finite sum and thus easier to estimate, we omit it.
 \section{Iterative Lemma}

Due to the unbounded nature of the vector field of perturbation $P$, the normal frequencies $\Omega=(\Omega_{j})_{j\in\mathbb{Z}\setminus J}$ is dependent on the angular variable $x$. Therefore, in order to obtain an invariant torus of higher dimension, we need to eliminate the variable coefficient part of the specific normal frequency before changing that into action-angle variables. To this end, we prove the following lemma.
\begin{lem}\label{51}
	Let $N=\sum_{j\in J}\omega_{j}(\sigma)y_{j}+\frac12\Omega_{t}(x,\sigma)z_{t}\bar{z}_{t}+\frac12\sum_{j\in\mathbb{Z}\setminus J\setminus\lbrace t\rbrace}\Omega_{j}z_{j}\bar{z}_{j}$ be a normal form Hamiltonian and $H=N+P$ be a small perturbation of it on $D(s,r)\times\Pi$, where $J=\lbrace i_{1},\cdots,i_{n}\rbrace\subset\mathbb{Z}$. Moreover, $\Omega_{t}$ has the decomposition $\Omega_{t}=\sum_{l=0}^{m}\Omega_{t}^{l}$ and $\Omega_{t}^{l}$ depends only on $x\in \mathbb{T}^{J_{l}}_{s}$ and $\sigma\in\Pi$, where $J_{0}\subset J_{1}\subset\cdots\subset J_{m}=J$. Suppose that
	
	(i)The following Diophantine conditions hold with constants $\tau_{l}\geq 2\sharp J_{l}+10$.
	$$\vert\langle k,\omega\vert_{J_{l}}\rangle\vert\geq\alpha_{l}\frac{1}{\langle k\rangle^{\tau_{l}}},0\neq k\in\mathbb{Z}^{J_{l}},l=0,1,\cdots,m.$$
	
	(ii)There are constants $\gamma_{l}=\gamma_{0}^{(\frac{12}{11})^{l}}(l=0,1,\cdots,m),n$ such that $\Vert\Omega_{t}^{l}\Vert_{s,\tau_{l}+10}\leq\alpha_{l}\gamma_{l} n$, where $\gamma_{0}$ is taken sufficiently small. In addition, the constants satisfy that $\sum_{l=0}^{m}\gamma_{l}sn=:\gamma sn\leq 1$.
	
	(iii)$\Vert X_{P}\Vert^{\lambda}_{r,a,p-1,D\times\Pi}\leq\epsilon\ll 1.$
	
	Then there exists a symplectic coordinate transformation $\Phi:D(s,r')\times\Pi\rightarrow D(s,r)$ which satisfies $\Vert \Phi-id\Vert^{\lambda}_{r,a,p,D(s,r')\times\Pi}\leq 2^{-n^{2}},\Vert D\Phi\Vert^{\lambda}_{r,r,D(s,r')\times \Pi},\Vert D\Phi^{-1}\Vert^{\lambda}_{r,r,D(s,r')\times \Pi}\leq e^{C\gamma n(s+r'/r)}$ such that $H_{+}=H\circ\Phi=N_{+}+P_{+}=\sum_{j\in J}\omega_{j}y_{j}+\frac12\bar{\Omega}_{t}z_{t}\bar{z}_{t}+\frac12\sum_{j\in\mathbb{Z}\setminus J\setminus\lbrace t\rbrace}\Omega_{j}z_{j}\bar{z}_{j}+P_{+}$ satisfies $\Vert X_{P_{+}}\Vert^{\lambda}_{r,a,p-1,D(s,r')\times\Pi}\leq \epsilon^{1-}$. Here $r'=(2^{-n^{2}})^{exp}r$.
\end{lem}
\begin{proof}
	We construct a function $F$ that solves the linear homological equation
	$$\lbrace N,F\rbrace=-\frac12\tilde{\Omega}_{t}z_{t}\bar{z}_{t}.$$
	Let $F=F(x,\sigma)z_{t}\bar{z}_{t}=\sum_{l=0}^{m}F^{l}(x,\sigma)z_{t}\bar{z}_{t}$. From the form of $\Omega_{t}$, the equation can naturally be decomposed into
	$$\lbrace N,F^{l}(x,\sigma)z_{t}\bar{z}_{t}\rbrace=-\frac12\tilde{\Omega}^{l}_{t}z_{t}\bar{z}_{t},l=0,1,\cdots,m.$$
	Equivalently, we solve $\partial_{\omega|_{J_l}}F^l=\frac12\widetilde\Omega_t^l$ with $[F^l]=0$. The sign is consistent with the direct normal-form substitution verified below.
	Due to the real nature of $F^{l}(x)$ on the $x$-axis, when the imaginary part of $x$ is relatively small, we can obtain the following estimate.
	$$\begin{aligned}
		\vert F^{l}(x)-F^{l}(Rex)\vert&=\vert\sum\hat{F}^{l}(k)(e^{\mathrm{i}kx}-e^{\mathrm{i}kRex})\vert\\&\leq\sum\vert\hat{F}^{l}(k)\vert\,\vert e^{-kImx}-1\vert\\
		&\leq\sum\frac{\vert\hat{\Omega}^{l}_{t}(k)\vert}{\vert\langle k,\omega\vert_{J_{l}}\rangle\vert}e^{\vert k\vert s}\vert k\vert s\\
		&\leq \gamma_{l} sn,	
	\end{aligned}
	$$
	\begin{equation}\label{133}
		\vert e^{\pm2\mathrm{i}F(x)}\vert=\vert e^{\pm2\mathrm{i}F(x)}e^{\mp2\mathrm{i}F(Rex)}\vert\leq e^{2\vert F(x)-F(Rex)\vert}\leq e^{2s\gamma n}.
	\end{equation}
	Consider the Hamiltonian canonical equations generated by $F$
	$$\left\{ \begin{array}{ll}
	\dot{x}=0,\\
	\dot{y}=-\partial_{x}F(x)z_{t}\bar{z}_{t},\\
	\dot{z}_{t}=-2\mathrm{i}F(x)z_{t},\\
	\dot{\bar{z}}_{t}=2\mathrm{i}F(x)\bar{z}_{t},\\
	\dot{z}_{j}=\dot{\bar{z}}_{j}=0,j\in\mathbb{Z}\setminus J\setminus\lbrace t\rbrace.
	\end{array}\right. $$
	From this, one readily deduces the Poincaré map
	$$ X_{F}^{1}=\begin{pmatrix}
	x(0)\\
	y(0)-\partial_{x} F(x(0))z_{t}(0)\bar{z}_{t}(0)\\
	z_{t}(0)e^{-2\mathrm{i} F(x(0))}\\
	\bar{z}_{t}(0)e^{2\mathrm{i} F(x(0))}\\
	z_{j}(0)\\
	\bar{z}_{j}(0)
	\end{pmatrix}.$$
	 Moreover, $x$ and all products $z_j\bar z_j$ are unchanged, and therefore
	$$
	N\circ X_F^1=\sum_{j\in J}\omega_jy_j+
	(\frac12\Omega_t-\partial_\omega F)z_t\bar z_t+
	\frac12\sum_{j\in\mathbb Z\setminus J\setminus\{t\}}\Omega_jz_j\bar z_j=N_+.
	$$
	Thus we take $\Phi=X_F^1$ and $P_+=P\circ\Phi$.
	
	 Based on the above estimate \eqref{133} of $F$, we obtain that $\vert z_{t}(1)\vert=\vert e^{-2\mathrm{i}F(x)}z_{t}(0)\vert\leq C\vert z_{t}(0)\vert$ and $\vert y(1)\vert\leq\vert y(0)\vert+\int_{0}^{1}\vert\partial_{x}F(x)z_{t}\bar{z}_{t}\vert d\tau\leq\vert y(0)\vert+ Cn\vert z_{t}(0)\bar{z}_{t}(0)\vert$. Then the flow $X_{F}^{1}$ of the vector field $X_{F}$ exists on $D(s,r')$ and takes this domain into $D(s,r)$. Differentiating this explicit map and its inverse, which is obtained by replacing $F$ by $-F$, gives diagonal phase factors bounded by $e^{2s\gamma n}$. The off-diagonal blocks in the scaled spatial norm are bounded by $C\gamma n(r'/r)$ and $C\gamma n(r'/r)^2$, with an additional phase factor in the normal components. Since $s\gamma n\leq1$, these contributions are absorbed by $e^{C\gamma n(s+r'/r)}$. The remaining conclusions can be easily derived from the classical KAM theory and we omit the proof.
\end{proof}
The preservation of mass, momentum, and the asymptotic expansion of frequencies under the symplectic transformation in Lemma \ref{51} is straightforward; its proof, being even simpler, is omitted.

Now we give the set-up of iteration parameters. Let $v\geq 0$ be the $v^{th}$ KAM step. The recursions below start from positive initial constants $s_0,r_0,\epsilon_0$, with $r_0$ and $\epsilon_0$ sufficiently small.

$J_{v}=\lbrace i\in \mathbb{Z}:\vert i\vert\leq v\rbrace.$
In other words, we excite two oscillators 
$$z_{j}=\sqrt{2(I_{j}+y_{j})}e^{-\mathrm{i}x_{j}},\bar{z}_{j}=\sqrt{2(I_{j}+y_{j})}e^{\mathrm{i}x_{j}}$$
$$z_{-j}=\sqrt{2(I_{-j}+y_{-j})}e^{-\mathrm{i}x_{-j}},\bar{z}_{-j}=\sqrt{2(I_{-j}+y_{-j})}e^{\mathrm{i}x_{-j}}$$
every iteration.

Fix sufficiently large constants $A_\alpha,A_\beta>0$ and sufficiently small constants $\alpha_0,\beta_0>0$, with $\beta_0\leq\min\{1,m_0\}/20$. Set
$$\alpha_i=\frac{\alpha_0}{((i+1)^{i+1})^{A_\alpha}},\quad i\geq0,\qquad
\alpha_v^i=\frac{\alpha_i}{10}(9+2^{-v}),\quad 0\leq i\leq v,$$
$$\beta_v=\frac{\beta_0}{((v+1)^{v+1})^{A_\beta}},\qquad \tau_v=10v+12,$$
which are used to dominate the measure of removed parameters. In particular, $\alpha_0^0=\alpha_0$ and $\tau_v\geq2\sharp J_v+10$.

$m_{v}=\frac{m_{0}}{10}(9+2^{-v})$, which is used for describing the growth of external frequencies. Thus $m_v\geq9m_0/10$ and $\beta_v\leq\beta_0\leq m_v/10$.

$E_{v}=E_{0}\langle v\rangle^{\frac{5}{2}}$, which is used to dominate the norm of internal frequencies,

$M_{1,v}=M_{1,0}\langle v\rangle^{2},M_{2,v}=M_{2,0}(2-2^{-v}),M_{v}=M_{1,v}+M_{2,v}$, which are used to dominate the Lipschitz semi-norm of frequencies.

$L_{v}=L_{0}$, which is used to dominate the inverse Lipschitz semi-norm of internal frequencies.

$s_{v}=\frac{s_{0}}{(v+1)^{20(v+1)^{2}}}$, which dominates the width of the angle variable $x$. In particular, $s_1=2^{-80}s_0<s_0$ and $s_{v+1}<s_v$ for all $v\geq0$,

$\sigma_{v}=\frac{s_{v+1}}{20}$, which serves as a bridge from $s_{v}$ to $s_{v+1}$,

$B_{v}=c_{v}\sigma_{v}^{-C(v+1)}$, here $c_{v}=(\tau_{v}^{\tau_{v}})^{exp}$,

$\epsilon_{v+1}=\epsilon_{v}r_{v}^{\frac{1}{10}}B_{v},\varepsilon_{v}=r_{v}\epsilon_{0}^{(\frac{12}{11})^{v}}$, which dominates the size of the perturbation $P_{v}$ in the $v^{th}$ KAM iteration,

$\gamma_{v}=\epsilon_{v}^{\frac{1}{2}}$, which is used to dominate the norm of external frequencies,

$\iota_{v}=2^{v^{2}}$, which is used for the estimate of measure,

$\Lambda_{v}=10v+10,\rho_{v}=\rho_{0}(1-\sum_{i=1}^{v}\frac{1}{10i^{2}})$, which are the parameters in the asymptotic estimate condition. The empty sum is zero, and $\rho_v\geq(1-\pi^2/60)\rho_0>0$,

$K_{1,v}=|\ln\epsilon_{v}|,K_{2,v}=\frac{|\ln \epsilon_{v}|}{10s_{v}}$, which is the  length of the truncation.

$r_{v+1}=r_{v}^{\frac{11}{10}}, D_{v}=D(s_{v},r_{v}),\lambda_{v}=\frac{\beta_{v}}{M_{v}}$.

And then we are at the position to prove the iterative lemma.

\begin{lem}
	Let $N_{v}=\sum_{j\in J_{v}}\omega_{v,j}(\sigma)y_{j}+\frac12\sum_{j\in\mathbb{Z}\setminus J_{v}}\Omega_{v,j}(\sigma)z_{j}\bar{z}_{j}$ be a normal form Hamiltonian and $H_{v}=N_{v}+P_{v}$ be a small perturbation of it on $D(s_{v},r_{v})\times \Pi_{v}$, $\Omega_{j}(j\in\mathbb{Z}\setminus J_{v})$ have the decomposition $\Omega_{v,j}=\sum_{i=0}^{v}\Omega_{v,j}^{i}$ and $\Omega_{v,j}^{i}$ depends only on $x\in \mathbb{T}^{J_{i}}_{s_{i}}$ and $\sigma\in\Pi_{v}$. Suppose that \\
	(i)$\lambda_{v}=((\omega_{v,j})_{j\in J_{v}},(\Omega_{v,j})_{j\in\mathbb{Z}\setminus J_{v}})$ satisfies
	\begin{equation}\label{52}
		\lambda_{v,n}=\vert n\vert^{2}+\sigma_{n}+n\bar{\lambda}_{v}+\tilde{\lambda}_{v}+\hat{\lambda}_{v,n},n\in\mathbb{Z}
	\end{equation}
	with estimates for $\forall n,l\in\mathbb{Z}$
	$$\Vert \hat{\lambda}_{v,n}\Vert_{D_{v}\times\Pi}\leq\sum_{i=0}^{v}\epsilon_{i}\langle n\rangle^{-1},\Vert\frac{\partial}{\partial \sigma_{l}} \hat{\lambda}_{v,n}\Vert_{D_{v}\times\Pi}\leq\sum_{i=0}^{v}\epsilon_{i}(\langle n\rangle^{-1}\langle l\rangle^{-1}+e^{-2\vert n-l\vert\rho_{i}}+e^{-2\vert n+l\vert\rho_{i}}),$$
	$$\Vert \bar{\lambda}_{v}\Vert_{D_{v}\times\Pi,L},\Vert\tilde{\lambda}_{v}\Vert_{D_{v}\times\Pi,L}\leq\sum_{i=0}^{v}\epsilon_{i}.$$
	(ii)$$\vert \langle k,\omega_{v}\vert_{J_{i}}\rangle\vert\geq\alpha_{v}^{i}\frac{1}{\langle k\rangle^{\tau_{i}}},0\neq k\in\mathbb{Z}^{J_{i}}$$
	
	$$\vert\langle k,\omega_{v}\rangle+\langle l,\bar{\Omega}_{v}\rangle\vert\geq\beta_{v}\frac{\langle l\rangle_{2}}{\langle k\rangle^{\tau_{v}}},k\neq 0,0<\vert l\vert\leq 2\ or\ l\in I_{3},$$
	$$\vert\langle l,\bar{\Omega}_{v}\rangle\vert\geq m_{v}\langle l\rangle_{2},0<\vert l\vert\leq 2,$$
	$$\Vert\tilde{\Omega}_{v,j}^{i}\Vert_{s_{i},\tau_{i}+10}\leq \alpha_{v}^{i}\gamma_{i}\langle j\rangle,$$
	$$\vert\omega_{v}\vert_{\Pi_{v}}\leq E_{v},\vert \omega_{v}\vert^{lip}_{\Pi_{v}}\leq M_{1,v},\vert\omega_{v}'^{-1}\vert^{lip}_{\omega_{v}'(\Pi_{v}^{J_{v}})\times\Pi'_{v}}\leq L_{v},\Vert\Omega_{v}\Vert^{lip}_{-1,\Pi_{v}}\leq M_{2,v},$$
    where $I_{3}:=\lbrace l:|l|=3\rbrace\setminus\lbrace l=\pm(e_{i}+e_{j}-e_{l}):j=l,|j|>10|i|\rbrace\setminus\lbrace l=\pm(e_{i}+e_{j}-e_{l}):|i+j-l|>K_{2,v}\rbrace$ and $\omega_{v,j}':=j^{2}+\sigma_{j}+\hat{\omega}_{v,j}$.
	
	(iii)$P_{v}$ satisfies momentum conservation, mass conservation, $P_{v}\in FAE_{\Lambda_{v},\epsilon_{v},\rho_{v}}, \Vert X_{P_{v}}\Vert_{r_{v},a,p-1,D_{v}\times\Pi}^{\lambda_{v}}\leq\epsilon_{v}$ and $\Vert X_{\partial_{\sigma_{d}}P_{v}}\Vert_{r_{v},a,p-1,D_{v}\times\Pi}\leq\epsilon_{v}\frac{1}{\langle d\rangle}$ for any $d\in\mathbb{Z}$. Moreover, there is a decomposition $P_{v}=P_{v,1}+P_{v,2}$ with $P_{v,2}$ is a cubic term whose indices lie in $\lbrace l=\pm(e_{i}+e_{j}-e_{l}):j=l,|j|>10|i|\rbrace$ and $\Vert X_{P_{v,1}}\Vert_{r_{v},a,p-1,D_{v}\times\Pi_{v}}\le \varepsilon_{v}$.\\
	Then there exists a symplectic coordinate transformation $\Phi_{v+1}:D(s_{v+1},r_{v+1})\times \Pi_{v}\rightarrow D(s_{v},r_{v})$ and a subset $$\Pi_{v+1}=\Pi_{v}\setminus\cup_{0\leq i\leq v,k\in\mathbb{Z}^{J_{i}},\vert k\vert>\iota_{v},l=0} \tilde{R}^{v+1}_{kl}(\alpha_{v+1}^{i},\tau_{i})\setminus\cup_{k\in\mathbb{Z}^{J_{v}}, 0\leq\vert l\vert\leq 2,l\in I_{3}}\tilde{R}^{v+1}_{kl}(\beta_{v+1},\tau_{v+1}),$$
	such that for $H_{v+1}=H_{v}\circ\Phi_{v+1}=N_{v+1}+P_{v+1}=\sum_{j\in J_{v+1}} \omega_{v+1,j}(\sigma)y_{j}+\frac12\sum_{j\in\mathbb{Z}\setminus J_{v+1}}\Omega_{v+1,j}(\sigma)z_{j}\bar{z}_{j}+P_{v+1}$, 
	
	(i)the same assumptions as above are satisfied with `$v+1$' in place of `$v$'.
	
	(ii)For $\forall \sigma\in\Pi_{v}$, define its truncation $\sigma_{k,\alpha,\tau}$ whose $jth$ element equals to that of $\sigma$ if $\vert j\vert <\frac{\langle k\rangle^{2\tau+2}}{\alpha^{2}}$ and $0$ otherwise. Then $R^{v+1}_{kl}(\alpha,\tau)\subset\tilde{R}^{v+1}_{kl}(\alpha,\tau)$, where
	$$R^{v+1}_{kl}(\alpha,\tau)=\lbrace \sigma\in \Pi_{v}:\vert\langle k,\omega_{v+1}(\sigma)\rangle+\langle l,\bar{\Omega}_{v+1}(\sigma)\rangle\vert<\alpha\frac{\langle l\rangle_{2}}{\langle k\rangle^{\tau}}\rbrace,$$
	$$\tilde{R}^{v+1}_{kl}(\alpha,\tau)=\lbrace \sigma\in \Pi_{v}:\vert\langle k,\omega_{v+1}(\sigma_{k,\alpha,\tau})\rangle+\langle l,\bar{\Omega}_{v+1}(\sigma_{k,\alpha,\tau})\rangle\vert<2\alpha\frac{\langle l\rangle_{2}}{\langle k\rangle^{\tau}}\rbrace.$$
\end{lem}
\begin{proof}
	We first apply Lemma \ref{51} twice, setting $t$ equals to $v+1$ and $-v-1$ respectively. Then we excite two oscillators
	$$z_{v+1}=\sqrt{2(I_{v+1}+y_{v+1})}e^{-\mathrm{i}x_{v+1}},\bar{z}_{v+1}=\sqrt{2(I_{v+1}+y_{v+1})}e^{\mathrm{i}x_{v+1}},$$
	$$z_{-v-1}=\sqrt{2(I_{-v-1}+y_{-v-1})}e^{-\mathrm{i}x_{-v-1}},\bar{z}_{-v-1}=\sqrt{2(I_{-v-1}+y_{-v-1})}e^{\mathrm{i}x_{-v-1}}$$ 
	and make one KAM step.
	
	The assumption (iii) can be directly derived from the proof in the previous section. Now we focus on the asymptotic estimate of the new frequency $\lambda_{v+1}$. In fact, we only need to consider the expression of $\Omega'_{j}:=\Omega_{v+1,j}-\Omega_{v,j}=2R^{z\bar{z}}_{jj}+\langle \partial_{x}\Omega_{j},F^{y}\rangle$ when $\vert j\vert$ is sufficiently large due to the smallness of $\epsilon_{v}$. Let $m=n=1,t=j-1$, we obtain from the asymptotic estimate of $P_{v}$ that when $\vert j\vert$ is sufficiently large such that $\vert j\vert\geq \Lambda_{v}+1$, we have that
	$$R^{z\bar{z}}_{jj}=ta+b+\frac{c}{t},$$
	where $a,b$ can be expanded as Taylor series with respect to $\frac{1}{t}$:
	$$a=a_{1}+a_{2}\frac{1}{t}+a_{3}\frac{1}{t^{2}},$$
	$$b=b_{1}+b_{2}\frac{1}{t}.$$
	At the same time, the coefficients satisfy that 
	$$
	\begin{aligned}
	&\Vert a_{1}\Vert_{D(s_{v},r_{v})\times\Pi_{v},L},\Vert a_{2}\Vert_{D(s_{v},r_{v})\times\Pi_{v},L},\Vert b_{1}\Vert_{D(s_{v},r_{v})\times\Pi_{v},L},\\
	&\Vert c\Vert_{D(s_{v},r_{v})\times\Pi_{v}},\Vert a_{3}\Vert_{D(s_{v},r_{v})\times\Pi_{v}},\Vert b_{2}\Vert_{D(s_{v},r_{v})\times\Pi_{v}}\leq\epsilon_{v}
	\end{aligned}
	$$
	 and
	 $$
	 \begin{aligned}
	 	\sup_{D(s_{v},r_{v})\times\Pi_{v}}\vert\frac{\partial c}{\partial\sigma_{l}}\vert,\sup_{D(s_{v},r_{v})\times\Pi_{v}}\vert\frac{\partial a_{3}}{\partial\sigma_{l}}\vert,\sup_{D(s_{v},r_{v})\times\Pi_{v}}\vert\frac{\partial b_{2}}{\partial\sigma_{l}}\vert\leq\epsilon_{v}(\frac{1}{\vert l\vert}+\langle j\rangle e^{-2\vert j-l\vert\rho_{v}}+\langle j\rangle e^{-2\vert j+l\vert\rho_{v}})
	 \end{aligned}
	 $$  
	 for any $l\in\mathbb{Z}$. And $a_{1},a_{2},b_{1}$ are independent of $j$.
	
	Let $\bar{\lambda}_{v+1}=\bar{\lambda}_{v}+2a_{1},\tilde{\lambda}_{v+1}=\tilde{\lambda}_{v}+2(-a_{1}+a_{2}+b_{1}),\hat{\lambda}_{v+1,j}=\hat{\lambda}_{v,j}+\frac{2a_{3}}{j-1}+\frac{2b_{2}}{j-1}+\frac{2c}{j-1}$. Then (i) is obviously valid with `$v+1$' in place of `$v$'.

	As to the diophantine conditions, in view of the definition of $\Pi_{v+1}$ and the inclusion relation $R^{v+1}_{kl}\subset\tilde{R}^{v+1}_{kl}$, only the case $\langle k\rangle\leq\iota_{v}$ remains to verify. In this case, we have
	$$
	\begin{aligned}
		\vert\langle k,(\omega_{v+1}-\omega_{v})\vert_{J_{i}}\rangle\vert&\leq\vert k\vert\vert\omega_{v+1}-\omega_{v}\vert\\
		&\leq\vert k\vert B_{v}\epsilon_{v}\\
		&\leq(\alpha_{v}^{i}-\alpha_{v+1}^{i})\frac{\vert k\vert}{\iota_{v}^{\tau_{i}+1}}\\
		&\leq (\alpha_{v}^{i}-\alpha_{v+1}^{i})\frac{1}{\langle k\rangle^{\tau_{i}}}
	\end{aligned}
	$$
	on $D_{v+1}\times\Pi_{v}$ for $0\neq k\in\mathbb{Z}^{J_{i}}$ and $\langle k\rangle\leq \iota_{v}$. 
	
	A detailed justification and verification of the inclusion relation $R^{v+1}_{kl}\subset\tilde{R}^{v+1}_{kl}$ will be addressed in the following section.
	\end{proof}
\section{measure estimate}
Since the Diophantine condition is insensitive to variations in $\sigma_{j}$ with $\vert j\vert$ sufficiently large—owing to the asymptotic estimate of frequencies and the condition $\sigma\in l^{2}$—we can fix a specific value for $\sigma_{j}$ and treat all other values as perturbative in nature. The set $\tilde{R}^{v+1}_{kl}$ thus defined is a cylinder set in the infinite-dimensional product measure space we choose (see Appendix \ref{117} for details). Crucially, the measure of $\tilde{R}^{v+1}_{kl}$ reduces to that of a finite-dimensional product measure, which can be computed directly.
\begin{lem}
	The inclusion relation $R^{v+1}_{kl}(\alpha,\tau)\subset\tilde{R}^{v+1}_{kl}(\alpha,\tau)$ holds.
\end{lem}
\begin{proof}
Let $f(\sigma)=\langle k,\omega_{v+1}(\sigma)\rangle+\langle l,\bar{\Omega}_{v+1}(\sigma)\rangle,g(\sigma)=\langle k,\omega'_{v+1}(\sigma)\rangle+\langle l,\bar{\Omega}'_{v+1}(\sigma)\rangle$, where $\omega'_{v+1,j}=\vert j\vert^{2}+\sigma_{j}+\hat{\omega}_{v+1,j},\bar{\Omega}'_{v+1,j}=\vert j\vert^{2}+\sigma_{j}+\bar{\hat{\Omega}}_{v+1,j}$. It follows from the momentum conservation, mass conservation and asymptotic expansion of $\lambda_{v+1}$ that
$$
\begin{aligned}
\vert f(\sigma)-f(\sigma_{k,\alpha,\tau})\vert&=\vert g(\sigma)-g(\sigma_{k,\alpha,\tau})\vert\\
&\leq \langle k\rangle \sup_{l}\sum_{\vert j\vert\geq\frac{\langle k\rangle^{2\tau+2}}{\alpha^{2}}}\vert\frac{\partial\omega'_{v+1,l}}{\partial\sigma_{j}}\vert\frac{1}{\vert j\vert}+\sup_{l}\sum_{\vert j\vert\geq\frac{\langle k\rangle^{2\tau+2}}{\alpha^{2}}}\vert\frac{\partial\bar{\Omega}'_{v+1,l}}{\partial\sigma_{j}}\vert\frac{1}{\vert j\vert}\\
&\leq C\langle k\rangle(\sum_{\vert j\vert\geq\frac{\langle k\rangle^{2\tau+2}}{\alpha^{2}}}\frac{1}{\vert j\vert^{2}})^{\frac{1}{2}}\\
&\leq C\frac{\alpha}{\langle k\rangle^{\tau}}.
\end{aligned}
$$
From this we have reached that $R^{v+1}_{kl}(\alpha,\tau)\subset\tilde{R}^{v+1}_{kl}(\alpha,\tau)$.
\end{proof}
We let $\Pi_{v}^{R,1}=\cup_{0\leq i\leq v,k\in\mathbb{Z}^{J_{i}},\vert k\vert>\iota_{v},l=0} \tilde{R}^{v}_{kl}(\alpha_{v}^{i},\tau_{i}),\Pi_{v}^{R,2}=\cup_{k\in\mathbb{Z}^{J_{v}}, 0\leq \vert l\vert\leq 2,l\in I_{3}}\tilde{R}^{v}_{kl}(\beta_{v},\tau_{v})$. Then the following measure estimate hold.
\begin{lem}
	$\mathbb{P}(\Pi_{v}^{R,1})\leq\frac{1}{\iota_{v}}\sum_{i=0}^{v}((i+1)^{i+1})^{exp}\alpha^{i}_{v},\mathbb{P}(\Pi_{v}^{R,2})\leq((v+1)^{v+1})^{exp}\beta_{v}.$
\end{lem}
\begin{proof}
	We only give the proof of the most difficult case that $l$ has two non-zero components of opposite sign. In this case, rewriting $R^{v}_{kl}$ as
	$$R^{v}_{kij}=\lbrace \sigma\in\Pi_{v}:\vert\langle k,\omega_{v}(\sigma_{k,\beta_{v},\tau_{v}})\rangle+\bar{\Omega}_{v,i}(\sigma_{k,\beta_{v},\tau_{v}})-\bar{\Omega}_{v,j}(\sigma_{k,\beta_{v},\tau_{v}})\vert<\beta_{v}\frac{\vert i^{2}- j^{2}\vert}{\langle k\rangle^{\tau_{v}}}\rbrace,i,j\in\mathbb{Z}\setminus J_{v},i\neq\pm j,$$
		$$R^{v}_{k(-j)j}=\lbrace \sigma\in\Pi_{v}:\vert\langle k,\omega_{v}(\sigma_{k,\beta_{v},\tau_{v}})\rangle+\bar{\Omega}_{v,-j}(\sigma_{k,\beta_{v},\tau_{v}})-\bar{\Omega}_{v,j}(\sigma_{k,\beta_{v},\tau_{v}})\vert<\beta_{v}\frac{1}{\langle k\rangle^{\tau_{v}}}\rbrace,-j,j\in\mathbb{Z}\setminus J_{v},$$
		we only need to  estimate the measure of 
		$$\Theta^{v}:=\cup_{i\neq j}R^{v}_{kij}.$$
		By the momentum conservation, mass conservation and the asymptotic estimate \eqref{52}, we have
		$$R^{v}_{kij}=\lbrace \sigma\in\Pi_{v}:\vert\langle k,\omega'_{v}(\sigma_{k,\beta_{v},\tau_{v}})\rangle+\bar{\Omega}'_{v,i}(\sigma_{k,\beta_{v},\tau_{v}})-\bar{\Omega}'_{v,j}(\sigma_{k,\beta_{v},\tau_{v}})\vert<\beta_{v}\frac{\vert i^{2}- j^{2}\vert}{\langle k\rangle^{\tau_{v}}}\rbrace,i,j\in\mathbb{Z}\setminus J_{v},i\neq\pm j,$$
		$$R^{v}_{k(-j)j}=\lbrace \sigma\in\Pi_{v}:\vert\langle k,\omega'_{v}(\sigma_{k,\beta_{v},\tau_{v}})\rangle+\bar{\Omega}'_{v,-j}(\sigma_{k,\beta_{v},\tau_{v}})-\bar{\Omega}'_{v,j}(\sigma_{k,\beta_{v},\tau_{v}})\vert<\beta_{v}\frac{1}{\langle k\rangle^{\tau_{v}}}\rbrace,-j,j\in\mathbb{Z}\setminus J_{v},$$
		where 
		$$\omega'_{v,j}=\vert j\vert^{2}+\sigma_{j}+\hat{\omega}_{v,j},j\in J_{v},$$
		$$\bar{\Omega}'_{v,j}=\vert j\vert^{2}+\sigma_{j}+\bar{\hat{\Omega}}_{v,j},j\in\mathbb{Z}\setminus J_{v}.$$
		
		Since these sets are all cylinder sets in the infinite dimensional product measure space, we only need to estimate their measures in the finite dimensional space. Let $\tilde{\Pi}_{v}:=\lbrace \sigma_{k,\beta_{v},\tau_{v}}:\sigma\in\Pi_{v}\rbrace,\tilde{R}^{v}_{kij}:=\lbrace\sigma_{k,\beta_{v},\tau_{v}}:\sigma\in R^{v}_{kij}\rbrace$. We introduce the perturbed frequencies 
		$$\zeta_{j}=\left\{ \begin{array}{ll}
		\omega'_{v,j}(\sigma_{k,\beta_{v},\tau_{v}}),j\in J_{v},\\
		\sigma_{j},\vert j\vert\leq\frac{\langle k\rangle^{2\tau_{v}+2}}{\beta_{v}^{2}},j\in\mathbb{Z}\setminus J_{v}.
		\end{array}\right. $$
		as parameters over the domain $Z:=\omega'_{v}(\tilde{\Pi}_{v})\times\prod_{\vert j\vert\leq\frac{\langle k\rangle^{2\tau_{v}+2}}{\beta_{v}^{2}},j\in\mathbb{Z}\setminus J_{v}}[0,\frac{1}{\langle j\rangle}]$ and consider the resonance zones $\dot{R}^{v}_{kij}=\zeta(\tilde{R}^{v}_{kij})$ in $Z$. Regarding $\bar{\Omega}'_{v}$ as a function of $\zeta$, then from the iterative lemma above, we know
		$$\vert\zeta_{J_{v}}\vert\leq E_{v}\leq C\langle v\rangle^{\frac{5}{2}},\vert\sigma_{J_{v}}\vert^{lip}_{Z}\leq L_{v}\leq C,$$
		$$\vert\bar{\Omega}'_{v,i}-\bar{\Omega}'_{v,j}\vert_{Z}\geq m_{v}\vert i^{2}-j^{2}\vert\geq\frac{9}{10}m_{0}\vert i^{2}-j^{2}\vert.$$
		
		Now we consider a fixed $\dot{R}^{v}_{kij}$.
		
		Case 1:$\vert k\vert<\frac{9m_{v}}{10E_{v}}\vert i^{2}-j^{2}\vert$, we get $\vert \langle k,\zeta_{J_{v}}\rangle\vert<\frac{9m_{v}}{10}\vert i^{2}-j^{2}\vert$. In view of $\beta_{v}\leq\frac{m_{v}}{10}$, we know $\dot{R}^{v}_{kij}$ is empty.
		
		Case 2:$0<\vert i^{2}-j^{2}\vert\leq\frac{10E_{v}}{9m_{v}}\vert k\vert$. Replacing $(k,i,j)$ by $(-k,j,i)$ if necessary, choose $b\in J_v$ with $k_b>0$. Fix $w_1=e_b$ and write $\zeta=aw_{1}+w_{2}$ with $w_{1}\perp w_{2}$. As a function of $a$, for $t>s$, 
		$$\langle k,\zeta_{J_{v}}\rangle\vert^{t}_{s}=k_b(t-s),$$
		$$(\vert\bar{\Omega}'_{v,i}-\bar{\Omega}'_{v,j}\vert)\vert^{t}_{s}\leq\epsilon_{0}(t-s).$$
		Thus
		$$(\langle k,\zeta_{J_{v}}\rangle+\bar{\Omega}'_{v,i}-\bar{\Omega}'_{v,j})\vert^{t}_{s}\geq\frac{1}{2}k_b(t-s).$$
		The tangential block has dimension $\sharp J_v=2v+1$. For each fixed normal parameter, its section perpendicular to $w_1$ has dimension $2v$. Integrating also over the normal parameters, we get
\[
\begin{aligned}
\operatorname{mea}(\dot{R}^{v}_{kij})&\leq C(\operatorname{diam}\,\omega'_v(\tilde\Pi_v))^{2v}\beta_{v}\frac{\vert i^{2}-j^{2}\vert}{\langle k\rangle^{\tau_{v}}}\prod_{\langle j\rangle\leq\frac{\langle k\rangle^{2\tau_{v}+2}}{\beta_{v}^{2}},j\in\mathbb{Z}\setminus J_{v}}\frac{1}{\langle j\rangle}\\
&\leq C((v+1)^{v+1})^{exp}\frac{\beta_{v}}{\langle k\rangle^{\tau_{v}-1}}\prod_{\langle j\rangle\leq\frac{\langle k\rangle^{2\tau_{v}+2}}{\beta_{v}^{2}},j\in\mathbb{Z}\setminus J_{v}}\frac{1}{\langle j\rangle}.
\end{aligned}
\]
		Going back to the original parameter domain $\Pi_{v}$ by the inverse frequency map $\zeta^{-1}$, we get 
		$$\mathbb{P}(R^{v}_{kij})\leq C((v+1)^{v+1})^{exp}\frac{\beta_{v}}{\langle k\rangle^{\tau_{v}-1}}.$$
		
		Case 3:$i=-j$. By an argument completely analogous to that of Case 2, we obtain
		$$\mathbb{P}(R^{v}_{kij})\leq C((v+1)^{v+1})^{exp}\frac{\beta_{v}}{\langle k\rangle^{\tau_{v}}}.$$
		
		For any fixed $k\in\mathbb{Z}^{J_v}\setminus\lbrace 0\rbrace$, we discuss the number of nonempty $R^{v}_{kij},i,j\in\mathbb{Z}\setminus J_{v},i\neq j,\sum_{b\in J_{v}}k_{b}b+i-j=0$.
		
		In view of Case 1 and Case 2, such that $R^{v}_{kij}$ is nonempty, the $i,j$ with $i\neq\pm j$ must satisfy
		$$\vert i\vert+\vert j\vert\leq 2\vert i^{2}-j^{2}\vert\leq CE_{v}\vert k\vert.$$
		Hence the number is no more than $C(v+1)^{exp}\vert k\vert^{2}$.
		
		In view of Case 3, if $\pm\frac{1}{2}\sum_{b\in J_{v}}k_{b}b\in\mathbb{Z}\setminus J_{v}$, then $j$ is uniquely determined by $\sum_{b\in J_{v}}k_{b}b+2j=0$; otherwise, there is no such $R^{v}_{k(-j)j}$ at all. Hence the number is at most 1. We finally reach the conclusion by summing up $k$.
\end{proof}
\section{proof of the main theorem}\label{sec:main-proof}
\subsection{Convergence and almost periodic dynamics}\label{sec:convergence}
In this section, we prove the main theorem \ref{53}. Considering the Hamiltonian system  \eqref{54}\eqref{55}, we introduce action-angle variables by setting
$$\left\{ \begin{array}{ll}
q_{0}=\sqrt{2(I_{0}+y_{0})}e^{-\mathrm{i}x_{0}},\\
\bar{q}_{0}=\sqrt{2(I_{0}+y_{0})}e^{\mathrm{i}x_{0}},\\
q_{j}=z_{j},j\neq0,\\
\bar{q}_{j}=\bar{z}_{j},j\neq0,
\end{array}\right. $$
where $I_{0}$ is fixed. Therefore, up to a constant term, the Hamiltonian can be written as
$$H=N+P=\sum_{j\in J_{0}}\omega_{j}y_{j}+\frac12\sum_{j\in\mathbb{Z}\setminus J_{0}}\Omega_{j}z_{j}\bar{z}_{j}+G(q(x,y,z,\bar{z})),$$
where 
$$\omega_{j}=j^{2}+\sigma_{j},j\in J_{0},$$
$$\Omega_{j}=j^{2}+\sigma_{j},j\in\mathbb{Z}\setminus J_{0}.$$

To apply the iterative lemma with $v=0$, set $N_{0}=N,P_{0}=P,\Pi_0=\prod_{j\in\mathbb{Z}}[0,\frac{1}{\langle j\rangle}]\setminus(\Pi_{0}^{R,1}\cup\Pi_{0}^{R,2})$. Then it is obvious that assumptions (i) and (ii) hold true, and $P_{0}$ satisfies momentum conservation and mass conservation. In the same way as that of Lemma 7.1 in \cite{LY11}, we have that the Hamiltonian vector field $X_{G}$ is an analytic map from some neighborhood of the origin in $l^{a,p}_{\mathbb{C}}$ into $l^{a,p-1}_{\mathbb{C}}$ with $\Vert X_{G}\Vert_{a,p-1}=O(\Vert q\Vert_{a,p}^{3})$. Then we know there exists $r>0$ such that the Hamiltonian vector field $X_{P}$ is real analytic from $D(1,r)$ to $P^{a,p-1}$ with $\Vert X_{P}\Vert_{r,a,p-1,D(1,r)}=O(r^{2})$. Moreover, since $X_{P}$ is independent of $\sigma$, we know $\Vert X_{P}\Vert^{lip}_{r,a,p-1,D(1,r)\times\Pi_0}=0$. Therefore, $\epsilon_{0}=O(r)$ satisfies the smallness condition when $r$ is small enough. 

Finally, we verify the first type of asymptotic estimate condition of $P_{0}$. By direct calculation, it is easy to obtain that
$$\frac{\partial^{2}P_{0}}{\partial\bar{z}_{n+t}\partial z_{m+t}}=-\frac{\mathrm{i}}{2}\int_{\mathbb{T}}\bar{u}\partial_{x}ue^{\mathrm{i}(m-n)x}dx+\frac{1}{2}(m+t)\int_{\mathbb{T}}\bar{u}ue^{\mathrm{i}(m-n)x}dx,$$
$$\frac{\partial^{2}P_{0}}{\partial z_{n+t}\partial z_{m-t}}=\frac{1}{4}(n+m)\int_{\mathbb{T}}\bar{u}^{2}e^{\mathrm{i}(m+n)x}dx,$$
$$\frac{\partial^{2}P_{0}}{\partial\bar{z}_{n+t}\partial\bar{z}_{m-t}}=-\frac{\mathrm{i}}{2}\int_{\mathbb{T}} u\partial_{x}u e^{-\mathrm{i}(n+m)x}dx.$$
Therefore, we have that $P_{0}\in FAE_{\Lambda_{0},\epsilon_{0},\rho_{0}}$. With the iterative lemma now applicable, we proceed to prove the convergence of the infinite iteration and the existence of an invariant torus. The full-dimensional property is proved in Section~\ref{sec:limiting-deformation}. We henceforth write $\Pi=\bigcap_{v\geq0}\Pi_v$ for the retained parameter set of Theorem~\ref{53}; Section 6 gives $\mathbb P(\Pi)>0$. All limiting statements below are made for $\sigma\in\Pi$.

We introduce a sequence of coordinate transformations mapping action-angle variables to Cartesian coordinates:
$$\begin{aligned}
\tau_{v}:D_{J_{v}}(s_{v},r_{v})&\rightarrow l^{a,p}_{\mathbb{C}}\times l^{a,p}_{\mathbb{C}}\\
z_{j}&=\sqrt{2(I_{j}+y_{j})}e^{-\mathrm{i}x_{j}},j\in J_{v},\\
\bar{z}_{j}&=\sqrt{2(I_{j}+y_{j})}e^{\mathrm{i}x_{j}},j\in J_{v},\\
z_{j}&=z_{j},j\in\mathbb{Z}\setminus J_{v},\\
\bar{z}_{j}&=\bar{z}_{j},j\in\mathbb{Z}\setminus J_{v},
\end{aligned}$$
for any $v\in\mathbb{N}$. The symplectic transformations obtained in the KAM procedure of Chapter \ref{301} and in Lemma \ref{51}(for eliminating the variable-coefficient normal frequency $\Omega_{v}$) are denoted by $g_{v}$ and $h_{v}$, respectively. Let $f_{v}=\tau_{v}\circ g_{v}\circ\tau_{v}^{-1},\tilde{f}_{v}=\tau_{v-1}\circ h_{v}\circ h_{-v}\circ\tau_{v-1}^{-1}$ and $f^{v}=f_{0}\circ\tilde{f}_{1}\circ f_{1}\circ\cdots\circ\tilde{f}_{v}\circ f_{v}$. Up to a coordinate transformation, this is precisely the composition of the first 
$v$ transformations in the iteration lemma. 

We shall need the following Cartesian estimate for the transformations
constructed in Lemma~\ref{51}. 

\begin{lem}[Cartesian Lipschitz estimate for $\widetilde f_v$]
\label{lem:chapter7-tilde-fv-lip}
Let $\mathcal D_v\subset\tau_{v-1}\bigl(D_{J_{v-1}}(s_v,r_v)\bigr)
 \subset l^{a,p}_{\mathbb C}\times l^{a,p}_{\mathbb C},
 \widetilde f_v=\tau_{v-1}\circ h_v\circ h_{-v}\circ\tau_{v-1}^{-1},v\geq1$. Put
$w_j:=\langle j\rangle^p e^{a\langle j\rangle},
 I_{*,v}:=\min_{j\in J_{v-1}}I_j,
 \delta_v:=2^{-v}.$ As in the choice of the iterative parameters, we have that
\begin{equation}\label{eq:chapter7-action-gap}
 r_v^2\leq 2^{-v^2}I_{*,v}
\end{equation}
and that $\gamma_0$ is sufficiently small. Then, uniformly with respect to
$\sigma\in\Pi_v$,
\begin{equation}\label{eq:chapter7-tildef-upper}
 \|\widetilde f_v(Q)-\widetilde f_v(Q')\|_{l^{a,p}\times l^{a,p}}
 \leq(1+\delta_v)
 \|Q-Q'\|_{l^{a,p}\times l^{a,p}},
 \qquad Q,Q'\in\mathcal D_v.
\end{equation}
The same estimate holds for $\widetilde f_v^{-1}$ on
$\widetilde f_v(\mathcal D_v)$. Consequently,
\begin{equation}\label{eq:chapter7-tildef-bilip}
 \frac{1}{1+\delta_v}\|Q-Q'\|_{l^{a,p}\times l^{a,p}}
 \leq
 \|\widetilde f_v(Q)-\widetilde f_v(Q')\|_{l^{a,p}\times l^{a,p}}
 \leq
 (1+\delta_v)\|Q-Q'\|_{l^{a,p}\times l^{a,p}}.
\end{equation}
Moreover,
\begin{equation}\label{eq:chapter7-tildef-D}
 \sup_{Q\in\mathcal D_v}\|D\widetilde f_v(Q)\|
 \leq1+\delta_v,
 \qquad
 \sup_{Q\in\widetilde f_v(\mathcal D_v)}
 \|D\widetilde f_v^{-1}(Q)\|
 \leq1+\delta_v.
\end{equation}

\end{lem}

\begin{proof}
We first consider one of the two transformations in Lemma~\ref{51},
corresponding to a fixed $t\in\{v,-v\}$. Denote its angle-dependent coefficient by
\[
 F_t=\sum_{i=0}^{v-1}F_t^i,
\]
where
\[
 \partial_{\omega|_{J_i}}F_t^i=\frac12\widetilde\Omega_t^i,
 \qquad [F_t^i]=0.
\]
By the Diophantine conditions and the estimates in Lemma~\ref{51}, 
\begin{equation}\label{eq:chapter7-F-derivatives}
 A_{1,t}+A_{2,t}
 :=
 \sup_{D(s_v)}
 \sum_{j\in J_{v-1}}|\partial_{x_j}F_t|
 +
 \sup_{D(s_v)}
 \max_{j\in J_{v-1}}
 \sum_{k\in J_{v-1}}|\partial_{x_kx_j}F_t|
 \leq Cv^C\sum_{i=0}^{v-1}\gamma_i.
\end{equation}
We denote the right-hand side by $\Gamma_v$. Since $F_t$ is real on the
real torus, integration along the vertical segment in the lifted angle strip
gives
\begin{equation}\label{eq:chapter7-imF}
 S_t:=\sup_{D(s_v)}|\operatorname{Im}F_t|
 \leq s_vA_{1,t}
 \leq s_v\Gamma_v.
\end{equation}
Thus
\begin{equation}\label{eq:chapter7-phase-modulus}
 |e^{\pm2\mathrm{i}F_t(x)}|
 \leq e^{2S_t}
 \leq1+Cs_v\Gamma_v.
\end{equation}

In these coordinates,
$\varpi=\sum_{j\in J_{v-1}}dx_j\wedge dy_j-
\frac{\mathrm{i}}{2}\sum_{j\notin J_{v-1}}d\bar z_j\wedge dz_j$.
Consequently, the phase $e^{-2\mathrm{i}F_t}$ is paired with the action shift $-\partial_xF_t\,z_t\bar z_t$. The explicit diagonal map in Lemma~\ref{51} reads
\begin{equation}\label{eq:chapter7-explicit-flow}
 x'=x,\qquad
 y'_j=y_j-\partial_{x_j}F_t(x)z_t\bar z_t,
 \qquad
 z'_t=e^{-2\mathrm{i}F_t(x)}z_t,
 \qquad
 \bar z'_t=e^{2\mathrm{i}F_t(x)}\bar z_t,
\end{equation}
while all the other normal coordinates are unchanged. Expressing
\eqref{eq:chapter7-explicit-flow} in Cartesian coordinates, for
$j\in J_{v-1}$ we obtain
\begin{equation}\label{eq:chapter7-cartesian-formula}
 q'_t=e^{-2\mathrm{i}F_t(x)}q_t,
 \qquad
 \bar q'_t=e^{2\mathrm{i}F_t(x)}\bar q_t,
 \qquad
 q'_j=q_j(1+\eta_{t,j})^{1/2},
 \qquad
 \bar q'_j=\bar q_j(1+\eta_{t,j})^{1/2},
\end{equation}
where
\begin{equation}\label{eq:chapter7-eta}
 \eta_{t,j}
 :=
 -\frac{2\partial_{x_j}F_t(x)q_t\bar q_t}
 {q_j\bar q_j}.
\end{equation}

We next record the estimates which allow us to work directly in Cartesian
coordinates. For $j\in J_{v-1}$ put
\[
 P_j:=q_j\bar q_j=2(I_j+y_j),
 \qquad
 R_j:=P_j^{1/2},
 \qquad
 \xi_j:=\frac{q_j}{R_j}=e^{-\mathrm{i}x_j}.
\]
By \eqref{eq:chapter7-action-gap},
\begin{equation}\label{eq:chapter7-basic-bounds}
 |P_j|\geq I_j,
 \qquad
 |q_j|+|\bar q_j|\leq C\sqrt{I_j},
 \qquad
 w_t|q_t|+w_t|\bar q_t|\leq2r_v.
\end{equation}
For two points $Q,Q'\in\mathcal D_v$, the algebraic identity
\[
 R_j(Q)-R_j(Q')
 =
 \frac{P_j(Q)-P_j(Q')}
 {R_j(Q)+R_j(Q')}
\]
and the definition of $\xi_j$ imply
\begin{equation}\label{eq:chapter7-xi-difference}
 |\xi_j(Q)-\xi_j(Q')|
 \leq
 \frac{C}{w_j\sqrt{I_j}}
 \|Q-Q'\|_{l^{a,p}\times l^{a,p}}.
\end{equation}
The real part of $x_j(Q)-x_j(Q')$ is taken to be its representative in
$[-\pi,\pi]$. Since
\[
 e^{-s_v}\leq|\xi_j|\leq e^{s_v},
\]
the elementary relations
\[
 \operatorname{Im}x_j=\log|\xi_j|,
 \qquad
 |e^{\mathrm{i}\theta}-1|
 \geq\frac{2}{\pi}|\theta|,
 \qquad |\theta|\leq\pi,
\]
give
\begin{equation}\label{eq:chapter7-angle-difference}
 |x_j(Q)-x_j(Q')|
 \leq
 \frac{C}{w_j\sqrt{I_j}}
 \|Q-Q'\|_{l^{a,p}\times l^{a,p}}.
\end{equation}

We now estimate the $t$-th component in
\eqref{eq:chapter7-cartesian-formula}. By
\eqref{eq:chapter7-F-derivatives} and
\eqref{eq:chapter7-angle-difference},
\begin{equation}\label{eq:chapter7-F-difference}
 |F_t(x(Q))-F_t(x(Q'))|
 \leq
 \frac{C\Gamma_v}{\sqrt{I_{*,v}}}
 \|Q-Q'\|_{l^{a,p}\times l^{a,p}}.
\end{equation}
From \eqref{eq:chapter7-phase-modulus},
\eqref{eq:chapter7-F-difference} and
\eqref{eq:chapter7-basic-bounds}. we obtain
\begin{align}
 w_t|q'_t(Q)-q'_t(Q')|
 &\leq
 (1+Cs_v\Gamma_v)
 w_t|q_t(Q)-q_t(Q')|
 \notag\\
 &\quad+
 C\Gamma_v\frac{r_v}{\sqrt{I_{*,v}}}
 \|Q-Q'\|_{l^{a,p}\times l^{a,p}}.
 \label{eq:chapter7-t-difference}
\end{align}
The same estimate holds for $\bar q_t'$. 

A direct finite-difference calculation in
\[
 \eta_{t,j}
 =
 -2(\partial_{x_j}F_t)
 \frac{q_t\bar q_t}{q_j\bar q_j}
\]
gives
\begin{equation}\label{eq:chapter7-eta-difference}
 w_j|q_j(Q')|
 \left|
 (1+\eta_{t,j}(Q))^{1/2}
 -(1+\eta_{t,j}(Q'))^{1/2}
 \right|
 \leq
 C\Gamma_v2^{-v^2/2}
 \|Q-Q'\|_{l^{a,p}\times l^{a,p}}.
\end{equation}
Consequently,
\begin{align}
 w_j|q'_j(Q)-q'_j(Q')|
 &\leq
 (1+C\Gamma_v2^{-v^2})
 w_j|q_j(Q)-q_j(Q')|
 \notag\\
 &\quad+
 C\Gamma_v2^{-v^2/2}
 \|Q-Q'\|_{l^{a,p}\times l^{a,p}},
 \label{eq:chapter7-j-difference}
\end{align}
and the same estimate holds for $\bar q'_j$.

All the remaining Cartesian coordinates are unchanged. Combining
\eqref{eq:chapter7-t-difference} and
\eqref{eq:chapter7-j-difference} in the
$l^{a,p}\times l^{a,p}$ norm gives, for the single transformation associated
with $t$,
\begin{equation}\label{eq:chapter7-single-lip}
 \|\tilde{f}_{v}(Q)-\tilde{f}_{v}(Q')\|
 \leq
 \left[
 1+C\Gamma_v
 \left(
 s_v+v^{C}2^{-v^2/2}
 \right)
 \right]
 \|Q-Q'\|.
\end{equation}
The inverse has exactly the same form with $F_t$ replaced by $-F_t$, and
therefore satisfies the same estimate. The derivative estimates in \eqref{eq:chapter7-tildef-D} follow by similar argument and we omit it.
\end{proof}

For completeness, the Cartesian estimate for the KAM map must be distinguished from its scaled action--angle estimate. On the smaller admissible charts, the estimates of Sections 3--4 give
\[
 \sup\|Df_v-I\|\le\mu_v:=e^{C_v\varepsilon_v/r_v}-1,\qquad
 \sum_v\mu_v<\infty.
\]
Let $C_\tau$ be a uniform bound for the ratio between the length of an admissible action--angle path, expressed in Cartesian coordinates, and the Cartesian distance between its endpoints. Integrating $D(f_v-id)$ along that path gives
\begin{equation}\label{eq:chapter7-f-bilip}
 (1-C_\tau\mu_v)\|Q-Q'\|
 \le\|f_v(Q)-f_v(Q')\|
 \le(1+C_\tau\mu_v)\|Q-Q'\|.
\end{equation}
The smallness choice for the initial perturbation ensures $C_\tau\mu_v<1/2$; any finite initial segment is absorbed in the constants below. Combining these bounds with \eqref{eq:chapter7-tildef-bilip} yields
\begin{equation}\label{eq:chapter7-prefix-bilip}
 C_*^{-1}\|Q-Q'\|\le\|f^v(Q)-f^v(Q')\|\le C_*\|Q-Q'\|,
\end{equation}
with a finite constant $C_*$ independent of $v$. The weighted calculation giving this Cartesian estimate is similar to the discussion in Section~\ref{sec:limiting-deformation} and we omit here.

From the estimates in \eqref{302} and Lemma~\ref{51}, together with the Cartesian bounds above, we conclude that the sequence $\lbrace f^{v}\rbrace_{v\geq 1}$ converges uniformly to a homeomorphism $f$ on $D_{*}:=\lbrace \vert q_{j}\vert^{2}=2I_{j},j\in\mathbb{Z}\rbrace$ and $H\circ f=N_*$ which is a normal form with limit frequency $\omega_*$. We use the notation $T_I=D_*$ for this straight torus.

Write $E_s:=l^{a,s}_{\mathbb C}\times l^{a,s}_{\mathbb C}$, with the real restriction understood. Using the $p-1$ derivative estimate in \eqref{302} and the same weighted Cartesian estimates with $w_j$ replaced by $w_j/\langle j\rangle$, we obtain bounded extensions of $Df^v$ to $E_{p-1}$, uniformly on the intermediate inputs of the construction. The smallness of $X_{P_v}$ and the convergence of the frequencies therefore give
\[
 \sup_{Q\in T_I}
 \|X_H(f^v(Q))-Df^v(Q)X_{N_v}(Q)\|_{E_{p-1}}
 \leq 4^{-v},
\]
and hence
\[
 \sup_{Q\in T_I}
 \|X_H(f^v(Q))-Df^v(Q)X_{N_*}(Q)\|_{E_{p-1}}
 \leq 2^{-v}.
\]
We set
\[
\begin{aligned}
 Y_I:\mathbb T^{\mathbb Z}&\longrightarrow T_I\subset E_p,\\
 \theta&\longmapsto
 \bigl((\sqrt{2I_j}e^{-\mathrm{i}\theta_j})_{j\in\mathbb Z},
       (\sqrt{2I_j}e^{\mathrm{i}\theta_j})_{j\in\mathbb Z}\bigr).
\end{aligned}
\]
With the product topology on $\mathbb T^{\mathbb Z}$, the map $Y_I$ is naturally a homeomorphism. The rapid decay of the actions gives
\[
 \sum_{j\in\mathbb Z}
 w_j^2(1+|\omega_{*,j}|^2)I_j<\infty.
\]
Thus $q(t):=Y_I(\omega_*t)$ belongs to $C^1(\mathbb R;E_p)$ and is the flow generated by $X_{N_*}$. Consequently $f^v(q(t))$ is an approximate solution satisfying
\[
 \|\partial_t(f^v(q(t)))-X_H(f^v(q(t)))\|_{E_{p-1}}
 \leq 2^{-v}.
\]

The time derivatives can nevertheless be compared in $E_p$. Indeed, uniform convergence of the weighted tails shows that $X_{N_*}:T_I\to E_p$ is continuous, so $X_{N_*}(T_I)$ is compact in $E_p$. Section~\ref{sec:limiting-deformation} establishes the uniform strong convergence in \eqref{geo:L-factor}: for every fixed $h\in E_p$,
\[
 \sup_{Q\in T_I}\|(Df^v(Q)-L(Q))h\|_{E_p}\longrightarrow0,
 \qquad
 \sup_{v,Q\in T_I}\|Df^v(Q)\|_{\mathcal L(E_p)}<\infty.
\]
Taking a finite net in the compact set $X_{N_*}(T_I)$ therefore gives
\[
 \sup_{Q\in T_I}
 \|(Df^v(Q)-L(Q))X_{N_*}(Q)\|_{E_p}\longrightarrow0.
\]
It follows that
\[
 \partial_t(f^v(q(t)))
 =Df^v(q(t))X_{N_*}(q(t))
 \longrightarrow L(q(t))X_{N_*}(q(t))
\]
uniformly in $t$ in $E_p$. Together with the uniform convergence of $f^v(q(t))$, this proves that $f(q(t))\in C^1(\mathbb R;E_p)$.

The linear part of $X_H$ maps $E_p$ continuously into $E_{p-2}$, while its nonlinear part maps $E_p$ continuously into $E_{p-1}$. We can therefore pass to the limit in the approximate equation in $E_{p-2}$ and obtain
\[
 \partial_t f(q(t))=X_H(f(q(t))).
\]
Hence $f(q(t))=f(Y_I(\omega_*t))=:\Phi(\omega_*t)$ is a solution of equation \eqref{2}. Since $f$ and $Y_I$ are both homeomorphisms, $\Phi$ is a homeomorphism. Consequently, by Lemma 10.2 in \cite{BGR}, this solution is almost periodic but not quasi-periodic.

\subsection{Symplectic deformation and full dimension}
\label{sec:limiting-deformation}

Fix $\sigma\in\Pi$. We work on the real phase space $E$ defined in the introduction and retain the notation of the preceding subsection:
\[
 T_I=D_* =\{(q,\bar q)\in E:|q_j|^2=2I_j\ (j\in\mathbb Z)\},
 \qquad I_j>0,\quad\sum_jw_j^2I_j<\infty.
\]
 This torus is compact in $E$ since $Y_I$ is a homeomorphism. Throughout this subsection, $\|\cdot\|$ denotes the norm of $E$ or the corresponding operator norm. The perturbation sizes $\epsilon_v,\varepsilon_v$, the domain radii $r_v$, the widths $s_v$ and the actions $I_j$ have the meanings fixed in Section 5 and Section~\ref{sec:convergence}.

We verify that the limiting map $f:T_I\to\mathcal T=f(T_I)$ is a symplectic deformation as defined in Section~\ref{sec:main-result}. Besides the homeomorphism property, we establish a bounded, invertible symplectic derivative $L$ satisfying
\begin{equation}\label{geo:def-rem}
 \lim_{\substack{Q'\to Q\\Q'\in T_I,\ Q'\ne Q}}
 \frac{\|f(Q')-f(Q)-L(Q)(Q'-Q)\|}{\|Q'-Q\|}=0,
\end{equation}
with $L,L^{-1}$ strongly continuous and uniformly bounded on $T_I$. Since all $I_j$ are positive, the image $\mathcal T$ is then full dimensional by definition. We also prove the tangent-cone identity and hence the Lagrangian property; this last argument applies to every symplectic deformation of a full-support straight torus under the definition of Section~\ref{sec:main-result}.

The finite transformations and their estimates are as follows. The diagonal elimination uses the old action--angle chart, whereas the KAM map uses the new one. Accordingly,
\begin{equation}\label{geo:maps}
 \widetilde f_v=\tau_{v-1}\circ h_v\circ h_{-v}\circ\tau_{v-1}^{-1},
 \qquad f_v=\tau_v\circ g_v\circ\tau_v^{-1},\qquad
 \psi_v:=\widetilde f_v\circ f_v,
\end{equation}
and $f^v=f_0\circ\psi_1\circ\cdots\circ\psi_v$. For $M\ge N$ define
\[
 \mathcal F_{N,M}:=\psi_{N+1}\circ\cdots\circ\psi_M,
 \qquad \mathcal F_{N,N}:=\mathrm{Id}.
\]
Thus $f^M=f^N\circ\mathcal F_{N,M}$, with the rightmost factor acting first.

We recall the Cartesian estimates supplied by the KAM construction. On the smaller admissible charts, $|q_j|,|\bar q_j|\asymp\sqrt{2I_j}$ for $j\in J_v$. The weighted amplitude ratios and coordinate factors are included in $C_v$ below. To display the coordinate weights, put
\[
 \mathsf A_v=\operatorname{diag}(\sqrt{2I_j})_{j\in J_v},\qquad
 W_J=\operatorname{diag}(w_j)_{j\in J_v},\qquad
 W_{J^c}=\operatorname{diag}(w_j)_{j\notin J_v},
\]
\[
 S_v=\operatorname{diag}(W_J\mathsf A_v,W_J\mathsf A_v^{-1},W_{J^c},W_{J^c}).
\]
At a fixed KAM step $v$, write $F$ for the generating function constructed in Section 3. In the displays involving the coordinate weights $S_v$ or $\mathsf W$, norms are taken in the corresponding unweighted $l^2$ coordinate spaces. The weighted estimates of Sections 3 and 4, on the smaller domains, give
\begin{equation}\label{geo:weighted}
 \|S_vX_F\|\le C_v\varepsilon_v,\qquad
 \|S_vDX_FS_v^{-1}\|\le C_v\frac{\varepsilon_v}{r_v}.
\end{equation}
The identity $F_{yy}=0$ is used in the second estimate. The finite-dimensional factors and the strip losses are included in $C_v$, as in the choice of parameters in Section 5.

Write $\mathsf W=\operatorname{diag}(w_j,w_j)_{j\in\mathbb Z}$ for the Cartesian weight. Direct differentiation of $q_j=\sqrt{2(I_j+y_j)}e^{-\mathrm{i}x_j}$ gives
\[
 \|\mathsf W D\tau_v S_v^{-1}\|+
 \|S_vD\tau_v^{-1}\mathsf W^{-1}\|\le C,
\]
\[
 \|\mathsf W D^2\tau_v[\Xi_1,\Xi_2]\|
 \le\frac {C_v}{r_v}\|S_v\Xi_1\|\,\|S_v\Xi_2\|.
\]
For $\widehat X_v=(\tau_v)_*X_F$, the chain rule therefore yields, after enlarging $C_v$ to absorb these coordinate factors,
\[
 D\widehat X_v
 =D\tau_v\,DX_F\,D\tau_v^{-1}
   +D^2\tau_v[D\tau_v^{-1}(\cdot),X_F],
 \qquad\|D\widehat X_v\|\le C_v\frac{\varepsilon_v}{r_v}.
\]
The variational equation for its flow implies
\begin{equation}\label{geo:kam}
 \sup_{0\le t\le1}\|Df_v^t-I\|
 \le e^{C_v\varepsilon_v/r_v}-1=\mu_v,
 \qquad\sum_{v\ge1}\mu_v<\infty.
\end{equation}
Here $f_v^t=\tau_v\circ X_F^t\circ\tau_v^{-1}$, so $f_v^1=f_v$. The summability follows from the parameter choices in Section 5; in particular $\varepsilon_v/r_v=\epsilon_0^{(12/11)^v}$ absorbs the factors $C_v$.

For the diagonal map, use $F_t(x)$ from \eqref{eq:chapter7-F-derivatives}, with $t\in\{v,-v\}$. Set $n_t=q_t\bar q_t$, $m_j=q_j\bar q_j$ and
$\eta_{t,j}=-2(\partial_{x_j}F_t)n_t/m_j$. For $H=(h,\bar h)$,
\[
 Dx_k(Q)H=-\frac1{2\mathrm{i}}\left(\frac{h_k}{q_k}-\frac{\bar h_k}{\bar q_k}\right)=:\zeta_k,
\]
\begin{align*}
 D\eta_{t,j}(Q)H={}&-\frac{2n_t}{m_j}\sum_{k\in J_{v-1}}
 \partial_{x_kx_j}F_t\,\zeta_k
 -\frac{2\partial_{x_j}F_t}{m_j}(\bar q_th_t+q_t\bar h_t)\\
 &+\frac{2(\partial_{x_j}F_t)n_t}{m_j^2}
 (\bar q_jh_j+q_j\bar h_j).
\end{align*}
Separating the phase multiplier in \eqref{eq:chapter7-cartesian-formula}, and using \eqref{eq:chapter7-action-gap}, \eqref{eq:chapter7-F-derivatives} and \eqref{eq:chapter7-basic-bounds}, gives
\[
 D\widetilde f_v(Q)=\widetilde U_v(Q)+\widetilde K_v(Q),\qquad
 \|\widetilde K_v(Q)\|\le\kappa_v,
 \qquad\sum_{v\ge1}\kappa_v<\infty.
\]
On real points, $\widetilde U_v$ is a unitary phase rotation on the modes $v,-v$ and the identity on all other modes. For example, the weighted calculation above bounds $\kappa_v$ by
$C\Gamma_v\langle v\rangle^{1/2}2^{-v^2/2}$, with $\Gamma_v$ from \eqref{eq:chapter7-F-derivatives}.

For two admissible inputs, lift each real angle difference to $[-\pi,\pi]$ and take the image of the straight action--angle path under $\tau_v$ (or $\tau_{v-1}$ for $\widetilde f_v$). This path has length at most $C\|Q'-Q\|$. For the diagonal factor its new normal coordinate varies linearly. Thus the difference of the phase multipliers is paired with a new-mode difference of size at most $Cr_v$, and \eqref{eq:chapter7-angle-difference}--\eqref{eq:chapter7-single-lip} give a remainder bounded by a summable multiple of $\|Q'-Q\|$. For the KAM factor, \eqref{geo:kam} gives
\[
 \|f_v(Q')-f_v(Q)-Df_v(Q)(Q'-Q)\|
 \le C\mu_v\|Q'-Q\|.
\]
Using the identity
\[
 R_{A\circ B}(Q',Q)=R_A(BQ',BQ)+DA(BQ)R_B(Q',Q),
 \quad R_A(Q',Q)=A(Q')-A(Q)-DA(Q)(Q'-Q),
\]
we obtain, on all real intermediate inputs of the construction,
\begin{align}
 \|\psi_v(Q')-\psi_v(Q)\|&\le(1+\ell_v)\|Q'-Q\|,
 &\|\psi_v(Q)-Q\|&\le d_v,\label{geo:lip}\\
 D\psi_v(Q)&=U_v(Q)+K_v(Q),
 &\|K_v(Q)\|&\le c_v,\label{geo:block}\\
 \|R_{\psi_v}(Q',Q)\|&\le b_v\|Q'-Q\|,&&\label{geo:rem}
\end{align}
where
\begin{equation}\label{geo:sums}
 \sum_{v\ge1}(\ell_v+d_v+c_v+b_v)<\infty.
\end{equation}
In \eqref{geo:block},
\[
 U_v(Q)=\widetilde U_v(f_v(Q)),\qquad
 K_v(Q)=\widetilde U_v(f_v(Q))(Df_v(Q)-I)
       +\widetilde K_v(f_v(Q))Df_v(Q).
\]
In particular $U_v$ is continuous on each fixed-stage domain and is the identity outside the modes $v,-v$. The constants in \eqref{geo:sums} incorporate the finitely many initial transformations. The strict forward inclusions in the construction ensure that, for every fixed $N$,
\begin{equation}\label{geo:domain}
 \overline{\bigcup_{M\ge N}\mathcal F_{N,M}(T_I)}
 \quad\text{is a compact subset of }\operatorname{Dom}(f^N).
\end{equation}
Compactness follows from the uniform tail convergence below and compactness of $T_I$; the strict inclusion is supplied by the domain margins in the finite-step construction.

\begin{samepage}
\begin{prop}\label{prop:geometric-limit}
For each $\sigma\in\Pi$, the limiting map $f:T_I\to\mathcal T=f(T_I)$ is a symplectic deformation in the sense defined in Section~\ref{sec:main-result}. In particular, $\mathcal T$ is a full-dimensional torus. Its derivative is
\begin{equation}\label{geo:derivative}
 L(Q)h=\lim_{v\to\infty}Df^v(Q)h\qquad(h\in E),
\end{equation}
and
\begin{equation}\label{geo:uniform-inverse}
 \sup_{Q\in T_I}\bigl(\|L(Q)\|+\|L(Q)^{-1}\|\bigr)<\infty.
\end{equation}
Moreover,
\begin{equation}\label{geo:tangent-conclusion}
 T_{f(Q)}\mathcal T=L(Q)T_QT_I,
 \qquad (T_{f(Q)}\mathcal T)^{\perp_\varpi}=T_{f(Q)}\mathcal T.
\end{equation}
Consequently, the full-dimensional invariant torus in Theorem~\ref{53} is Lagrangian.
\end{prop}
\end{samepage}

\begin{proof}
On real points the phase blocks $U_v$ and their inverses have norm one. Therefore every consecutive derivative product is bounded by
\[
 \|D\psi_m\cdots D\psi_n\|
 \le\prod_{v=m}^n(1+c_v)\le K,
 \qquad K:=\exp\Big(\sum_{v\ge1}c_v\Big).
\]
The base points may be different at different stages. After enlarging $K$, this bound includes the fixed factor $f_0$. Similarly, let
$K_\ell=\prod_{v\ge1}(1+\ell_v)<\infty$.
Telescoping \eqref{geo:lip} gives uniform convergence
\[
 p_N(Q):=\lim_{M\to\infty}\mathcal F_{N,M}(Q),\qquad Q\in T_I,
 \qquad\operatorname{Lip}(p_N)\le K_\ell.
\]
By \eqref{geo:domain},
\begin{equation}\label{geo:factor}
 f=f^N\circ p_N.
\end{equation}

\emph{Strong convergence and continuity of the derivatives.}
For $M\ge v$ set
\[
 Q_v^{(M)}=\mathcal F_{v,M}(Q),\qquad
 B_v^{(M)}=D\psi_v(Q_v^{(M)}),\qquad U_v^{(M)}=U_v(Q_v^{(M)}).
\]
For fixed $v$, the intermediate inputs converge uniformly on $T_I$; hence
$U_v^{(M)}\to U_v^\infty(Q):=U_v(p_v(Q))$ uniformly in operator norm.
Denote by $\mathsf P_0$ the zero-mode projection and by $\mathsf P_v$ the projection onto the pair of modes $v,-v$. Put
\[
 \mathsf P_{\le N}=\mathsf P_0+\sum_{v=1}^N\mathsf P_v,
 \qquad\mathsf P_{>N}=\mathrm{Id}-\mathsf P_{\le N}.
\]
The block-diagonal operator
\[
 \mathcal U_N(Q)=\mathsf P_{\le N}
          +\sum_{v>N}U_v^\infty(Q)\mathsf P_v
\]
has norm one and the explicit bounded inverse
\[
 \mathcal U_N(Q)^{-1}=\mathsf P_{\le N}
          +\sum_{v>N}(U_v^\infty(Q))^{-1}\mathsf P_v.
\]
Its finite counterparts converge strongly, uniformly in $Q$ for each fixed vector. Indeed, the finitely many blocks below a fixed cutoff converge uniformly, while the remaining contribution applied to $h\in E$ is bounded by $2\|\mathsf P_{>R}h\|$.

Expanding the product of $U_v^{(M)}+K_v(Q_v^{(M)})$ yields
\begin{equation}\label{geo:tail-op}
 \left\|D\mathcal F_{N,M}(Q)-\prod_{v=N+1}^MU_v^{(M)}\right\|
 \le\vartheta_N,
 \qquad\vartheta_N:=\exp\Big(\sum_{v>N}c_v\Big)-1\longrightarrow0.
\end{equation}
Since the phase blocks act on disjoint modes,
\begin{equation}\label{geo:strong-tail}
 \|(D\mathcal F_{N,M}(Q)-I)h\|
 \le2\|\mathsf P_{>N}h\|+\vartheta_N\|h\|.
\end{equation}
For fixed $R>N$ and $M>R$, factor
\[
 D\mathcal F_{N,M}(Q)
 =D\mathcal F_{N,R}(\mathcal F_{R,M}(Q))D\mathcal F_{R,M}(Q).
\]
The first factor converges uniformly in operator norm, because it is the derivative of a fixed finite composition evaluated on uniformly convergent compact sets. Applied to a fixed $h$, the second factor differs from $h$ by at most
$2\|\mathsf P_{>R}h\|+\vartheta_R\|h\|$. First choosing $R$ large and then taking $M$ large proves uniform strong convergence on $T_I$:
\[
 L_N(Q)h:=\lim_{M\to\infty}D\mathcal F_{N,M}(Q)h,
 \qquad\|L_N(Q)\|\le K,
 \qquad\|L_N(Q)-\mathcal U_N(Q)\|\le\vartheta_N.
\]
For each $h$, the map $Q\mapsto L_N(Q)h$ is continuous, as a uniform limit of continuous maps. Therefore
\begin{equation}\label{geo:L-factor}
 L(Q)=Df^N(p_N(Q))L_N(Q)
      =\mathop{\mathrm{s\! -\! lim}}_{M\to\infty}Df^M(Q)
\end{equation}
is strongly continuous and uniformly bounded. The last identity shows that it is independent of $N$.

\emph{Invertibility and strong continuity of the inverse.}
Choose one $N$ with $\vartheta_N<1/2$. Since $\|\mathcal U_N(Q)^{-1}\|=1$, a Neumann series gives
\[
 L_N(Q)^{-1}
 =[\mathrm{Id}+\mathcal U_N(Q)^{-1}(L_N(Q)-\mathcal U_N(Q))]^{-1}
   \mathcal U_N(Q)^{-1},\qquad\|L_N(Q)^{-1}\|\le2.
\]
The finite-step derivative $Df^N(p_N(Q))$ is invertible. By \eqref{geo:domain}, its inverse is norm-continuous on the compact input set, so
\[
 C_N:=\sup_{Q\in T_I}\|[Df^N(p_N(Q))]^{-1}\|<\infty,
 \quad L(Q)^{-1}=L_N(Q)^{-1}[Df^N(p_N(Q))]^{-1},
 \quad\sup_Q\|L(Q)^{-1}\|\le2C_N.
\]
This proves \eqref{geo:uniform-inverse}. For $Q_m\to Q$ and fixed $h\in E$,
\[
 L(Q_m)^{-1}h-L(Q)^{-1}h
 =L(Q_m)^{-1}[L(Q)-L(Q_m)]L(Q)^{-1}h\longrightarrow0.
\]
Indeed, $L(Q)^{-1}h$ is a fixed vector, the middle difference tends to zero strongly, and the left factor is uniformly bounded. Thus the inverse field is strongly continuous.

Every finite composition is symplectic. Continuity of $\varpi$ and the strong convergence in \eqref{geo:derivative} imply, for $h,k\in E$,
\[
 \varpi(L(Q)h,L(Q)k)
 =\lim_{M\to\infty}\varpi(Df^M(Q)h,Df^M(Q)k)=\varpi(h,k).
\]
The preceding Neumann argument establishes invertibility independently of this symplectic identity.

\emph{First-order remainder.}
For $Q,Q'\in T_I$, let
$\Delta_v=\mathcal F_{v,M}(Q')-\mathcal F_{v,M}(Q)$. Then
\[
 \Delta_{v-1}=B_v^{(M)}\Delta_v+\mathcal R_v,
 \qquad\|\mathcal R_v\|\le b_v\|\Delta_v\|.
\]
Successive substitution gives the exact tail remainder as a sum of
$B_{N+1}^{(M)}\cdots B_{v-1}^{(M)}\mathcal R_v$, with empty products equal to the identity. The uniform product and Lipschitz bounds therefore give
\[
 \|\mathcal F_{N,M}(Q')-\mathcal F_{N,M}(Q)
       -D\mathcal F_{N,M}(Q)(Q'-Q)\|
 \le KK_\ell\sum_{v>N}b_v\|Q'-Q\|.
\]
Passing to the limit and writing $\eta_N=KK_\ell\sum_{v>N}b_v\to0$, we obtain
\begin{equation}\label{geo:tail-rem}
 \|p_N(Q')-p_N(Q)-L_N(Q)(Q'-Q)\|
 \le\eta_N\|Q'-Q\|.
\end{equation}
For large $N$, $2\eta_N<1$. Together with $\|L_N(Q)^{-1}\|\le2$, this implies that $p_N$ is injective. The finite prefix $f^N$ is injective, so $f$ is a continuous injection of the compact torus $T_I$, hence a homeomorphism onto $\mathcal T$.

Set $p=p_N(Q)$, $p'=p_N(Q')$ and $h=Q'-Q$. By \eqref{geo:factor} and \eqref{geo:L-factor},
\begin{align*}
 f(Q')-f(Q)-L(Q)h
 ={}&f^N(p')-f^N(p)-Df^N(p)(p'-p)\\
 &+Df^N(p)[p'-p-L_N(Q)h].
\end{align*}
For fixed $N$, the first line is $o(\|h\|)$, since $f^N$ is $C^1$ and $p_N$ is Lipschitz. The second is bounded by $K\eta_N\|h\|$. First letting $Q'\to Q$ and then $N\to\infty$ proves \eqref{geo:def-rem}. Thus $f$ satisfies every requirement of the symplectic deformation definition in Section~\ref{sec:main-result}. Since $I_j>0$ for every $j\in\mathbb Z$, its image $\mathcal T$ is full dimensional.

\emph{Transport of tangent cones.}
At $Q=(q,\bar q)\in T_I$, differentiating $|q_j|^2=2I_j$ along curves gives
\[
 T_QT_I=\{(h,\bar h)\in E:h_j\in \mathrm{i}q_j\mathbb R\text{ for every }j\}.
\]
Conversely, for $h_j=\mathrm{i}c_jq_j$ with $c_j\in\mathbb R$, the curve
$\gamma_j(t)=q_je^{\mathrm{i}tc_j}$ is $C^1$ in the weighted space, with derivative $h_je^{\mathrm{i}tc_j}$, by dominated convergence. This proves the equality even when the sequence $(c_j)$ is unbounded. Testing symplectic orthogonality against tangent vectors supported in a single mode gives
\[
 (T_QT_I)^{\perp_\varpi}
 =\{(v,\bar v)\in E:\operatorname{Re}(\bar q_jv_j)=0\text{ for every }j\}
 =T_QT_I,
\]
because $I_j>0$ for all $j\in\mathbb Z$.

If $\gamma\in C^1([0,1];E)$ takes values in $T_I$, the first-order remainder yields
\[
 (f\circ\gamma)'(t)=L(\gamma(t))\gamma'(t).
\]
This derivative is continuous: for $t\to t_0$,
\begin{align*}
 &\|L(\gamma(t))\gamma'(t)-L(\gamma(t_0))\gamma'(t_0)\|\\
 &\quad\le C\|\gamma'(t)-\gamma'(t_0)\|
       +\|[L(\gamma(t))-L(\gamma(t_0))]\gamma'(t_0)\|
       \longrightarrow0.
\end{align*}
The second term uses strong continuity on the fixed vector $\gamma'(t_0)$.

For the inverse, write $Y=f(Q)$, $Y'=f(Q')$, $h=Q'-Q$ and
\[
 Y'-Y=L(Q)h+r(Q',Q),\qquad r(Q',Q)=o(\|h\|).
\]
For $Q'$ sufficiently close to $Q$,
\[
 \|Y'-Y\|\ge\frac{\|h\|}{2\|L(Q)^{-1}\|},
\]
and hence
\[
 f^{-1}(Y')-f^{-1}(Y)-L(Q)^{-1}(Y'-Y)
 =-L(Q)^{-1}r(Q',Q)=o(\|Y'-Y\|).
\]
The field $Y\mapsto L(f^{-1}(Y))^{-1}$ is strongly continuous and uniformly bounded. The same curve argument applies to $f^{-1}$, giving both inclusions in
\[
 T_{f(Q)}\mathcal T=L(Q)T_QT_I.
\]
Finally, an invertible symplectic operator satisfies
$(L(Q)V)^{\perp_\varpi}=L(Q)V^{\perp_\varpi}$. Thus
\[
 (T_{f(Q)}\mathcal T)^{\perp_\varpi}
 =L(Q)(T_QT_I)^{\perp_\varpi}
 =L(Q)T_QT_I=T_{f(Q)}\mathcal T.
\]
Thus the full-dimensional torus $\mathcal T$, obtained as a symplectic deformation of $T_I$, has a Lagrangian tangent cone at every point. This completes the proof of Theorem~\ref{53} and of the asserted Lagrangian property.
\end{proof}

\section{small denominator equation with large variable coefficient}\label{119}
In this section, we consider the first order partial differential equation
\begin{equation}\label{411}
(\mathrm{i}\partial_{\omega}+\lambda(1+a(x)))u=p(x),x\in\mathbb{T}^{n},
\end{equation}
for the unknown function $u$ defined on the torus $\mathbb{T}^{n}$, where $\omega=(\omega_{1},\cdots,\omega_{n})\in\mathbb{R}^{n}$ and $\lambda\in\mathbb{C}$. To keep the exposition simple, we tailor the parameters and properties below to the context of the KAM theorem for almost periodic solutions in this paper, rather than striving for optimality or generality.

Let $s_{0}>0$ be a constant and $\gamma_{0}>0$ be sufficiently small. Define sequences $\lbrace s_{l}\rbrace_{l=0}^{\infty}$ and $\lbrace\gamma_{l}\rbrace_{l=0}^{\infty}$ and `bridges' from $s_{m}$ to $s_{m+1}$ as follows.
$$s_{l}=\frac{s_{0}}{(l+1)^{20(l+1)^{2}}},l=0,1,2,\cdots,$$
$$\gamma_{l}=\gamma_{0}^{(\frac{12}{11})^{l}},l=0,1,2,\cdots,$$
For $m\geq0$, put $M:=\max\{2,m+1\}$ and define
$$s_{m}^{k}=\frac{s_{m}}{M^{4k(m+1)}},\qquad k=1,2,\cdots,5.$$
Thus $s_{m+1}<s_m^5<s_m^4<s_m^3<s_m^2<s_m^1<s_m$, including the initial cases $m=0,1$.
Let $J_{0}\subset J_{1}\subset\cdots\subset J_{m}=J\subset \mathbb{Z}$ be a sequence of subset of $\mathbb{Z}$ and $\sharp J=n$.

We first introduce the following elementary lemma.
\begin{lem}\label{58}
	 Consider the first order partial differential equation
	 \begin{equation}\label{57}
	 	\partial_{\omega}u=p(x),x\in\mathbb{T}^{n},
	 \end{equation}
	for the unknown function $u$ defined on the torus $\mathbb{T}^{n}$, where $\omega=(\omega_{1},\cdots,\omega_{n})\in\mathbb{R}^{n}$ and $\lambda\in\mathbb{C}$. Assume\\
	(i)There are constants $\alpha>0$ and $\tau>n$ such that
	$$\vert\langle k,\omega\rangle\vert\geq\frac{\alpha}{\langle k\rangle^{\tau}},k\in\mathbb{Z}^{n}\setminus\lbrace 0\rbrace.$$
	(ii)$p(x)$ is real analytic in $x\in D(s)$, $\bar{p}=0$ and $\Vert p\Vert_{s,\tau}<\infty$.
	Then \eqref{57} has a unique zero-average solution $u(x):D(s)\rightarrow\mathbb{C}$ which is real analytic and satisfies
	$$\Vert u(x)\Vert_{D(s)}\leq\frac{1}{\alpha}\Vert p\Vert_{s,\tau}.$$
\end{lem}	
The proof is trivial in KAM theory and we omit it.

Our goal is to exploit the decomposition $a(x)=\sum_{l=0}^{m}a_{l}(x)$
and eliminate $a_{l}(x)$ successively through a sequence of transformations, thereby converting equation \eqref{411} into a constant‑coefficient equation that we are wl equipped to handle.
\begin{lem}\label{408}
	Consider the equation \eqref{411}.
	Assume\\
	(i)There are constants $\alpha_{l}(l=0,1,\cdots,m)$ and $\tau_{l}>\sharp J_{l}$ such that
	\begin{equation}
		\vert\langle k,\omega\vert_{J_{l}}\rangle\vert\geq\frac{\alpha_{l}}{\langle k\rangle^{\tau_{l}}},0\neq k\in\mathbb{Z}^{J_{l}},l=0,1,\cdots,m,
	\end{equation}
	\begin{equation}\label{405}
		\sup_{j\in J_{l}}\vert\omega_{j}\vert\leq C\langle l\rangle^{2},l=0,1,\cdots,m.
	\end{equation}
	(ii)$a:D(s_{m})\rightarrow\mathbb{C}$ has the decomposition
	$$a(x)=\sum_{l=0}^{m}a_{l}(x),$$
	where $a_{l}:D(s_{l})\rightarrow\mathbb{C}$ is real analytic and is of zero average:$[a_{l}]=0$. Moreover, assume $a_{l}(x)$ depends only on $x\in\mathbb{T}_{s_{l}}^{J_{l}}$ and 
	\begin{equation}\label{407}
	\Vert a_{l}\Vert_{s_{l},\tau_{l}}\leq\alpha_{l}\gamma_{l}.
	\end{equation}
	Then there exists an invertible coordinate transformation
	$$
	\begin{aligned}
	T:D(s_{m}^{2})&\rightarrow D(s_{m}^{1})\\
	x&\mapsto\phi=x+\sum_{l=0}^{m}b_{l}(x)\omega	
	\end{aligned}
	$$
	such that\\
	(i)Whenever $u(x)$ is a solution of \eqref{411} on $D(s_{m}^{1})$, $v(\phi):=u(x)=u(T^{-1}\phi)$ satisfies
	$$(\mathrm{i}\partial_{\omega}+\lambda)v=p^{*}(\phi)$$
	on $D(s_{m}^{3})$ with $\Vert p^{*}\Vert_{D(s_{m}^{3})}\leq 2\Vert p\Vert_{D(s_{m})}$.\\
	(ii)$b_{l}$ is real analytic in $D(s_{m}^{2})$ and satisfies $\Vert b_{l}\Vert_{D(s_{m}^{2})}\leq \gamma_{l}^{1-},l=0,1,
	\cdots,m$.\\
	(iii)There exists a family of functions $\tilde{b}_{l}(l=0,1,\cdots,m)$ which is real analytic such that $T^{-1}:D(s_{m}^{3})\rightarrow D(s_{m}^{2})$ has the form
	$$x=T^{-1}\phi=\phi+\sum_{l=0}^{m}\tilde{b}_{l}(\phi)\omega$$
	with $\Vert\tilde{b}_{l}\Vert_{D(s_{m}^{3})}\leq \gamma_{l}^{1-}$.
\end{lem}
\begin{proof} The basic idea for the proof comes from \cite{HXY}. In the present scenario for constructing almost periodic solutions, the size of the frequency $\sup_{j\in J_{l}}\vert\omega_{j}\vert$ naturally grows toward $\infty$ as the dimension of the torus increases ($J_{l}\rightarrow \mathbb{Z})$. So some necessary modifications are needed.
	
   We observe that $a(x)$ admits an expansion $a(x)=\sum_{i=1}^{v-1}a_{i}(x)$, where $a_{i}(x)$ is analytic in $\lbrace x\in \mathbb{C}^{i}/2\pi\mathbb{Z}^{i}:\vert Imx\vert\leq s_{i}\rbrace$. This provides an opportunity to successively eliminate $a_{i}(x)$ through a sequence of transformations. Since 
$a_{i}(x)$ depends only on the first $i$ variables, the invertibility of these transformations is evident and we arrive at the same estimate \eqref{412}. Our plan is to eliminate each $a_{l}$ step by step through a sequence of transformations. We achieve this by a mathematical induction argument. Let $\tilde{s}_{l}=s_m^1M^{-4l}$ for $l=0,1,\cdots,m+1$. These are the corresponding `bridges', with $\tilde{s}_0=s_m^1$ and $\tilde{s}_{m+1}=s_m^2$, for the $m+1$ steps $j=0,\ldots,m$. Suppose we are at the $j$-th step. More specifically, by some abuse of notation, we consider the equation
	\begin{equation}\label{406}
		(\mathrm{i}\partial_{\omega}+\lambda(1+\sum_{l=j}^{m}a_{l}(x)))u=p(x)
	\end{equation}
	on $D(\tilde{s}_{j})$. Let $b_{j}(x)=\partial_{\omega}^{-1}a_{j}(x)$. By lemma \ref{58}, $b_{j}(x):D(\frac{1}{\langle j\rangle^{4j}}s_{j})\rightarrow\mathbb{C}$ is real analytic and satisfies
	\begin{equation}\label{61}
	\Vert b_{j}(x)\Vert_{D(\frac{1}{\langle j\rangle^{4j}}s_{j})}\leq\gamma_{j}^{1-}.
	\end{equation}
	
	Put $r_j:=s_j/\langle j\rangle^{4j}$ and $\kappa_j:=\gamma_j^{1/10}$. The derivative estimate below and the reality of $b_j$ on the real torus give $|\operatorname{Im}b_j(x)|\leq\kappa_j|\operatorname{Im}x_{J_j}|$ for $x\in D(9r_j/10)$. We define
	$$\phi=T_jx=x+b_j(x)\omega.$$
	The domain estimates below give $T_j(D(\tilde s_{j+1}))\subset D(\tilde s_j)$ and $T_j^{-1}(D(\tilde s_{j+1}))\subset D(\tilde s_j)$. By \eqref{61}, the inverse on a smaller strip has the form
	$$x=T_{j}^{-1}\phi=\phi+\tilde{b}_{j}(\phi)\omega.$$
	Actually, the proof is as follows. Because $b_j$ depends only on $x_{J_j}$, we let $x_{J_{j}}=(T_{j}^{-1}\phi)_{J_{j}}=\phi_{J_{j}}+h(\phi_{J_{j}})$. Since $\phi=x+b_{j}\omega$, we have $\phi_{J_{j}}=\phi_{J_{j}}+h(\phi_{J_{j}})+b_{j}(\phi_{J_{j}}+h(\phi_{J_{j}}))\omega_{J_{j}}$, by which we get a Picard sequence
	$$\left\{ \begin{array}{ll}
	h_{0}(\phi_{J_{j}})=-b_{j}(\phi_{J_{j}})\omega_{J_{j}},\\
	h_{v}(\phi_{J_{j}})=-b_{j}(\phi_{J_{j}}+h_{v-1}(\phi_{J_{j}}))\omega_{J_{j}},v=1,2,\cdots.
	\end{array}\right. $$

	From \eqref{405}, \eqref{61} and the contraction mapping principle, there exists a $h(\phi_{J_{j}})=\lim_{v\rightarrow\infty}h_{v}(\phi_{J_{j}})$. Let $\tilde{b}_{j}(\phi)=-b_{j}(\phi_{J_{j}}+h(\phi_{J_{j}}))$ and we reach the conclusion.

For completeness, we verify the domains of both maps. Write $W_l:=\max_{k\in J_l}|\omega_k|$. By \eqref{61} and Cauchy's estimate, choosing $\gamma_0$ sufficiently small gives
\[
 \|Db_j\|_{D(9r_j/10),\,\ell^\infty\to\mathbb C}
 \leq\frac{10}{r_j}\|b_j\|_{D(r_j)}\leq\kappa_j,
 \qquad W_j\kappa_j\leq\frac12.
\]
Here $\log r_j^{-1}=O_{s_0}((j+1)^2\log(j+2))$ is absorbed by the decay of $\gamma_j$. The common strip satisfies $s_m^1\leq r_j/16$ for $0\leq j\leq m$. For $|\operatorname{Im}\phi|\leq s_m^1$, the Picard iteration preserves $|\operatorname{Im}x_{J_j}|\leq r_j/2$ and has contraction constant at most $W_j\kappa_j$. Thus the inverse exists on this common strip. Since $b_j$ is real on the real torus, for $x=T_j^{-1}\phi$ its low coordinates satisfy
\[
 |\operatorname{Im}x_{J_j}|
 \leq |\operatorname{Im}\phi_{J_j}|+W_j\kappa_j|\operatorname{Im}x_{J_j}|
 \leq |\operatorname{Im}\phi_{J_j}|+\tfrac12|\operatorname{Im}x_{J_j}|.
\]
Using $W_m\leq CM^2$, we therefore obtain
\[
 |\operatorname{Im}T_jx|\leq(1+W_m\kappa_j)|\operatorname{Im}x|,
 \qquad
 |\operatorname{Im}T_j^{-1}\phi|\leq(1+2W_m\kappa_j)|\operatorname{Im}\phi|.
\]
We may arrange $1+2CM^2\kappa_j\leq M^4$, uniformly for $0\leq j\leq m$. Consequently,
\[
 T_j^{\pm1}(D(\rho))\subset D(M^4\rho),
 \qquad 0<\rho\leq s_m^1/M^4.
\]
Taking $\rho=\tilde s_{j+1}$ proves the required stepwise inclusions.
	
	Define for $\phi\in D(\tilde{s}_{j+1})$ that $v(\phi)=u(T_{j}^{-1}\phi)=:u(x)$. By direct computation, we have
	\begin{equation}\label{63}
	\begin{aligned}
	\omega\cdot\partial_{x}u(x)=&\omega\cdot\partial_{x}v(x+b_{j}(x)\omega)\\
	=&\sum_{l\in J}\omega_{l}\partial_{x_{l}}v(x+b_{j}(x)\omega)\\
	=&\sum_{l\in J}\omega_{l}\partial_{\phi_{l}}v(\phi)+\sum_{l\in J}\sum_{i\in J}\omega_{l}\omega_{i}\partial_{\phi_{i}}v(\phi)\partial_{x_{l}}b_{j}(x)\\
	=&(\omega\cdot\partial_{\phi}v(\phi))(1+\omega\cdot\partial_{x}b_{j}(x))
	\end{aligned}
	\end{equation}
	and
	\begin{equation}\label{64}
	\lambda(1+\sum_{l=j}^{m}a_{l}(x))u(x)=\lambda(1+\sum_{l=j}^{m}a_{l}(T_{j}^{-1}\phi))v(\phi).
	\end{equation}
	
	From \eqref{63}\eqref{64}, the homological equation \eqref{406} can be rewritten as
	\begin{equation}\label{65}
	(\mathrm{i}\omega\cdot\partial_{\phi}v(\phi))(1+\omega\cdot\partial_{x}(b_{j}(T^{-1}_{j}\phi)))+\lambda(1+\sum_{l=j}^{m}a_{l}(T_{j}^{-1}\phi))v(\phi)=p(T_{j}^{-1}\phi).
	\end{equation}
	
	Let $\tilde{p}(x)=(1+a_{j}(x))^{-1}p(x),\tilde{a}_{l}(x)=(1+a_{j}(x))^{-1}a_{l}(x)$. Then \eqref{65} reads
	\begin{equation}\label{67}
	\mathrm{i}\omega\cdot\partial_{\phi}v(\phi)+\lambda(1+\sum_{l=j+1}^{m}\tilde{a}_{l}(T_{j}^{-1}\phi)) v(\phi)=\tilde{p}(T_{j}^{-1}\phi)=:p^{*}(\phi).
	\end{equation}
	From \eqref{407}, it follows that
	$$\Vert p^{*}\Vert_{D(\tilde{s}_{j+1})}\leq\Vert\tilde{p}\Vert_{D(\tilde{s}_{j})}\leq(1-C\gamma_{j})^{-1}\Vert p\Vert_{D(\tilde{s}_{j})}.$$
	Similarly, we obtain analogous estimates for $a_{l}^{*}(\phi):=\tilde{a}_{l}(T_{j}^{-1}\phi)$. Thus we complete the passage from the $j$-th step to the $(j+1)$-th step. Set $T=T_m\circ\cdots\circ T_0$. The common-strip estimate above controls every intermediate image in either composition order, and the $m+1$ factors give
\[
 T(D(s_m^2))\subset D(s_m^1),\qquad
 T^{-1}(D(s_m^3))\subset D(s_m^2),\qquad
 T(D(s_m^5))\subset D(s_m^4).
\]
Consequently, the lemma follows from the rapid decay of $\gamma_j$.
\end{proof}

We now conclude with an estimate of the solution to the original small denominator equation with large variable coefficient.

\begin{lem}\label{118}
	Consider the equation \eqref{411}. Assume\\
	(i)The assumptions in Lemma \ref{408} are satisfied.\\
	(ii)There is a constant $\beta$ such that the following Diophantine conditions hold:
	\begin{equation}\label{68}
		\vert\langle k,\omega\rangle+\lambda\vert\geq\frac{\beta}{\langle k\rangle^{\tau_{m}}},k\in\mathbb{Z}^{J_{m}}.
	\end{equation}
	(iii)$p(x)$ is analytic in $x\in D(s_{m})$.\\
	Then \eqref{411} has a unique solution $u=u(x):D(s_{m}^
	{5})\rightarrow\mathbb{C}$ which satisfies 
	$$\Vert u\Vert_{D(s_{m}^{5})}\leq(\tau_{m}^{\tau_{m}})^{exp}\frac{1}{\beta s_{m}^{C\tau_{m}}}\Vert p\Vert_{D(s_{m})},$$
	
\end{lem}
\begin{proof}
	Let $T$ be the coordinate transformation from Lemma \ref{408}. Then it follows that $v(\phi)=u(T^{-1}\phi)$ satisfies
	\begin{equation}\label{409}
		(\mathrm{i}\partial_{\omega}+\lambda)v=p^{*}(\phi)
	\end{equation}
	on $D(s_{m}^{3})$. By Fourier expansion, the solution to equation \eqref{409} is
	$$v(\phi)=\sum_{k\in\mathbb{Z}^{n}}\frac{1}{\lambda-\langle k,\omega\rangle}\hat{p}^{*}(k)e^{\mathrm{i}k\phi}.$$
	By the Diophantine conditions \eqref{68}, using the strictly positive width $s_m^3-s_m^4$, we obtain that
	$$\Vert v\Vert_{D(s_{m}^{4})}\leq(\tau_{m}^{\tau_{m}})^{exp}\frac{1}{\beta s_{m}^{C\tau_{m}}}\Vert p^{*}\Vert_{D(s_{m}^{3})}.$$
	Returning to the original equation \eqref{411}, we conclude that
	$$\Vert u\Vert_{D(s_{m}^{5})}\leq (\tau_{m}^{\tau_{m}})^{exp}\frac{1}{\beta s_{m}^{C\tau_{m}}}\Vert p\Vert_{D(s_{m})}.$$
\end{proof}

\appendix

\section{infinite dimensional product spaces}\label{117}
In this section, we briefly introduce the existence, uniqueness and construction of the product measure of a countably infinite number of probability measures. 

Let $\lbrace(\Omega_{k},\mathcal{F}_{k},\mathbb{P}_{k})\rbrace_{k\in\mathbb{Z}}$ be a countable number of probability spaces. We define the product space 
$$(\Omega,\mathcal{F},\mathbb{P}):=\bigotimes_{k\in\mathbb{Z}}(\Omega_{k},\mathcal{F}_{k},\mathbb{P}_{k})$$
of these infinitely many probability spaces in the manner as described in the following (a)-(c).\\
(a)Define the infinite product phase space as follows.
$$\Omega=\Pi_{k\in\mathbb{Z}}\Omega_{k}:=\lbrace\omega:=(\cdots,\omega_{-1},\omega_{0},\omega_{1},\cdots):\omega_{k}\in\Omega_{k},k\in\mathbb{Z}\rbrace.$$
(b)The definition of an infinite product measurable structure is as follows. For any non-empty finite subset $I\subset\mathbb{Z}$, we can define the product probability space 
$$(\Omega^{(I)},\mathcal{F}^{(I)},\mathbb{P}^{(I)}):=\bigotimes_{k\in I}(\Omega_{k},\mathcal{F}_{k},\mathbb{P}_{k})$$
in the usual way. Furthermore, we can also define natural projection $\pi_{I}$:
$$\pi_{I}:\Omega\rightarrow\Omega^{(I)},\omega:=(\omega_{k},k\in\mathbb{Z})\mapsto\pi_{I}(\omega)=\omega^{(I)}:=(\omega_{k},k\in I).$$
For any finite non-empty index set $I\subset\mathbb{Z}$ and any $A\in\mathcal{F}^{(I)}$, we call $\pi^{-1}_{I}(A)\subset\Omega$ a cylinder set with $A$ as the base and $I$ as the index set. Let
$$\mathcal{E}_{I}:=\lbrace\pi^{-1}_{I}(A):A\in\mathcal{F}^{(I)}\rbrace=\pi^{-1}_{I}(\mathcal{F}^{(I)}),$$
and furthermore
$$\mathcal{E}:=\cup_{I\subset\mathbb{Z},\sharp I<\infty}\mathcal{E}_{I}$$
which is the set consisting of all the cylinder sets in $\Omega$ and an algebra on $\Omega$. At this point, let 
$$\mathcal{F}=\bigotimes_{k\in\mathbb{Z}}\mathcal{F}_{k}:=\sigma(\mathcal{E})$$
which is the $\sigma$-algebra generated by $\mathcal{E}$ and gives the product measurable structure of an enumerable infinite number of measurable structures.\\
(c)Finally, the product measure of a countable number of probability measures is defined as follows. For any $B\in\mathcal{E}$, there exist a finite non-empty index set $I\subset\mathbb{Z}$ and a measurable set $A\in\mathcal{F}^{(I)}$ such that $B=\pi_{I}^{-1}(A)$. We define
$$\mathbb{P}(B):=\mathbb{P}^{(I)}(A).$$
The unique probability measure on $\mathcal{F}=\sigma(\mathcal{E})$ can be defined by the extension theorem of Carathéodory. The details of the proof can be found in standard textbooks on measure theory and we omit it.
\section{lemma from liu-yuan}
The following lemma is a generalization of the Liu–Yuan lemma \cite{LY10}, adapted to the almost periodic setting. The proof is completely similar to that of Theorem 1.4 in \cite{LY10} and we omit it. For simplicity, we continue to use the parameters introduced in section \ref{119}.
\begin{lem}\label{416}
	Consider the first order partial differential equation
	\begin{equation}\label{413}
		-\mathrm{i}\partial_{\omega}u+\lambda u+\mu(x)u=p(x),x\in\mathbb{T}^{n},
	\end{equation}
	for the unknown function  $u$ defined on the torus $\mathbb{T}^{n}$, where $\omega=(\omega_{1},\cdots,\omega_{n})\in\mathbb{R}^{n}$ and $\lambda\in\mathbb{C}$. Assume\\
	(1)There are constants $\alpha_{l}(l=0,1,\cdots,m),\tilde{\gamma}>0$ and $\tau_{l}>\sharp J_{l}$ such that
	$$
	\begin{aligned}
	\vert\langle k,\omega\vert_{J_{l}}\rangle\vert&\geq\frac{\alpha_{l}}{\vert k\vert^{\tau_{l}}},k\in\mathbb{Z}^{J_{l}}\setminus\lbrace 0\rbrace,l=0,1,\cdots,m\\
	\vert k\cdot\omega+\lambda\vert&\geq\frac{\beta\tilde{\gamma}}{1+\vert k\vert^{\tau_{m}}},k\in\mathbb{Z}^{J_{m}}.
	\end{aligned}
	$$
	(2)$\mu:D(s_{m})\rightarrow\mathbb{C}$ has the decomposition
	$$\mu(x)=\sum_{l=0}^{m}\mu_{l}(x),$$
	where $\mu_{l}:D(s_{l})\rightarrow\mathbb{C}$ is real analytic and is of zero average: $[\mu_{l}]=0$. Moreover, assume $\mu_{l}(x)$ depends only on $x\in\mathbb{T}^{J_{l}}_{s_{l}}$ and
	$$\vert\mu_{l}\vert_{s_{l},\tau_{l}+1}\leq \alpha_{l}\gamma_{l}\tilde{\gamma}.$$
	(3)$p(x)$ is analytic in $x\in D(s_{m})$.\\
	Then \eqref{413} has a unique solution $u(x)$ which is defined in a narrower domain $D(s_{m}-\sigma)$ with $0<\sigma<s_{m}$, and satisfies
	$$\Vert u\Vert_{D(s_{m}-\sigma)}\leq\frac{(\tau_{m}^{\tau_{m}})^{exp}}{\beta\tilde{\gamma}\sigma^{C\tau_{m}}}e^{C\tilde{\gamma}s\sum_{l=0}^{m}\gamma_{l}}\Vert p\Vert_{D(s_{m})}.$$

\end{lem}

\section{a technical lemma}
\begin{lem}\label{56}
	Let $R=(R_{ij})_{i,j\in\mathbb{Z}}$ be a bounded operator on $l^{2}$ which depends on $x\in\mathbb{T}^{n}$ such that all elements $(R_{ij})$ are analytic on $D(s)$. Suppose $F=(F_{ij})_{i,j\in\mathbb{Z}}$ is another operator on $l^{2}$ depending on $x$ whose elements satisfy
	$$\sup_{x\in D(s')}\vert F_{ij}(x)\vert\leq\frac{1}{1+\vert \vert i\vert-\vert j\vert\vert}\sup_{x\in D(s-\sigma)}\vert R_{ij}(x)\vert.$$
	Then $F$ is a bounded operator on $l^{2}$ for every $x\in D(s')$, and
	$$\sup_{x\in D(s')}\Vert F(x)\Vert\leq\frac{4^{n+2}}{\sigma^{n}}\sup_{x\in D(s)}\Vert R(x)\Vert.$$
\end{lem}
\begin{proof}
	The proof is completely similar to that of lemma M.3 in \cite{KP2} and we omit it.
\end{proof}

\section*{Acknowledgments}
We express our deepest gratitude to Joackim Bernier for invaluable discussions that greatly clarified our understanding of the notion of full-dimensional invariant tori, and for generously sharing with us his joint results with Michela Procesi on this subject as soon as they became available. The work was supported by National Natural Science Foundation of China (Grant No. 12371189).

\section*{Ethics Declarations}

\noindent\textbf{Data Availability:} No datasets were generated or analyzed during the current study.

\vspace{10pt}

\noindent\textbf{Conflict of Interest:} The authors declare no conflict of interest  to disclose.

\newcommand{\etalchar}[1]{$^{#1}$}


\begin{thebibliography}{BBHM18}
	\bibitem[Agr19]{Agr}
	G. P. Agrawal.
	\newblock {\em Nonlinear Fiber Optics}.
	\newblock Elsevier, London, 2019.
	
	
	\bibitem[BB13]{BB13}
	M. Berti and P. Bolle.
	\newblock Quasi-periodic solutions with {S}obolev regularity of {NLS} on
	{$\mathbb{T}^d$} with a multiplicative potential.
	\newblock {\em J. Eur. Math. Soc. (JEMS)}, 15(1):229--286, 2013.
	
	\bibitem[BB20]{BB20}
	M. Berti and P. Bolle.
	\newblock {Quasi-periodic solutions of nonlinear wave equations on the
		{$d$}-dimensional torus}.
	\newblock EMS Monographs in Mathematics. EMS Publishing House, Berlin, [2020]
	\copyright 2020.
	
	\bibitem[BBHM18]{BBHM18}
	P. Baldi, M. Berti, E. Haus, and R. Montalto.
	\newblock Time quasi-periodic gravity water waves in finite depth.
	\newblock {\em Invent. Math.}, 214(2):739--911, 2018.
	
	\bibitem[BBM16]{BBM}
	P. Baldi, M. Berti, and R. Montalto.
	\newblock KAM for autonomous quasi-linear perturbations of KdV.
	\newblock {\em Ann. Inst. H. Poincar\'{e} C Anal. Non Lin\'{e}aire}, 33:1589-1638, 2016.
	
	\bibitem[BBP13]{BBP13}
	M. Berti, L. Biasco, and M. Procesi.
	\newblock K{AM} theory for the {H}amiltonian derivative wave equation.
	\newblock {\em Ann. Sci. \'{E}c. Norm. Sup\'{e}r. (4)}, 46(2):301--373,
	2013.
	
	\bibitem[BCN80]{BCN80}
	H. Brézis, J. M. Coron and L. Nirenberg
	\newblock Free vibrations for a nonlinear wave equation and a theorem of P. Rabinowitz.
	\newblock {\em Commun. Pure. Appl. Math.}, 33(5):667-684, 1980.
	
	\bibitem[BGR24]{BGR}
	J. Bernier, B. Gr\'{e}bert, and T. Robert.
	\newblock Infinite dimensional invariant tori for nonlinear Schr\"{o}dinger equations.
	\newblock {\em preprint}, https://arxiv.org/abs/2412.11845,
	2024.
	
	\bibitem[BHM23]{BHM}
	M. Berti, Z. Hassainia, and N. Masmoudi.
	\newblock Time quasi-periodic vortex patches of Euler equation in the plane.
	\newblock {\em Invent. Math.}, 233:1279-1391,
	2023.
	
	\bibitem[BMP23]{BMP}
	L. Biasco, J.E. Massetti, and M. Procesi.
	\newblock Weak Sobolev almost periodic solutions for the 1D NLS.
	\newblock {\em Duke Math. J.}, 172(14):2643-2714, 2023.
	
	
	\bibitem[BP26]{BP26}
J. Bernier and M. Procesi.
Lagrangian invariant tori for completely resonant nonlinear Schr\"odinger equations.
Preprint, arXiv:2610.26044v1, 2026.

    \bibitem[Bou96]{Bou96-GAFA}
	J.~Bourgain.
	\newblock Construction of approximative and almost periodic solutions of
	perturbed linear {S}chr\"{o}dinger and wave equations.
	\newblock {\em Geom. Funct. Anal.}, 6(2):201--230, 1996.
	
	\bibitem[Bou98]{B1}
	J.~Bourgain.
	\newblock Quasi-periodic solutions of {H}amiltonian perturbations of 2{D}
	linear {S}chr\"{o}dinger equations.
	\newblock {\em Ann. of Math.}, 148(2):363--439, 1998.
	
	\bibitem[Bou05a]{B2}
	J.~Bourgain.
	\newblock {\em Green's function estimates for lattice {S}chr\"{o}dinger
		operators and applications}, volume 158 of {\em Annals of Mathematics
		Studies}.
	\newblock Princeton University Press, Princeton, NJ, 2005.
	
	\bibitem[Bou05b]{B3}
	J.~Bourgain.
	\newblock On invariant tori of full dimension for 1D periodic NLS.
	\newblock {\em J. Funct. Anal.}, 229:62-94, 2005.
	
	\bibitem[CLSY18]{CLSY18}
	H. Cong, J. Liu, Y. Shi and X. Yuan.
	\newblock The stability of full dimensional KAM tori for nonlinear Schrödinger equation.
	\newblock {\em J. Differ. Equations}, 264(7):4504-4563, 2018.
	
	\bibitem[CMP21]{CMP}
	L. Corsi, R. Montalto, and M. Procesi.
	\newblock Almost-periodic response solutions for a forced quasi-linear Airy equation.
	\newblock {\em J. Dyn. Differ. Equations}, 33:1231-1267, 2021.
	
    
	
	\bibitem[Con23]{con23}
	H. Cong.
	\newblock The existence of full dimension KAM tori for nonlinear {S}chr\"{o}dinger equation.
	\newblock {\em Math. Ann.}, 390:671-719, 2024.
	
	\bibitem[CY21]{CY21}
	H. Cong and X. Yuan.
	\newblock The existence of full dimensional invariant tori for 1-dimensional nonlinear wave equation.
	\newblock {\em Ann. Inst. H. Poincar\'{e} C Anal. Non Lin\'{e}aire}, 38(3):759-786, 2021.
	
	

	\bibitem[EGK16]{EGK}
	L.~H. Eliasson, B. Gr\'{e}bert, and S.~B. Kuksin.
	\newblock K{AM} for the nonlinear beam equation.
	\newblock {\em Geom. Funct. Anal.}, 26(6):1588--1715, 2016.
	
	\bibitem[EK10]{EK}
	L.~H. Eliasson and S.~B. Kuksin.
	\newblock K{AM} for the nonlinear {S}chr\"{o}dinger equation.
	\newblock {\em Ann. of Math.}, 172(1):371--435, 2010.
	
	\bibitem[FMM24]{FMM24}
	L. Franzoi, N. Masmoudi and R. Montalto.
	\newblock Space Quasi-Periodic Steady Euler Flows Close to the Inviscid Couette Flow.
	\newblock {\em Arch. Rat. Mech. Anal.}, 248(5):81, 2024.
	
	\bibitem[GX13]{GX}
	J. Geng and X. Xu.
	\newblock Almost-periodic solutions of one dimensional {S}chr\"{o}dinger equation with the external parameters.
	\newblock {\em J. Dyn. Differ. Equations}, 25:435-450, 2013.
	
	\bibitem[GH17]{GH}
	J. Geng and W. Hong.
	\newblock Invariant tori of full dimension for second KdV equations with the external parameters.
	\newblock {\em J. Dyn. Differ. Equations}, 29:1325-1354, 2017.
	
	\bibitem[Hor76]{Hor76}
	L. H\"{o}rmander.
	\newblock The boundary problems of physical geodesy.
	\newblock {\em Arch. Rat. Mech. Anal.}, 62(1):1-52, 1976.
	
	\bibitem[HXY24]{HXY}
	S. Hu, H. Xue, and X. Yuan.
	\newblock Long time stability for KAM tori of the derivative nonlinear {S}chr\"{o}dinger equation.
	\newblock {\em preprint}, https://arxiv.org/abs/2401.11639, 2024.
	
	\bibitem[Kla80]{Kla80}
	S. Klainerman.
	\newblock Global existence for nonlinear wave equations.
	\newblock {\em Commun. Pure. Appl. Math.}, 33(1):43-101, 1980.
	
	
	\bibitem[KP96]{KP}
	S. B. Kuksin and J. P\"{o}schel.
	\newblock Invariant {C}antor manifolds of quasi-periodic oscillations for a
	nonlinear {S}chr\"{o}dinger equation.
	\newblock {\em Ann. of Math.}, 143(1):149--179, 1996.
	
	\bibitem[KP03]{KP2}
	T. Kappeler and J. P\"{o}schel.
	\newblock {\em KdV\&KAM.}
	\newblock Springer-Verlag, Berlin, Heidelberg, 2003.
	
	\bibitem[Kuk93]{K}
	S.~B. Kuksin.
	\newblock {\em Nearly integrable infinite-dimensional {H}amiltonian systems},
	volume 1556 of {\em Lecture Notes in Mathematics}.
	\newblock Springer-Verlag, Berlin, 1993.
	
	\bibitem[Kuk97]{K2}
	S.~B. Kuksin.
	\newblock On small-denominators equations with large variable coefficients.
	\newblock {\em J. Appl. Math. Phys.}, 48:262-271, 1997.
	
	\bibitem[Kuk00]{Kuk00}
	S.~B. Kuksin.
	\newblock {\em Analysis of {H}amiltonian {PDE}s}, volume~19 of {\em Oxford
		Lecture Series in Mathematics and its Applications}.
	\newblock Oxford University Press, Oxford, 2000.
	
	\bibitem[Kuk04]{Kuk04}
	S.~B. Kuksin.
	\newblock {\em Fifteen years of KAM for PDE}, volume~212 of {\em American Mathematical Society Translation}.
	\newblock American Mathematical Society, Providence, 2004.
	
	\bibitem[LY10]{LY10}
	J. Liu and X. Yuan.
	\newblock Spectrum for quantum Duffing oscillator and small-divisor equation with large variable coefficient.
	\newblock {\em Commun. Pure. Appl. Math.}, 63(9):1145-1172, 2010.
	
	\bibitem[LY11]{LY11}
	J. Liu and X. Yuan.
	\newblock A {KAM} theorem for {H}amiltonian partial differential equations with
	unbounded perturbations.
	\newblock {\em Comm. Math. Phys.}, 307(3):629--673, 2011.
	
	\bibitem[LY14]{LY14}
	J. Liu and X. Yuan.
	\newblock KAM for the derivative nonlinear Schr\"{o}dinger equation with periodic boundary conditions.
	\newblock {\em J. Differential Equations}, 256:1627-1652, 2014.
	
	\bibitem[MOT76]{MOT76}
	K. Mio, T. Ogino, and S. Takeda.
	\newblock Modified Nonlinear Schrödinger Equation for Alfvén Waves Propagating along the Magnetic Field in Cold Plasmas.
	\newblock {\em J. Phys. Soc. Jpn.}, 41(1):265-271, 1976.
	
	\bibitem[MV11]{MV11}
	C. Mouhot and C. Villani.
	\newblock On Landau damping.
	\newblock {\em Acta Math.}, 207(1):29-201, 2011.
	
	\bibitem[Pos90]{Pos90}
	J. P\"{o}schel.
	\newblock Small divisors with spatial structure in infinite-dimensional Hamiltonian systems.
	\newblock {\em Comm. Math. Phys.}, 127(2):351-393, 1990.
	
	
	\bibitem[Pos96a]{P1}
	J. P\"{o}schel.
	\newblock A {KAM}-theorem for some nonlinear partial differential equations.
	\newblock {\em Ann. Scuola Norm. Sup. Pisa Cl. Sci. (4)}, 23(1):119--148, 1996.
	
	\bibitem[Pos96b]{P2}
	J. P\"{o}schel.
	\newblock Quasi-periodic solutions for a nonlinear wave equation.
	\newblock {\em Comment. Math. Helv.}, 71(2):269--296, 1996.
	
	\bibitem[Pos02]{P3}
	J. P\"{o}schel.
	\newblock On the construction of almost periodic solutions for a nonlinear {S}chr\"{o}dinger equation.
	\newblock {\em Ergodic Theory Dyn. Syst.}, 22(5):1537-1549, 2002.
	
	\bibitem[Rab67]{Rab67}
	P. H. Rabinowitz.
	\newblock Periodic solutions of nonlinear hyperbolic partial differential equations.
	\newblock {\em Commun. Pure. Appl. Math.}, 20(1):145–205, 1967.
	
	\bibitem[Str08]{Str08}
	M. Struwe.
	\newblock {\em Variational Methods.}
	\newblock Springer-Verlag, Berlin, Heidelberg, 2008.
	
	
	\bibitem[Wan16]{Wang16}
	W. Wang.
	\newblock Energy supercritical nonlinear {S}chr\"{o}dinger equations:
	quasiperiodic solutions.
	\newblock {\em Duke Math. J.}, 165(6):1129--1192, 2016.
	
	\bibitem[Wan19]{Wang19}
	W. Wang.
	\newblock Quasi-periodic solutions to nonlinear {PDE}s.
	\newblock In {\em Harmonic analysis and wave equations}, volume~23 of {\em Ser.
		Contemp. Appl. Math. CAM}, pages 127--175. Higher Ed. Press, Beijing, 2019.
	
	\bibitem[Wan20]{Wang20}
	W. Wang.
	\newblock Space quasi-periodic standing waves for nonlinear {S}chr\"{o}dinger
	equations.
	\newblock {\em Comm. Math. Phys.}, 378(2):783--806, 2020.
	
	\bibitem[Way90]{W}
	C.~E. Wayne.
	\newblock Periodic and quasi-periodic solutions of nonlinear wave equations via
	{KAM} theory.
	\newblock {\em Comm. Math. Phys.}, 127(3):479--528, 1990.
	
	\bibitem[XYY26]{XYY}
	H. Xue, Z. You, and X. Yuan.
	\newblock Construction of periodic solutions of multi-dimensional
	nonlinear wave equations with unbounded perturbation.
	\newblock {\em J. Differ. Equations}, 453:113881, 2026.
	
	
	
	\bibitem[Yua21]{Yua21}
	X. Yuan.
	\newblock KAM Theorem with Normal Frequencies of Finite Limit‐Points for Some Shallow Water Equations.
	\newblock {\em Commun. Pure. Appl. Math.}, 74(6):1193-1281, 2021.
	
	\bibitem[YY25]{YY25}
	Z. You and X. Yuan.
	\newblock Periodic Response Solutions to Multidimensional Nonlinear Schrödinger Equations with Unbounded Perturbation.
	\newblock {\em Regular \& Chaotic Dynamics}, https://doi.org/10.1134/S1560354725530048, 2025.
	
	\bibitem[YY26]{YY26}
	Z. You and X. Yuan.
	\newblock Quasi-periodic Dynamics for Multi-dimensional Quasi-linear Schrödinger Equations via Resonant Mode Control.
	\newblock {\em preprint.}, https://arxiv.org/abs/2601.13611, 2026.
	
	\bibitem[Zeh76]{Zeh76}
	E. Zehnder.
	\newblock Moser's implicit function theorem in the framework of analytic smoothing.
	\newblock {\em Math. Ann.}, 219(2):105-121, 1976.
\end{thebibliography}
\end{document}